\documentclass[12pt]{amsart}

\usepackage{amscd,amstext,amsfonts,amsthm,amsxtra,amsmath, xcolor}
\usepackage{amssymb}

\usepackage{mathrsfs}
\usepackage[all]{xy}

\usepackage[margin=1in]{geometry}
\theoremstyle{plain}
\newtheorem{theorem} {Theorem} [section]
\newtheorem{lemma} [theorem]{Lemma}
\newtheorem{proposition}[theorem]{Proposition}

\theoremstyle{definition}
\newtheorem{definition} [theorem] {Definition}
\newtheorem{example} [theorem] {Example}
\newtheorem{notation} [theorem]  {Notation}
\newtheorem{remark}[theorem]{Remark}

\def\Sym{\operatorname{Sym}}
\def\Gr{\operatorname{Gr}}
\def\GL{\operatorname{GL}}
\def\aut{\mathfrak{aut}}
\def\Hom{\operatorname{Hom}}
\def\Ker{\operatorname{Ker}}
\def\Symb{\operatorname{Symb}}

\def\Aut{\operatorname{Aut}}

\def\span{\operatorname{span}}
\def\PSp{\operatorname{PSp}}
\def\Sp{\operatorname{Sp}}

\title[Characterizations of horospherical varieties]{Characterizations of smooth projective horospherical varieties of Picard number one}

\date{\today}

\author[J. Hong]{Jaehyun Hong}
\address{Center for Complex Geometry,  Institute for Basic Science (IBS), 55  Expo-ro, Yuseong-gu, Daejeon, Korea 34126}
\email{jhhong00@ibs.re.kr}

\author[S.-Y. Kim]{Shin-Young Kim}
  \address{Center for Geometry and Physics, Institute for Basic Science (IBS), 77 Cheongam-ro, Nam-gu, Pohang, Gyeongbuk, Korea 37673}
\email{shinyoungkim@ibs.re.kr}

\keywords{geometric structure, local equivalence, horospherical variety, Cartan geometry}
\subjclass[2000]{32M12, 53A55, 14J45}

\begin{document}

\begin{abstract}
Let $X$ be a  smooth projective horospherical variety of Picard number one.  We show that a  uniruled projective manifold of Picard number one is biholomorphic to $X$ if its variety of minimal rational tangents at a general point is projectively equivalent to that of $X$. 
To get a local flatness of the geometric structure arising from the variety of minimal rational tangents, we apply the methods of $W$-normal complete step prolongations. We compute the associated Lie algebra cohomology space of degree two   and show the vanishing of holomorphic sections of the vector bundle having this cohomology space as a fiber.
\end{abstract}

\maketitle

\section{Introduction}

 Let $M$ be a projective uniruled algebraic variety over $\mathbb C$.
Given a covering family $\mathcal K$ of minimal rational curves,  by collecting the tangent directions of rational curves in $\mathcal K$ passing through each point in $M$, we define  a subbundle $\mathcal C(M)$ of the projectivizaition $\mathbb P(TM) $ of the tangent bundle, called the variety of minimal rational tangents associated to $\mathcal K$. For a precise definition, see Definition \ref{d.vmrt}. 

The theory of varieties of minimal rational tangents, introduced by Hwang and Mok, has played an important role in the complex geometry of a uniruled projective manifold of Picard number one, that is, a Fano manifold of Picard number one. A  general philosophy in this theory is that one should be able to recover the complex geometry of a uniruled projective manifold of Picard number
one,   such as the deformation rigidity and the stability of the tangent bundles, from the projective geometry of its varieties of minimal rational tangents. There are many results manifesting this philosophy, one of which is recognizing rational homogeneous varieties of Picard number one by their varieties of minimal rational tangents.

\begin{theorem} [Main Theorem of \cite{Mk08}, Main Theorem of \cite{HH}] \label{previous results} Let $X$ be a rational homogeneous variety associated to a long root and let $\mathcal C_o(X) \subset \mathbb PT_{o}X$ be the variety of minimal rational tangents at a base point $o \in X$. Let $M$ be a Fano manifold of Picard number one and $\mathcal C_x(M) \subset \mathbb PT_xM$ be the variety of minimal rational tangents at a general point $x \in X$ associated to a minimal dominant rational component. Suppose that $\mathcal C_x(M) \subset \mathbb PT_xM$ are projectively equivalent to $\mathcal C_o (X)\subset \mathbb PT_oX$ for general $x \in M$. Then $M$ is biholomorphic to $X$.
\end{theorem}
Here, for a rational  homogeneous variety $X$ associated to a long root, there is a canonical choice of a covering family $\mathcal  K $ of minimal rational curves and $\mathcal C(X)$ is the variety of minimal rational tangents associated with $\mathcal K $.

\medskip

Two ingredients of the proof of Theorem \ref{previous results}  are the following:  
\begin{enumerate}
\item[$\bullet$] (Local equivalence problem)  existence of a biholomorphism $\phi$ from a connected open subset $U$ of $M$ onto an open subset of $X$ whose differential $d\phi$ maps $\mathcal C_x(M)$ onto $\mathcal C_{\phi(x)}(X)$ for any $x \in U$;
\item[$\bullet$] (Extension problem)   extension of a biholomorphism   between connected   open subsets of two Fano manifolds of Picard number one which preserves varieties of minimal rational tangents.
\end{enumerate}

\noindent
The second problem has an answer we can apply to any Fano manifold  of Picard number one under mild geometric conditions (Theorem \ref{Cartan-Fubini}), while the first problem should be treated, case by case, depending on $X$. We will focus on this local equivalence problem in this paper.

Given a quasi-homogeneous variety $X$ of Picard number one which is uniruled and smooth,  we may ask the same question, whether we can characterize $X$, or, the open dense orbit $X^0$ of the automorphism group of $X$, in terms of the variety of minimal rational tangents. To make this a reasonable question,
it is natural to assume that there is a minimal rational curve in $X$ contained in $X^0$ completely. This happens, for example, when the boundary $X \backslash X^0$ has codimension $\geq 2$ in $X$. Among such quasi-homogeneous varieties, we will consider smooth projective horospherical varieties of Picard number one.

For a reductive algebraic group $L$, a normal $L$-variety $X$ is said to be \emph{horospherical} if it has an open $L$-orbit $L/H$ whose isotropy group $H$ contains the unipotent part of a Borel subgroup of $L$. Then the normalizer $P$ of $H$ in $L$ is a parabolic subgroup of $L$ and the open orbit $L/H$ is isomorphic to a $(\mathbb C^*)^r$- bundle over the rational homogeneous variety $L/P$ for some $r \in \mathbb Z_{\geq 0}$.

Pasquier classified smooth projective horospherical varieties of Picard number one and obtained the following result.

\begin{theorem}[Theorem 0.1  of \cite{Pa}] \label{thm:Horospheical pasquier}
Let $L$ be a reductive group.
 Let $X$ be a smooth nonhomogeneous projective horospherical $L$-variety with Picard number one.
 Then $X$ is uniquely determined by its two closed $L$-orbits $ Y$ and $Z$, isomorphic to $L/P_Y$ and $L/P_Z$, respectively; and $(L, P_Y, P_Z)$ in one of the triples listed below.
 \begin{enumerate}
   \item[\rm(1)] $(B_m,{\alpha_{m-1}},{\alpha_{m}})$ for $m\geq3$;
   \item[\rm(2)] $(B_3,{\alpha_{1}},{\alpha_{3}})$;
   \item[\rm(3)] $(C_m,{\alpha_{i+1}},{\alpha_{i}})$ for $m\geq2$, $i\in \{1,\ldots, m-1\}$;
   \item[\rm(4)] $(F_4,{\alpha_{2}},{\alpha_3})$;
   \item[\rm(5)] $(G_2,{\alpha_{2}},{\alpha_1})$.

 \end{enumerate}


Here, we take the convention that $\alpha_3, \alpha_4$ are short roots  when $L$ is $F_4$ and $\alpha_1$ is a short root when $L$ is $G_2$.
\end{theorem}

The main result in this paper is to characterize smooth horospherical varieties 
by using the variety of minimal rational tangents as Theorem \ref{previous results}.
As before, there is a canonical choice of a covering family of minimal rational curves on them (Proposition \ref{vmrt of X}).

\begin{theorem} \label{main results}

 Let $X$ be a smooth projective horospherical variety of Picard number one. Let $\mathcal C_o(X)$ denote the variety of minimal rational tangents of $X$ at a base point $o$ of $X$.
 Let $M$ be a Fano manifold of Picard number one and $\mathcal C_x(M) \subset \mathbb PT_xM$ be the variety of minimal rational tangents at a general point $x \in X$ associated to a minimal dominant rational component. Suppose that $\mathcal C_x(M) \subset \mathbb PT_xM$ are projectively equivalent to $\mathcal C_o (X)\subset \mathbb PT_oX$ for general $x \in M$. Then $M$ is biholomorphic to $X$.
\end{theorem}

For the smooth horospherical varieties  $(C_m, \alpha_{i+1}, \alpha_i)$ and  $(G_2, \alpha_2, \alpha_1)$,  Hwang and   Li  solve the characterization problem,  proving   that these horospherical varieties are characterized by their varieties of minimal rational tangents (\cite{HL19}, \cite{HL21}). In this paper we prove the same characterization  for other smooth horospherical varieties in the list of Theorem \ref{thm:Horospheical pasquier}.

As we said at the beginning, the main issue is how to obtain the local equivalence of geometric structures modeled on $\mathcal C(X) \subset \mathbb P(TX)$.

One effective way  to solve the local equivalence problem of geometric structures is by constructing  Cartan connections. In the  cases dealt with  Theorem \ref{previous results}, the existence of a  Cartan connection solving local equivalence problem of geometric structures modeled on the subbundle $\mathcal C(X) \subset \mathbb P(TX)$  was proved by  Tanaka (Theorem \ref{theorem tanaka}). Later on, Morimoto extended the theory of Cartan connections to geometric structures satisfying the condition (C) (Theorem \ref{theorem morimoto}).   Hwang and Li developed a way to solve the local equivalence problem from the vanishing of certain types of sections (Theorem \ref{theorem HwangLi}).
In this paper we use  the prolongation methods
 (Theorem \ref{thm:prolongation methods} and Proposition \ref{prop:prolongation methods}) together with the computation of the Lie algebra cohomology. 

The computation of Lie algebra cohomology $H^2(\mathfrak m, \Gamma)$ is of independent interest, where $\mathfrak m$ is the negative part of a graded Lie algebra $\mathfrak g$ and $\Gamma$ is a representation of $\mathfrak g$. When $\mathfrak g$ is not semisimple, we cannot apply the theory of Kostant directly to compute it (\cite{Ko61}). We suggest a way to compute the Lie algebra cohomology $H^2(\mathfrak m, \Gamma)$ by reducing it to the computation of  the Lie algebra cohomology associated with the restriction of the action to the semisimple part of $\mathfrak g$. It is worth comparing our procedure with the Lyndon-Hochschild-Serre spectral sequence, which is a tool to compute the Lie algebra cohomology when $\mathfrak g$ is not semisimple (\cite{HS53}, \cite{HS73}).

 The horospherical variety of type $(B_3, \alpha_1, \alpha_3)$ is a smooth hyperplane section of the Spinor variety $\mathbb S_{5}$ of dimension 10 and its variety of minimal rational tangents is a hyperplane section of the Grassmannian $\Gr(2,5)$. Any two smooth hyperplane sections of $\mathbb S_5$ are projectively equivalent  (\cite{FH18}). Theorem \ref{main results} implies that the smooth hyperplane section of $\mathbb S_5$ can be characterized by the variety of minimal rational tangents.
 However, a smooth codimension two linear section   of $\mathbb S_5$ does not have this property. Indeed, there are two non-isomorphic smooth codimension two linear sections of $\mathbb S_{5}$, both of which are quasi-homogeneous (Proposition 4.8 of \cite{BFM}) and have the same variety of minimal rational tangents  being codim 2 linear section of $\Gr(2,5)$.
\\

This paper is organized as follows.
We review the theory of varieties of minimal rational tangents and collect results on the second fundamental forms and the third fundamental forms of varieties of minimal rational tangents in Section \ref{sect:vmrt}. Section \ref{sect:geometric structures} devotes to the theory of geometric structures, correspondence between $G_0$-structures and ${\bf S}$-structures, and we review the prolongation methods which gives  the local  equivalence under the vanishing of sections of vector bundles associated to Lie algebra cohomology of degree 2. We describe the varieties of minimal rational tangents of smooth horospherical varieties of Picard number one in Section  \ref{horospherical varieties}.

In the remaining sections, we prove Theorem \ref{main results} in the case when $X$ is  $(B_m, \alpha_{m-1}, \alpha_m)$ or $(F_4, \alpha_2, \alpha_3)$  or  $(B_3, \alpha_{1}, \alpha_3)$. For notational purposes, we break the case into two parts: The first two cases and the last case.
Restricting ourselves to horospherical varieties  $(B_m, \alpha_{m-1}, \alpha_m)$ and $(F_4, \alpha_2, \alpha_3)$, we confirm that the requirements to apply   theories in Section 2 and Section 3 hold (Section \ref{sect:horospherical BF type}), we  compute the Lie algebra cohomology $H^2(\mathfrak m, \mathfrak g)$ (Section \ref{sect:H2 cohomology}), and we  show  that any section of the associated vector bundle  with fiber $H^2(\frak m,\frak g)$ vanishes, which completes the proof of main Theorem  (Section \ref{sect:local equivalence}). In the final section, we repeat the same process  for $X=(B_3, \alpha_{1}, \alpha_3)$. The computations are relatively shorter since the structure of the Lie algebra   of $\Aut(X)$ and the projective geometry of the variety of minimal rational tangents are less complicated. \\

\noindent{ \bf Acknowledgements.}
We would like to appreciate  Qifeng Li and Jun-Muk Hwang for  valuable discussions, particularly for sharing the ideas in the preprint for the   $G_2$-type case. We   are also grateful to Baohua Fu for letting us know an  example of two non-isomorphic quasi-homogeneous varieties with the same variety of minimal rational tangents.
The first named author was supported by the Institute of Basic Science IBS-R032-D1. The second named author was  supported by the Institute of Basic Science IBS-R003-D1.

\section{Varieties of minimal rational tangents} \label{sect:vmrt}

\subsection{Definitions and properties}

We recall definitions and properties of varieties of minimal rational tangents. For details, see  Section 1 of \cite{HM99} and Section 1 of \cite{Mk08}.

\begin{definition} \label{d.vmrt} Let $M$ be a uniruled projective manifold and let $\mathcal K$ be a minimal rational component, i.e., an irreducible component of the space of rational curves on $M$ such that the rational curves in $\mathcal K$ sweep out a Zariski open subset of $M$ and such that, with respect to a fixed ample line bundle on $M$, the degree of the rational curves in $\mathcal K$  is minimal among all such irreducible components. Any rational curve $C$ in $\mathcal K$ can be represented by a parameterized rational curve $f: \mathbb P^1 \rightarrow M$ so that $C=f(\mathbb P^1)$.

 Denote    by $\rho:\mathcal U \rightarrow \mathcal K$ and $\mu:\mathcal U \rightarrow M$  the universal family associated to $\mathcal K$.
 We call a rational curve $C$ \emph{standard} if $f^*TM$ is decomposed as  $f^*TM=\mathcal O(2) \oplus \mathcal O(1)^p \oplus \mathcal O^q$ where $f:\mathbb P^1 \rightarrow M$ is a parameterized rational curve representing $C $.
Then there is a smallest closed subvariety $E \subset M$   such that, for every $x \in M\backslash E$, generic rational curve $C$ in $\mathcal K$ passing through $x$ is standard.
We call $E$    the \emph{bad locus} of $\mathcal K$

Define the tangent map $\tau: \mathcal U \dashrightarrow \mathbb P(TM)$ by $\tau([f])=[df(T_0\mathbb P^1)]$ for any $f:\mathbb P^1 \rightarrow M$ with $df(0)\not=0$.
For $x \in M$ the image $\mathcal C_x$ of the rational map $\tau_x: \mathcal U_x:=\mu^{-1}(x) \dashrightarrow \mathbb P(T_xM)$ is called the \emph{variety of minimal rational tangents} of $M$ at $x$.
The proper image $\mathcal C \subset \mathbb PTM$  of $\tau: \mathcal U \dashrightarrow \mathbb P TM$ is called the
(compactified) total space of varieties of minimal rational tangents.

\end{definition}

We will use the following results on varieties on minimal rational tangents $\mathcal C_x$ and the distribution defined by the linear span of its affine cone $\widehat{\mathcal C}_x$ for a general point $x \in M$. By a \emph{distribution} on a complex manifold $M$ we mean a subbundle of the tangent bundle $TM$ of $M$.

\begin{proposition}[Proposition II.3.7 of \cite{Kol}] \label{general mrc not intersect  codimension 2 subvariety} Let $M$ be a uniruled projective manifold and $\mathcal K$ be a minimal rational component on $M$. Let $Z \subset M$ be a subvariety of codimension $\geq 2$. Then   a general $C \in \mathcal K$   does not intersect $Z$.
\end{proposition}

\begin{proposition} [Proposition 13 of \cite{HM98}, Proposition 1.2.2 of \cite{HM99}] \label{non-integrability}   Let $M$ be a uniruled projective manifold of Picard number one and $\mathcal K$ be a minimal rational component on $M$.
Assume that the variety $\mathcal C_x$ of minimal rational tangents at a general point $x \in M$ is linearly degenerate, that is, the affine cone $\widehat{\mathcal C}_x$ does not span the whole tangent space $T_xM$. Let $D$ be a proper meromorphic distribution on $M$ which contains the linear span of $\widehat{\mathcal C}_x$ at a general point. Then $D$ cannot be integrable.
\end{proposition}

\begin{theorem} [Theorem 1.2 of \cite{HM01}] \label{Cartan-Fubini} Let $M_1$ and $M_2$ be  uniruled projective manifolds of Picard number one and $\mathcal K_1$ and let $\mathcal K_2$ be  minimal rational components on $M_1$ and $M_2$, respectively.
Let $E_1$ and $E_2$ be the bad loci of $\mathcal K_1$ and $\mathcal K_2$ and let $\mathcal C_1 $ and $\mathcal C_2$ be the varieties of minimal rational tangents of $(M_1, \mathcal K_1)$ and $(M_2, \mathcal K_2)$.
Assume that the general fiber $\mathcal C_{1,x}$  is positive dimensional and the Gauss map on
$\mathcal C_{1,x}$ is generically finite.

Let $U_1 \subset M_1 \backslash E_1$ and $U_2 \subset M_2 \backslash E_2$ be connected open subset and $f: U_1\rightarrow  U_2$ be a biholomorphism such that $[df](\mathcal C_1|_{U_1}) = \mathcal C_2|_{U_2}$.
Then there is a unique biholomorphic map $F: M_1 \rightarrow M_2$ such that $F|_{U_1} = f$.

\end{theorem}
For the definition of the Gauss map of a projective variety, see Definition \ref{def:fundamental forms}. For example, if $\mathcal C_{x}$ is smooth and not linear, then the Gauss map on $\mathcal C_{x} \subset \mathbb P(T_x M)$ is generically finite. See \cite{HM01} for more details.

\subsection{Fundamental forms}
Fundamental forms are basic invariants of projective varieties. We investigate behaviors of relative second and third fundamental forms of the varieties of minimal rational tangents along the liftings of standard minimal rational curves.
\begin{definition} \label{def:fundamental forms}
Let $U$ be a vector space and $Z \subset \mathbb P(U)$ be a subvariety of dimension $n$. For a point $z$ in the smooth locus $Z^0$ of $Z$, denote by $\widehat T_zZ$ the affine tangent space of $Z$ at $z$. Then the intrinsic tangent space  at $z$ is $T_zZ=\widehat z ^* \otimes  (\widehat T_zZ/\widehat z)$ and the normal space is $N_z = T_z(\mathbb P(U))/T_zZ = \widehat z^* \otimes U/\widehat T_zZ$, where $\widehat z$ is the one-dimensional subspace of $U$ corresponding to the point $z$. The {\it Gauss} map $\gamma: Z^0 \rightarrow \Gr(n+1, U)$ is defined by sending $z \in Z^0$ to the affine tangent space $\widehat T_zZ$ of $Z$ at $z$.

The differential $d _z\gamma:\widehat{T}_zZ \rightarrow (\widehat{T}_zZ)^* \otimes (U/\widehat{T}_zZ)$ is symmetric in the sense that $d_z\gamma(v)(w) = d_z\gamma(w)(v)$ for any $v,w \in \widehat{T}_zZ$ and thus defines a linear map $II_z:S^2T_zZ \rightarrow N_z$, called the {\it second fundamental form} of $Z$ at $z$. The image of $II_z$ is called the {\it second normal space} $N^{(2)}_z$ of $Z$ at $z$. The {\it second osculating affine tangent space} $T^{(2)}_zZ$ is defined by the subspace of $U$ whose quotient space by $\widehat T_zZ$ is   $\widehat z \otimes  Im \,II_z  \subset U/\widehat T_zZ$.

Let $Z^{00}$ denote the points in $Z$ where the rank of the second fundamental form does not drop.
  The second Gauss map $\gamma^{(2)}: Z^{00} \rightarrow \Gr(n^{(2)}+1, U)$ is defined by sending $z \in Z^{00}$ to the second osculating affine tangent space $\widehat T^{(2)}_zZ$ of $Z$ at $z$.  The differential $d \gamma^{(2)}$ defines a linear map $III_z:S^3T_zZ \rightarrow \widehat x^* \otimes U/T^{(2)}_zZ = N_z/N^{(2)}_z$, called the {\it third fundamental form}  of $Z$ at $z$. For any $k \geq 4$ the $k$-th fundamental form is defined in a similar way.
\end{definition}

\begin{definition}

Let $\mathcal U$ be a   vector bundle on a manifold $M$ and $\pi: \mathbb P \mathcal U \rightarrow M$ be the projection map from its projectivization. Let $\mathcal Z \subset \mathbb P(\mathcal U)$ be a subvariety  and let $\varpi: \mathcal Z \rightarrow   M$ be the restriction of $\pi$ to $\mathcal Z$.

Let $\mathcal Z^0\subset \mathcal Z$ be an open subset  such that $\varpi|_{\mathcal Z^0}: \mathcal Z^0 \rightarrow M^0:=\varpi(\mathcal Z^0)$ is a submersion and for each $t \in M^0$, $\mathcal Z_t:=   \varpi^{-1}(t) $ is immersed in $\mathbb P \mathcal U_t:=\pi^{-1}(t)$ at any point in $ \mathcal Z^0_t:= \varpi^{-1}(t) \cap \mathcal Z^0$.
Define a vector bundle $T^{\varpi}$ on $\mathcal Z^0$ by $T^{\varpi}:=\cup_{z\in \mathcal Z^0} T_{z} \mathcal Z_{\varpi(z)}$, which is called the \emph{relative tangent bundle} of $\mathcal Z \subset \mathbb P(\mathcal U)$. In a similar way we can define the relative affine tangent bundle $\widehat {T}^{\varpi}$,  relative normal bundles $\mathcal N $ and  relative second  fundamental form  $II^{\varpi}$. Then
$II^{\varpi}$ is a section of $\Hom(\Sym^2 T^{\varpi}, \mathcal N)$.

Assume that the rank of $II^{\varpi}$ is constant. Then the image of $II^{\varpi}$ defines a subbundle $\mathcal N^{(2)}$ of $\mathcal N$, called the \emph{second  normal bundle}. Similarly, we can define the relative third fundamental form $III^{\varpi}$ as a section of $\Hom(\Sym^3 T^{\varpi}, \mathcal N/\mathcal N^{(2)})$.

\end{definition}

Let $M$ be a uniruled projective manifold and $ \mathcal C$ be the variety of minimal rational tangents associated to a choice of a minimal rational component $\mathcal K$ on $M$.
For a rational curve $C$ in $\mathcal K$ represented by $f: \mathbb P^1 \rightarrow M$, we will denote by  $f^{\sharp}$   the map $\mathbb P^1 \rightarrow \mathbb P(TM)$ mapping $z \in \mathbb P^1$ to $df(T_z\mathbb P^1)$ and by $C^{\sharp}$ the image $f^{\sharp}(\mathbb P^1)$ of $f^{\sharp}$.

\begin{proposition} [Proposition 2.2 of \cite{Mk08}] \label{parallel transport of II}

Let $M$ be a uniruled projective manifold and let  $ \mathcal C$ be the variety of minimal rational tangents associated to a choice of a minimal rational component $\mathcal K$ on $M$. Denote by $\varpi: \mathcal C \rightarrow M$   the restriction of the projection map $\mathbb P(TM) \rightarrow M$.
Let $C=[f]$ be a standard rational curve on $X$ so that $f^*TM =\mathcal O(2) \oplus \mathcal O(1)^p \oplus \mathcal O^q$ for some $p$ and $q$.
Then
\begin{enumerate}
\item the pull-back $(f^{\sharp})^*\widehat{T}^{\varpi} $ of the relative affine tangent bundle $ \widehat{T}^{\varpi} $ of $\mathcal C \subset \mathbb P(TM)$ is the positive part $ P:=\mathcal O(2) \oplus \mathcal O(1)^p$ of $f^*TM$;
\item the relative 2nd fundamental form $II^{\varpi}$ of $\mathcal C|_{C^{\sharp}}$ is constant and  the pull-back $(f^{\sharp})^*\widehat{T}^{(2), \varpi} $ of the relative  second osculating affine   bundle   $ \widehat{T}^{(2), \varpi} $ of $\mathcal C \subset \mathbb P(TM)$  is a subbundle  $P^{(2)} =\mathcal O(2) \oplus \mathcal O(1)^p \oplus \mathcal O^r$ of $f^*TM$, where $r$ is the dimension of the image of $II^{\varpi}$.

\end{enumerate}

\end{proposition}

Proposition \ref{parallel transport of II} implies that the second fundamental forms of the varieties of minimal rational tangents are constant along the lifting of standard minimal rational curves. However, the third fundamental forms can vary. We will show the consistency of the third fundamental forms under some assumptions  (Proposition \ref{parallel transport of III}).

Let $D$ be the distribution on $M$ defined by the linear span of the affine cone $\widehat{\mathcal C}_x$.  Then outside a subvariety $Sing(D)$ of codimension $\geq 2$, $D$ is a subbundle of $TM$ and we may think of relative fundamental forms of $\mathcal C \subset \mathbb P(D)$.
   The kernel of the Frobenius bracket $[\,,\,]: D \wedge  D \rightarrow TM/D$
   is related to the projective geometry of the variety of minimal rational tangents as follows.
\begin{proposition} [Proof of Proposition 1.2.1 and  Proposition 1.3.1 and Proposition 1.3.2 of \cite{HM99}] \label{kernel of frobenius braket}
 Let $M$ be a uniruled projective manifold and $\mathcal K$ be a minimal rational component on $M$. Assume that at a general point $x \in M$ the variety of minimal rational tangents $\mathcal C_x$ is linearly degenerate. Denote by $D_x \subset T_xM$ the linear span of the affine cone $\widehat{\mathcal C}_x$ of $\mathcal C_x$ and by  $[ \,,\,]: \wedge^2 D_x \rightarrow T_xM/D_x$ the Frobenius Lie bracket. Then
 for a generic $\alpha \in \widehat{\mathcal C}_x$ and $\xi , \eta \in T_{\alpha}\widehat{\mathcal C}_x$,
 \begin{eqnarray*}  \alpha \wedge \xi, \,\, \alpha \wedge II(\xi, \eta),\,\,\xi \wedge \eta
  \end{eqnarray*} are contained in the kernel of $[ \,,\,]: \wedge^2 D_x \rightarrow T_xM/D_x$.
 \end{proposition}

Proposition \ref{kernel of frobenius braket} implies that the second osculation affine tangent space $ \widehat  T^{(2)} _{\alpha}\mathcal C_x$ is contained in the kernel of $[\alpha, \,]: D_x \rightarrow T_xM/D_x$, and thus the rank of $[\alpha, \,]: D_x \rightarrow T_xM/D_x$ is less than or equal to the codimension of $ \widehat  T^{(2)} _{\alpha}\mathcal C_x$ in $D_x$.

\begin{proposition} [cf. Proposition 3.1 of \cite{Mk08}, Proposition 5.1 of \cite{HH}]  \label{parallel transport of III}
Let $M$ be a uniruled projective manifold of Picard number one and let  $ \mathcal C$ be the variety of minimal rational tangents associated to a choice of a minimal rational component $\mathcal K$ on $M$.
 Let $D$ be the distribution on $M$ defined by the linear span of the affine cone $\widehat{\mathcal C}_x$.
By Proposition \ref{general mrc not intersect  codimension 2 subvariety} we may take a standard minimal rational curve $C=f(\mathbb P^1)$ with $C \cap Sing(D) =\emptyset$.
 Let $s$ be the codimension of the second osculating affine tangent space of $\mathcal C_x$ in $\mathbb P(D_x)$.
 Assume that
  \begin{enumerate}
   \item the third fundamental form of $\mathcal C_x \subset \mathbb P(D_x)$ is surjective for generic $x \in C$;

   \item the rank of $[\mathcal O(2)_x, \,]: D_x \rightarrow T_xM/D_x$ is $s$ for generic $x \in C$.
\end{enumerate}

Then the relative 3rd fundamental form  $III^{\varpi}$ is constant along $C^{\sharp}$
and we have  $f^*D =\mathcal O(2) \oplus \mathcal O(1)^p \oplus \mathcal O^r \oplus \mathcal O(-1)^s$ and $f^*TM/D =\mathcal O(1)^s \oplus \mathcal O^t$.

\end{proposition}

\begin{proof}   If $s =0$, then $D$ is integrable and thus $D=TM$  (Proposition \ref{non-integrability}).
Assume that $s >0$.
Since $\widehat {\mathcal C}_x $ is contained in $D_x$, $P^{(2)}$ is a subbundle of $f^*D$. Furthermore, $f^*D/P^{(2)}$ is a subbundle of $f^*TM/P^{(2)}\simeq \mathcal O^{q-r}$ and thus is $\mathcal O(b_1)\oplus \dots \mathcal O(b_s) ,$  where $0 \geq b_1 \geq\dots \geq b_s$. It follows that the short exact sequence $0 \rightarrow P^{(2)}\rightarrow f^*D \rightarrow f^* D/P^{(2)} \rightarrow 0$ is split and
$$f^*D = \mathcal O(2) \oplus \mathcal O(1)^p \oplus  \mathcal O^r \oplus \mathcal O(b_1)\oplus \dots \mathcal O(b_s) ,$$ where $0 \geq b_1 \geq\dots \geq b_s$.

If all $b_j$ are zero, then the Frobenius Lie bracket $[\,,\,]: D \times D \rightarrow TM/D$ is zero 
and thus $D$ is integrable, a contradiction. Therefore, $b_s<0$.

The relative 3rd fundamental form $III^{\varpi}$ of $\mathcal C|_{C^{\sharp}}$ is a section of $$\Hom(\Sym^3 T^{\varpi}, Hom(f^*L, f^*D/P^{(2)}))|_{\mathcal C^{\sharp}} \subset \Hom(\Sym^3 T^{\varpi}, N/N^{(2)})|_{\mathcal C^{\sharp}}, $$
  which is isomorphic to $\Hom(\Sym^3\mathcal O(-1)^p,  \mathcal O(b_1-2) \oplus \dots \oplus \mathcal O(b_s -2))$. By the surjectivity of the third fundamental forms $III$, we have $b_s \geq -1$. Thus $b_s=-1$ and $f^*D = \mathcal O(2) \oplus \mathcal O(1)^p \oplus \mathcal O^{r'}  \oplus \mathcal O(-1)^{s'}$ for some $r' \geq r$ and $s' \leq s$.

The Chern number of $f^*(TM/D)$ is $s'$. Since every factor of $f^*(TM/D)$ has nonnegative degree, the rank of the positive part $f^*(TM/D)_+$ of $f^*(TM/D)$ is $\leq s'$.

On the other hand, under the Frobenius bracket $[\,,\,] :D \times D \rightarrow TM/D$, the image $[\mathcal O(2), f^*D ] $ is contained in $f^*(TM/D)_+$ and has dimension $\leq s' \leq s$. By the condition that the rank of $[\mathcal O(2)_x, \,]:D_x \rightarrow T_xM/D_x$ is $s$, we have  $s \leq s'$ and thus we have $s=s'$. Consequently, $r=r'$ and $f^* D= \mathcal O(2) \oplus \mathcal O(1)^p \oplus \mathcal O^r \oplus \mathcal O(-1)^s$ and $f^*(TM/D) = \mathcal O(1)^s \oplus \mathcal O^{q-r-2s}$. Therefore, $III^{\varpi}$ is a section of a trivial vector bundle and thus is constant.
\end{proof}

\section{Geometric structures} \label{sect:geometric structures}

\subsection{$G_0$-structures and ${\bf S}$-structures}
\label{sect:G0 and S structures}

Let $\mathfrak m=\bigoplus_{p<0}\mathfrak g_{p}$    a {\it fundamental} graded Lie algebra, that is, a graded Lie algebra with $[\frak g_p, \frak g_{-1}] = \frak g_{p-1}$ for any $p<0$. 

\begin{definition}\label{distribution type}
Let $D$ be a distribution on a manifold $M$. 
Define $D^p $ for $p<0$ inductively by the following property:
$$\underline{D}^p =[\underline{D}^{p+1}, \underline{D}^{-1}]+ \underline{D}^{p+1},$$
where $\underline{D}^r$ is the sheaf of local sections of the vector bundle $D^r$. Then $\Sym_x(D):=\sum_{p<0} D^p(x)/D^{p+1}(x)$ is endowed with a structure of graded Lie algebra, called the {\it symbol algebra} of $D$.

 A distribution $D$ on a manifold $M$ is called of type $\mathfrak m$ if for each $x \in M$ the symbol algebra $\Symb_x(D)$ is isomorphic to $\mathfrak m$ as a graded Lie algebra. In this case, the pair $(M,D)$ is called a \emph{filtered manifold of type} $\mathfrak m$.

For each $x \in M$, let $\mathscr R_x $ be the set of all isomorphisms $r \colon \mathfrak m\rightarrow \Symb_x(D)$ of graded Lie algebras. Then $\mathscr R:=\cup_{x\in M}\mathscr R_x $ is a principal $G_0(\mathfrak m)$-bundle on $M$, where  $G_{0}(\mathfrak m)$ is the automorphism group  of   the graded Lie algebra $\mathfrak m$.
We call   $\mathscr R$  the \emph{frame bundle} of $(M,D)$.

\end{definition}

\begin{definition}  \label{G_0 structure}  Let  $(M,D)$ be a filtered manifold of type $\frak m$.
Given a closed subgroup $G_0\subset G_{0}(\mathfrak m)$, a \emph{$G_0$-structure} on $(M,D)$ is a $G_0$-subbundle of the frame bundle $\mathscr R$ of $(M,D)$. Two $G_0$-structures $\mathscr P_1$ on $(M_1,D_1)$ and $\mathscr P_2$ on $(M_2,D_2)$ are \emph{equivalent} if there is a biholomorphism $\varphi: M_1 \rightarrow M_2$ such that $d\varphi: TM_1 \rightarrow TM_2$ induces an isomorphism from $\mathscr P_1$ onto $\mathscr P_2$.
The local equivalence of two $G_0$-structures is defined similarly for open sets $U_1 \subset M_1$ and $U_2 \subset M_2$.
 \end{definition}

\begin{definition}\label{cone structure}  Let   $(M,D)$ be a filtered manifold of type $\frak m$.  Let ${\bf S}$ be a nondegenerate subvariety of $\mathbb P {\mathfrak g_{-1}} $.
A fiber subbundle $\mathcal S \subset \mathbb P D$ is called an  {\it ${{\bf S}}$-structure on $(M,D)$} if for each $x \in M$, the fiber $\mathcal S_x \subset \mathbb P D_x$ is isomorphic to ${\bf S} \subset \mathbb P {\mathfrak g_{-1}}$ under a graded Lie algebra isomorphism
$  \mathfrak m \rightarrow   \Symb_x(D) $.

Two ${\bf S}$-structures $\mathcal S_1$ on $(M_1,D_1)$ and $\mathcal S_2$ on $(M_2,D_2)$ are said to be \emph{equivalent} if there exists a biholomorphism $\phi \colon M_1 \to M_2$ such that $d\phi \colon \mathbb P TM_1 \to \mathbb P TM_2$ sends $\mathcal S_1 \subset \mathbb P TM_1$ to $\mathcal S_2 \subset \mathbb P TM_2$. The \emph{local equivalence} of two ${\bf S}$-structures is defined similarly for   open subsets $U_1\subset M_1$ and $U_2 \subset M_2$.
\end{definition}

An ${\bf S}$-structure can be interpreted as a $G_0$-structure, and the local equivalence of ${\bf S}$-structures can be checked by using the local equivalence of the corresponding $G_0$-structures under some conditions.

\begin{definition}   Let ${\bf U}$ be a vector space and and let $ {\bf S} $ be a nondegenerate subvariety of $\mathbb P \bf U$. Consider the graded free Lie algebra $F({\bf U})$ generated by $\bf U$. Denote by $I({\bf S})$ the ideal  of $F(\bf U)$ generated by the relation $[v,w]=0$ for $v,w \in {\bf U}$ such that $v \in \widehat{{\bf S}}$ and  $w \in T_v \widehat{{\bf S}}$.  We call the  quotient graded Lie algebra $\mathfrak m({\bf S}, \mathbb P{\bf U}):= F({\bf U})/I({\bf S})$ \emph{the graded Lie algebra determined by} ${\bf S} \subset \mathbb P {\bf U}$.
\end{definition}

\begin{definition}
Let $\mathfrak m=\bigoplus_{p<0}\mathfrak g_{p}$ be a fundamental graded Lie algebra and let ${\bf S}$ be a nondegenerate subvariety of $\mathbb P {\mathfrak g_{-1}}$. If the   graded Lie algebra determined by ${\bf S} \subset \mathbb P  \frak g_{-1} $ is isomorphic to $\mathfrak m$, we say that  $\mathfrak m$ is  \emph{determined by} ${\bf S} \subset \mathbb P \mathfrak g_{-1}$.
\end{definition}

\begin{example} Let $\frak g=\bigoplus_{-\mu \leq i \leq \mu} \frak g_i$ be a simple graded Lie algebra. Let ${\bf S} \subset \mathbb P \frak g_{-1}$ be the projectivization of the cone of  highest weight vectors of the irreducible $\frak g_0$-module $\frak g_{-1}$. Then $\frak m=\bigoplus_{p<0} \frak g_{p}$ is determined by ${\bf S} \subset \mathbb P \frak g_{-1}$
(Proposition 7 of \cite{HM02}).
\end{example}

Let $G(\widehat{\bf S})$ denote {\it the linear automorphism group of $\widehat{\bf S} \subset \frak g_{-1}$}, i.e., the  subgroup of $\GL(\mathfrak g_{-1})$ consisting of   linear automorphism  of $\widehat{\bf S} \subset \frak g_{-1}$. Then $G(\widehat{\bf S})$ acts on
$\mathfrak m({\bf S}, \mathbb P \frak g_{-1})$ 
preserving the graded Lie algebra structure. 
This $G(\widehat{\bf S})$-action defines a homomorphism $G(\widehat{\bf S}) \rightarrow G_0(\frak m)$ induced by the isomorphism between $\mathfrak m({\bf S}, \mathbb P \frak g_{-1})$ 
and $\mathfrak m$. The induced map $G(\widehat{\bf S}) \rightarrow G_0(\frak m)$ is injective. 

\begin{proposition} \label{S-str and G_0-str} Let $\mathfrak m=\bigoplus_{p<0}\mathfrak g_{p}$ be a fundamental graded Lie algebra determined by a nondegenerate subvariety  $ {\bf S} $    of $\mathbb P {\mathfrak g_{-1}}$. Let $G(\widehat{\bf S})$ be \it the linear automorphism group of $\widehat{\bf S} \subset \frak g_{-1}$. Consider the induced map $G(\widehat{\bf S}) \rightarrow G_0(\frak m)$ and let $G_0 \subset G_0(\frak m)$ denote its image.
 Then, there is a one-to-one correspondence between $G_0$-structures  and ${\bf S}$-structures on filtered manifolds of type $\mathfrak m$. Furthermore, two $G_0$-structures $\mathscr P_1$ and $\mathscr P_2$ are equivalent if and only if the corresponding ${\bf S}$-structures $\mathcal S_1$ and $\mathcal S_2$ are equivalent.
\end{proposition}

\begin{proof} Let $(M,D)$ be a filtered manifold of type $\mathfrak m$.
Given   a $G_0$-structure  $\mathscr P$  on $(M,D)$, define $\mathcal S_x$ by $[r] ({\bf S})$ for any $r \in \mathscr P_{x}$, where $[r]$ is the isomorphism $\mathbb P {\frak g_{-1}}  \rightarrow \mathbb P D_x$ induced by the isomorphism  $r:\mathfrak m \rightarrow \Symb_x(D)$. Then $\mathcal S_x$ is well defined and the union $\mathcal S= \cup_{x \in M} \mathcal S_x$ defines an ${\bf S}$-structure on $(M,D)$.

Conversely, let $\mathcal S \subset \mathbb P D$ be an ${\bf S}$-structure on $(M,D)$. Define $\mathscr P_x$ by $\{ r \in \mathscr R_x: \varphi ({\bf S}) = \mathcal S_x \}$. Then $\mathscr P_x$ is well defined and the union $\mathscr P= \cup_{x \in M} \mathscr P_x$ defines a $G_0$-subbundle of the frame bundle $\mathscr R$ of $(M,D)$.
\end{proof}

\subsection{Lie algebra cohomologies}

\begin{notation} Let $\frak m=\bigoplus_{p <0} \frak g_p$ be a nilpotent Lie algebra and $\Gamma$ be a representation space of $\frak m$.
Define a complex
\begin{equation} \label{eq complex}
0 \stackrel{\partial}{\longrightarrow} \Gamma \stackrel{\partial}{\longrightarrow} \Hom( \frak m, \Gamma) \stackrel{\partial}{\longrightarrow} \Hom(\wedge^2 \frak m  ,\Gamma  ) \stackrel{\partial}{\longrightarrow} \dots
\end{equation}
by
\begin{eqnarray*}
\partial \phi (X_1, \dots,  X_{q+1}) &=& \sum_{i=1}^{q+1} (-1)^{i+1}X_i.\phi(X_1, \dots,\hat X_i, \dots , X_{q+1}) \\
&& + \sum_{1 \leq i<j \leq q+1} (-1)^{i+j} \phi([X_i, X_j], X_1, \dots, \hat X_i, \dots, \hat X_j , \dots , X_{q+1})
\end{eqnarray*}
 for $\phi \in \Hom(\wedge^q \frak m, \Gamma)$ and $X_1, \dots, X_{q+1} \in \frak m$.
The cohomology space
$$H^{q}(\frak m, \Gamma):=\frac{\text{Ker}\, (\partial : \Hom(\wedge^q \frak m  , \Gamma) \rightarrow \Hom(\wedge^{q+1}\frak m,   \Gamma))}
{\text{Im}\, (\partial : \Hom( \wedge^{q-1}\frak m ,\Gamma) \rightarrow \Hom(\wedge^q\frak m, \Gamma))}
$$ is called the {\it Lie algebra cohomology space associated to the representation $\Gamma$ of} $\frak m$.

Assume that   $\Gamma$ has a  gradation  such that $\frak g_p.\Gamma_{\ell} \subset \Gamma_{p+\ell}$ for $p<0$. Then
$\Hom(\wedge^q \frak m, \Gamma)$ has an induced  grading
$$\Hom(\wedge ^q\frak m, \Gamma)_{\ell} = \bigoplus_{\ell} \Hom(\wedge^q_j \frak m, \Gamma_{j +\ell}),$$
where
$$\wedge^q_j\frak m =\sum_{\substack{j_1+\dots+j_q=j\\ j_1,\dots, j_q <0}} \frak g_{j_1} \wedge \dots \wedge \frak g_{j_q}.$$

For each $\ell$
the complex (\ref{eq complex}) restricts to
  the complex
 $$0 \longrightarrow \Gamma_{\ell} \stackrel{\partial}{\longrightarrow} \Hom(\mathfrak m, \Gamma)_{\ell} \stackrel{\partial}{\longrightarrow} \Hom(\wedge^2 \mathfrak m, \Gamma)_{\ell} \stackrel{\partial}{\longrightarrow} $$
  %
  so that $H^q(\frak m, \Gamma)$   has a gradation
$$H^q(\frak m, \Gamma) = \bigoplus_{\ell } H^{ q}(\frak m, \Gamma)_{\ell}$$

\end{notation}

 \subsection{Prolongation methods} \label{sect:prolongation methods}

 We review the theory of Cartan connections and  prolongation methods in \cite{Ta}, \cite{Mo93}, \cite{HL19},  \cite{T70}, and \cite{HM}.

 Let $\frak m=\bigoplus_{p<0} \frak g_p$ be a fundamental graded Lie algebra and $G_0$ be a connected subgroup of $G_0(\frak m)$ with Lie algebra $\frak g_0$. Then there is a unique maximal transitive graded Lie algebra    $\frak g=\bigoplus_{\ell \in \mathbb Z} \frak g_{\ell}$ extending $\mathfrak m \oplus \mathfrak g_0$, called the  {\it prolongation}  of $(\frak m, \frak g_0)$.
 For $\ell \geq 1$,
 $\mathfrak g_{\ell}$ is  given by
 $$ \{ \alpha \in \oplus_{p<0}\Hom(\mathfrak g_p, \mathfrak g_{p+\ell}):  \alpha([u,v]) =[\alpha(u),v] + [u, \alpha(v)] \text{ and } \alpha(u) \in \frak g_{\ell-1} \text{ for all } u,v  \in \mathfrak g_{-1}\},  $$
and the Lie bracket $[\,\,,\,\, ]: \frak g_{\ell} \times \frak g_k \rightarrow \frak g_{\ell+k}$ is given by:
\begin{itemize}
    \item
$[\alpha, u] =\alpha(u)$ for $\alpha \in \frak g_{\ell}$ and  $u \in \frak g_k$ for $k <0$;

\item $[\alpha, \beta](u) =[\alpha(u),\beta] +[\alpha, \beta(u)]$ where $u \in \frak m$   for $\alpha \in \frak g_{\ell}$ and $\beta \in \frak g_{k}$ for $0 \leq k $
\end{itemize}
  (Section 5 of \cite{T70}).
Assume that  $\frak g$ is finite dimensional. Then there is a Lie group $G$ and its subgroup $G^0$ which contains $G_0$, with Lie algebras $\frak g$ and $\frak g^0:=\bigoplus _{\ell \geq 0} \frak g_{\ell}$. 
%


 \begin{definition}
By {\it a Cartan connection of type} $G/G^0$, we mean a principal $G^0$-bundle $P$ on a manifold $M$ with a $\frak g$-valued 1-form $\theta$ on $P$ satisfying the following properties.
\begin{enumerate}
\item $\theta: T_zP \rightarrow \frak g$ is an isomorphism for all $z  \in P$
\item $R_a^* \theta = \mathrm{Ad} (a)^{-1}\theta$ for $a \in G^0$
\item $\theta(\widetilde A) =A $ for $A \in \frak g^0$.
\end{enumerate}
Two Cartan connections $(P_1,\theta_1)$ and $(P_2, \theta_2)$ of type $G/G^0$ are {\it isomorphic} if there is a bundle isomorphism   $\Phi: P_1 \rightarrow P_2 $ such that $\Phi^*\theta_2 = \theta_1$. The local isomorphism of two Cartan connections is
defined similarly for open sets $U_1 \subset M_1$ and $U_2 \subset M_2$.
\end{definition}

For example, the quotient map $G \rightarrow G/G^0$ with the Maurer-Cartan form $\theta_G$ of $G$ is a Cartan connection of type $G/G^0$.
 A Cartan connection $(P,\theta)$ of type $G/G^0$ is locally isomorphic to $(G, \theta_G)$ if and only if $d \theta + [\theta, \theta] =0$.
 In this case, we say that the Cartan connection $(P, \theta)$ is {\it flat}.


 \begin{definition} 
 Let $(P,\theta)$ be a Cartan connection of type $G/G^0$. Then there is a function $K: P \rightarrow \Hom(\wedge^2 \frak m, \frak g)$ with
 $$d\theta+ [\theta, \theta]=\frac{1}{2}K(\theta, \theta)$$
called the {\it curvature} of $(P, \theta)$.
 \end{definition}

\begin{definition} \label{example simple}   Let $\frak g= \bigoplus_{\ell =-\mu}^{ \mu}\frak g_{\ell}$ be a simple graded Lie algebra  
and let $\frak m$ be the negative part $\bigoplus_{\ell<0} \frak g_{\ell}$. Let $G_0 \subset G_0(\frak m)$ be the subgroup with Lie algebra $\frak g_0$.
 Define a Hermitian metric $(\,\,,\,\,)$ on $\frak g$ induced by the Killing form of $\frak g$.  Denote by $\partial ^*$ the adjoint of $\partial$ with respect to $(\,\,,\,\,)$.
Then
$$\Hom(\wedge^2 \frak m, \frak g) = \partial \Hom(\frak m, \frak g) \oplus \Ker \partial ^*    $$
and
$$\Hom(\wedge^2 \frak m, \frak g)_{\ell+1} = \partial \Hom(\frak m, \frak g)_{\ell+1} \oplus (\Ker \partial ^*)_{\ell+1}    .$$
A Cartan connection $(P, \theta)$ of type $G/G^0$ is said to be {\it normal} if its curvature $K$ satisfies that its component $ K_{\ell+1}$ of degree $\ell+1$ has values in $ (\Ker \partial ^*)_{\ell+1}$ for any $\ell \geq 0$.

\end{definition}

\begin{theorem} [Theorem 2.7 and Theorem 2.9 of \cite{Ta}] \label{theorem tanaka} Let $\frak m$ and $G_0$ be as in Definition \ref{example simple}. Assume that $\frak g$ is the prolongation of $(\frak m, \frak g_0)$.
 Then for any $G_0$-structure $\mathscr P  $ on a filtered manifold $(M,D)$ of type $\frak m$, there is a normal Cartan connection $(P,\theta)$ of type $G/G^0$. 
Furthermore, given two $G_0$-structures $\mathscr P_1 $ on $ (M_1, D_1)$ and $\mathscr P_2 $ on $(M_2, D_2)$, $\mathscr P_1$ and $\mathscr P_2$ are locally equivalent if and only if the corresponding normal Cartan connections $(P_1,\theta_1)$ and $(P_2, \theta_2)$ are locally isomorphic.

Given a $G_0$-structure $\mathscr P$ on a filtered manifold  $(M,D)$  of type $\frak m$,  define a vector bundle $\mathcal H^2_k$ on $M$  by  $\mathcal H^2_k:=\mathscr P \times _{G_0} H^2(\frak m, \frak g)_k$ for $k \geq 1$.   If $H^0(M, \mathcal H_k^2)$ is zero for all $k \geq 1$, then the corresponding Cartan connection $(P, \theta)$ is flat, and $\mathscr P$ is locally equivalent to the standard $G_0$-structure on $G/G^0$.  
\end{theorem}

Theorem \ref{theorem tanaka} is extended to the case when $(\frak m, G_0)$ satisfies the condition (C) or $\mathscr P$ satisfies a pseudo-concavity type condition.


\begin{theorem} [Theorem 3.10.1 of \cite{Mo93}] \label{theorem morimoto} Assume that $(\mathfrak m,G_0)$ satisfies the condition (C), that is,  there is a subspace $W=\oplus_{\ell\geq 0}W_{\ell+1}$ of  $\Hom(\wedge^2 \mathfrak m, \mathfrak g)$ with
 $$\Hom(\wedge ^2 \mathfrak m, \mathfrak g)_{\ell+1} = W_{\ell+1} \oplus \partial \Hom(\mathfrak m, \mathfrak g)_{\ell+1} \,\,\text{ for any }\ell \geq 0 , $$
  which is stable under the action of $G^0$.  Then for any $G_0$-structure $\mathscr P  $ on a filtered manifold $(M,D)$ of type $\frak m$, there is a Cartan connection $(P,\theta)$ of type $G/G^0$, whose curvature has value in $W$.
Furthermore, given two $G_0$-structures $\mathscr P_1 $ on $ (M_1, D_1)$ and $\mathscr P_2 $ on $(M_2, D_2)$, $\mathscr P_1$ and $\mathscr P_2$ are locally equivalent if and only if the corresponding Cartan connections $(P_1,\theta_1)$ and $(P_2, \theta_2)$ are locally isomorphic.

\end{theorem}

For $\frak m$ and $G_0$   as in Definition \ref{example simple},
  $\Ker \partial ^*$ is stable under the action of $G^0$ (Lemma 1.12 of \cite{Ta}), and thus
thus  $(\frak m, G_0)$ satisfies the condition (C).

\begin{theorem} [Theorem 2.17 of \cite{HL19} and Theorem 2.6 of \cite{HL21}] \label{theorem HwangLi} Let $\frak g=\bigoplus_{\ell=-\nu}^{\mu} \frak g_{\ell}$ be the prolongation of $(\frak m, \frak g_0)$.
Given a $G_0$-structure $\mathscr P$ on a filtered manifold $(M,D)$ of type $\frak m$, define  define a vector bundle $\mathcal A^2_k$ on $M$  by  $\mathcal A^2_k:=\mathscr P \times _{G_0}  \left( \Hom( \wedge^2 \frak m, \frak g )_k/ \partial \Hom(\frak m, \frak g)_k \right)$ for $k \geq 1$.
If $H^0(M, \mathcal A_k)=0$ for $1 \leq k \leq \mu+\nu$, then there is a Cartan connection $(P,\theta)$.
If, furthermore, $H^0(M, \mathcal A_k)=0$ for $k \geq \mu+\nu+1$, then the corresponding  Cartan connection $(P, \theta)$ is flat.
\end{theorem}

Theorem \ref{theorem tanaka}, Theorem \ref{theorem morimoto}, and Theorem \ref{theorem HwangLi} enables us to transform the local equivalence problem of geometric structures to the local isomorphism problem  of Cartan connections, and the latter is more systematic than the former.
%
To deal with more general cases,  we weaken the requirement that $P  $ should be a principal bundle on $M$ as follows.

   A {\it geometric structure of order 0 of type} $(\frak m, G_0)$ is a $G_0$-structure $\mathscr P$ on a filtered manifold $(M,D)$ of type $\frak m$. For $\ell \geq 1$, we call a sequence    of principal bundles
  $${\bf P}^{(\ell)}: \mathscr P^{(\ell)} \longrightarrow \mathscr P^{(\ell-1)} \longrightarrow \dots \longrightarrow \mathscr P^{(0)} \longrightarrow M$$    a {\it geometric structure of order} $\ell$ {\it of type} $(\frak m, G_0, \dots, G_{\ell})$ if for $0 \leq i \leq \ell-1$,
  \begin{itemize}
      \item ${\bf P}^{(i)}:\mathscr P^{(i)} \longrightarrow \mathscr P^{(i-1)} \longrightarrow \dots \longrightarrow M$ is a geometric structure of type $(\frak m, G_0, \dots, G_i)$;
      \item $\mathscr P^{(i+1)} \longrightarrow \mathscr P^{(i)}$ is a principal $G_{i+1}$-subbundle of the universal  frame bundle $\mathscr S^{(i+1)}  {\bf P}^{(i)} \longrightarrow \mathscr P^{(i)}$  of $  {\bf P}^{(i)}$ of order $i+1$.
  \end{itemize}
   For the definition of the universal frame bundle $\mathscr S^{(i+1)}   {\bf P}^{(i)}$ of order $i+1$ of a geometric structure ${\bf P}^{(i)}$, see Definition 2.1 of  \cite{HM}.  The property we use is that a map between two geometric structures $  {\bf P}^{(i)} $, $ {\bf Q}^{(i)}$ of order $i$ induces a map between their universal frame bundles $\mathscr S^{(i+1)} {\bf  P}^{(i)}$, $ \mathscr S^{(i+1)}{\bf Q}^{(i)}$ of order $i+1$.


 The equivalence of two geometric structures ${\bf P}^{(\ell)} $ and ${\bf Q}^{(\ell)} $ is defined  inductively as follows.
Two geometric structures ${\bf P}^{(\ell)}:\mathscr P^{(\ell)} \longrightarrow \mathscr P^{(\ell-1)} \longrightarrow \dots \longrightarrow M$ and ${\bf Q}^{(\ell)}:\mathscr Q^{(\ell)} \longrightarrow \mathscr Q^{(\ell-1)} \longrightarrow \dots \longrightarrow M$ are {\it equivalent} if their truncations ${\bf P}^{(\ell-1)}$ and ${\bf Q}^{(\ell-1)}$ are equivalent and the lifting $\mathscr S^{(\ell)} {\bf P}^{(\ell-1)} \rightarrow \mathscr S^{(\ell)} {\bf Q}^{(\ell-1)}$ of their equivalence   maps $\mathscr P^{(\ell)}$ onto $\mathscr Q^{(\ell)}$. \\

Fix a set of subspaces $W=\{W^1_{\ell}, W^2_{\ell+1}\}_{\ell \geq 0}$ such that
\begin{eqnarray*}
 \Hom (\frak m,  \frak g )_{\ell} &= &W_{\ell}^1 \oplus \partial  \frak g_{\ell}\\
 \Hom(\wedge^2\frak m,   \frak g )_{\ell+1} &= &W_{\ell+1}^2 \oplus \partial \Hom(\frak m,  \frak g )_{\ell+1} .
 \end{eqnarray*}
Note that we don't require that the complement $\oplus_{\ell \geq 0} W^2_{\ell+1}$ of $\partial \Hom(\frak m, \frak g)$ should be stable under the action of $G^0$.

\begin{theorem} [Theorem 8.3 of \cite{T70}, Theorem 3.1  of \cite{HM}] \label{thm:prolongation methods}

  Let $\mathscr P$ be a $G_0$-structure on a filtered manifold $(M,D)$ of type $\frak m$.
  Then for each $\ell \geq 1$, there is a geometric structure
 $$  \mathscr S_W^{(\ell)}\mathscr P \stackrel{G_{\ell}}{\longrightarrow} \mathscr S_W^{(\ell-1)}\mathscr P \longrightarrow \dots \longrightarrow \mathscr S_W^{(1)}\mathscr P \stackrel{G_1}{\longrightarrow} \mathscr P \stackrel{G_0}{\longrightarrow} M $$
 of type $(\frak g_-, G_0, \cdots, G_{\ell})$.
  %
  Furthermore, two $G_0$-structures $\mathscr P$ and $\mathscr Q$ are   equivalent if and only if the corresponding geometric structures $\mathscr S_W^{(\ell)}\mathscr P$ and $\mathscr S_W^{(\ell)}\mathscr Q$ are  equivalent.
\end{theorem}

We call the limit $\mathscr S_W \mathscr P=\lim_{\ell} \mathscr S_W^{(\ell)}\mathscr P$ the {\it $W$-normal complete step prolongation of} $\mathscr P$. \\

As in  the case of $G_0$-structures modeled on a rational homogeneous variety $G/G^0$, we get a  local equivalence of geometric structures by the vanishing of   sections of vector bundles $  \mathcal H_k^2  =\mathscr P \times _{G_0} H^2(\frak m, \frak g)_k$.

\begin{proposition} [Theorem 7.4 of \cite{HM}] \label{prop:prolongation methods} Let   $\mathscr P$ be a $G_0$-structure on a filtered manifold  $(M,D)$  of type $\frak m$.
  If  $H^0(M, \mathcal H^2_k)$ is zero for all $k \geq 1$, then the $W$-normal complete step  prolongation  $\mathscr S_W \mathscr P $ of $\mathscr P$ is a Cartan connection of type $G/G^0$ which is flat, and $\mathscr P$ is locally equivalent to the standard $G_0$-structure on $G/G^0$.
\end{proposition}

We will use Proposition \ref{prop:prolongation methods} to prove  Theorem \ref{main results} (see Section \ref{sect:local equivalence} and Section \ref{sect:local equivalence B3}).


\section{Smooth horospherical varieties of Picard number one} \label{horospherical varieties}

\subsection{Classifications}

Let $\mathfrak l$ be a semisimple Lie algebra. We fix a Cartan subalgebra $\mathfrak h$ of $\mathfrak l$ and let $\Phi$ be the set of roots of $\mathfrak l$ relative to $\mathfrak h$. The root space decomposition of $\mathfrak l$ is given by
 \begin{eqnarray*} \mathfrak l =\mathfrak h \oplus \bigoplus_{\alpha \in \Phi} \mathfrak l_{\alpha}, \end{eqnarray*}
where $\mathfrak l_{\alpha}
$ is the root space of $\alpha \in \Phi$. For any root $\alpha$, let $U_{\alpha}$ be the root group of $\alpha$.

\begin{definition}\label{def: characteristic element E}
Let $\{\alpha_1, \cdots, \alpha_m \}$ be a set of simple roots of $\mathfrak l$. We define the \emph{characteristic element  associated with $\alpha_i$} as an element $E_{\alpha_i}$ in $\mathfrak h$ such that $\alpha_j(E_{\alpha_i})= \delta_{i,j}$ for $i,j=1, \dots, m$. Define a gradation $\mathfrak l=\bigoplus_{p\in \mathbb Z} \mathfrak l_p$  on $\mathfrak l$ by $\mathfrak l_p = \{ v \in \mathfrak l : [E_{\alpha_i}, v] = p v\}$  for $p \in \mathbb Z$, which is called the \emph{gradation associated with $\alpha_i$}. In general, given a representation $V$ of $\frak l$ we define the \emph{gradation associated with $\alpha_i$} in a similar way.
\end{definition}

\begin{notation} Given a set $\{\alpha_1, \cdots, \alpha_m \}$ of simple roots of  $\frak l$, let $\{\varpi_1, \dots, \varpi_m\}$ be the set of fundamental weights. For each $i=1, \dots, m$, let $P^{\alpha_i}$ denote the maximal parabolic subgroup associated to $\alpha_i$ and let $V_{\varpi_i}$ denote the irreducible representation with highest weight $\varpi_i$ of semisimple Lie group $L$ corresponding to $\frak l$.
\end{notation}

For a reductive algebraic group $L$, a normal $L$-variety is said to be \emph{horospherical} if
it has an open $L$-orbit $L/H$ whose isotropy group $H$ contains the unipotent part of a Borel subgroup of $L$.

\begin{theorem}[Theorem 0.1 and Theorem 1.11 of \cite{Pa}]\label{Horospheical}
Let $L$ be a reductive group.
 Let $X$ be a smooth nonhomogeneous projective horospherical $L$-variety with Picard number one.
 Then $X$ is uniquely determined by its two closed $L$-orbits $ Y$ and $Z$, isomorphic to $L/P_Y$ and $L/P_Z$, respectively; and $(L, \alpha, \beta)$ in one of the triples of the following list, where $P_Y=P^\alpha$ and $P_Z=P^\beta$ for simple roots $\alpha$ and $\beta$.
 \begin{enumerate}
   \item[\rm(1)] $(B_m,{\alpha_{m-1}},{\alpha_{m}})$ for $m\geq3$;
   \item[\rm(2)] $(B_3,{\alpha_{1}},{\alpha_{3}})$;
   \item[\rm(3)] $(C_m,{\alpha_{i+1}},{\alpha_{i}})$ for $m\geq2$, $i\in \{1,\ldots, m-1\}$;
   \item[\rm(4)] $(F_4,{\alpha_{2}},{\alpha_3})$;
   \item[\rm(5)] $(G_2,{\alpha_{2}},{\alpha_1})$.
 \end{enumerate}
  Moreover, the automorphism group $\Aut(X)$ of $X$ is
  $(SO(2m+1)\times \mathbb C^*)\ltimes V_{\varpi_m}$,   $(SO(7)\times \mathbb C^*)\ltimes V_{\varpi_3}$,  $((Sp(2m)\times \mathbb C^*)/\{\pm1\})\ltimes V_{\varpi_1}$,    $(F_4\times \mathbb C^*)\ltimes V_{\varpi_4}$  and $(G_2\times \mathbb C^*)\ltimes V_{\varpi_1}$, respectively.

  Finally, $\Aut(X)$ has two orbits in $X$,   $Z$ and the complement of $Z$ in $X$.
\end{theorem}

In Theorem \ref{Horospheical}, $Y=L/P_Y$ is contained in the open orbit $X^0$ of $\Aut(X)$ in $X$. In particular,  the base point $o$ of $Y=L/P_Y$ is contained in $X^0$. We will take $o$ as the base point of the quasi- homogeneous variety $X$. Let $\mathfrak g=(\mathfrak l+\mathbb C) \rhd U$ be the Lie algebra of $\Aut(X)$, where $\frak l$ is the Lie algebra of $L$. The characteristic element associated with $\alpha$ as in Definition \ref{def: characteristic element E} gives a gradation on $\frak l$ and a gradation on $U$. Then we shift the gradation on $U$ to identify the part $U_{-1}\oplus \bigoplus_{p<0} \frak l_{p}$ with the tangent space of $X$ at $o\in X$.

\begin{proposition}[Section 2, Proposition 48 and Proposition 49 of \cite{Ki}]\label{tangentsp of X}
Let $X$ be a smooth nonhomogeneous projective horospherical variety $(L, \alpha, \beta)$ of Picard number one. Let $G=\Aut(X)$ and let $\mathfrak g=(\mathfrak l+\mathbb C) \rhd U$ be the corresponding Lie algebra.
Then,
 \begin{enumerate}
    \item[(1)] there is a grading on $\frak l$ and $U$,
\begin{eqnarray*}
\mathfrak l=\bigoplus_{k=-\mu }^{\mu}\mathfrak l_k \text{ \, and \, }  U=\bigoplus_{k=-1}^{\nu }U_k,
\end{eqnarray*}
such that, with the   grading being defined    by
 \begin{eqnarray*}
 \mathfrak g_0&:=&(\mathfrak l_0\oplus \mathbb C)\rhd U_{ 0} \\
 \mathfrak g_p&:=&\mathfrak l_p \oplus U_{p} \text{ for } p \not=0,
 \end{eqnarray*}
 the negative part $\mathfrak m=\bigoplus_{p<0}\mathfrak g_{p}$ of $\frak g$  is identified with the tangent space of $X$ at the base point $o$ of $X$. By convention we set $\frak l_k=0$ for $k<-\mu$ or $k > \mu$, and $U_k=0$ for $k <-1$ or $k >\nu$.
 \item[(2)] $H^1(\frak m, \frak g)_{p}=0$ for $p>0$.
 \item[(3)] $\frak g=\bigoplus_{\ell \in \mathbb Z} \frak g_{\ell}$ is the prolongation of $(\frak m, \frak g_0)$.
 \end{enumerate}
\end{proposition}
As $\frak l_0$-representations, $\frak l_k$ and $U_k$ are irreducible, and as $\frak g_0$-representations, $\frak g_k$ is irreducible (the proof of Lemma 27 of \cite{Ki}). Due to (2) and (3), we could consider the prolongation methods of section \ref{sect:prolongation methods} for a smooth nonhomogeneous projective horospherical variety of Picard number one.

\subsection{Varieties of minimal rational tangents}

In this section we will describe the varieties of minimal rational tangents of horospherical varieties in the list of Theorem \ref{Horospheical}.
We will use the same notions as in Proposition \ref {tangentsp of X}. As we mentioned in the previous subsection, we take the base point $o$ of $Y=L/P_Y$ as the base point of $X$.
 For the root $\alpha$, we define $C_{\alpha}:=\overline{U_{-\alpha}.o}\subset Y$. Then, $C_{\alpha}$ is a minimal rational curve in $Y$ and thus in $X$.

For an arbitrary reductive group $L$ and for a finitely many irreducible $L$-representation spaces $V_i$ ($i=1, \dots, r$),  let $\mathcal H_{L}(\oplus_{i=1}^r V_i)$ denote the closure of the sum of highest weight vectors $v_i$ of $V_i$ in $\mathbb P(\oplus_{i=1}^r V_i)$. For example, $\mathcal H_{L}(V)$ for an irreducible representation space $V$ is the highest weight orbit, and $(L, \alpha_i, \alpha_j)$ is $\mathcal H_{L}(V_{\varpi_i} \oplus V_{\varpi_j})$, where $V_{\varpi_i}$ (respectively, $V_{\varpi_j}$) is the irreducible representation of $L$ of highest weight $\varpi_i$ (respectively, $ \varpi_j$).

 Let $L_0$ be the subgroup of $G=\Aut(X)$ with Lie algebra $\mathfrak l_0$. We need the descriptions of the representation $U_{-1}\oplus \mathfrak l_{-1}$ of $L_0$ and $\mathcal H_{L_0}( U_{-1}\oplus \mathfrak l_{-1})$ for $(B_m, \alpha_{m-1}, \alpha_{m})$, $(B_3, \alpha_{1}, \alpha_{3})$, $(C_m, \alpha_{m}, \alpha_{m-1})$, $(F_{4}, \alpha_{2}, \alpha_{3})$ and $(G_{2}, \alpha_{2}, \alpha_{1})$ and also the representation $U_{-1} \oplus \frak l_{-1} \oplus \frak l_{-2}$ of $L_0$ and $\mathcal H_{L_0}(U_{-1} \oplus \frak l_{-1} \oplus \frak l_{-2} )$ for $(C_m, \alpha_{i+1}, \alpha_{i})$, $1 < i+1 < m$. The following are parts of Lemma 3.5.1 and Proposition 3.5.2 of \cite{Ki0}.

\begin{lemma}\label{lem:type_descrip}
\begin{enumerate}

\item[(1)] $(B_m, \alpha_{m-1}, \alpha_{m})$, $m>2$ where $U=V_{\varpi_m}$; Denote $L_0=A_1\times A_{m-2}$. Let $V$ be the standard representation of $A_1$ and $W^*$ be the standard representation of $A_{m-2}$. Then
\begin{eqnarray*}
\mathfrak l_{-1}&=& \Sym^2 V \otimes W \\
U_{-1}&=& V.
\end{eqnarray*}
The closure  $\mathcal H_{L_0}( U_{-1}\oplus \mathfrak l_{-1})$ of the $L_0$-orbit at $v + v^2\otimes w$ is
\begin{eqnarray*}
 \mathcal H_{L_0}( U_{-1}\oplus \mathfrak l_{-1})&=&\mathbb P \{c v + v^2 \otimes w: c \in \mathbb C, v \in V, w \in W\} \\
 &\simeq& \mathbb P (\mathcal O_{\mathbb P(V)}(-1) \oplus \mathcal O_{\mathbb P(V)}(-2)^{m-1}),
\end{eqnarray*}
where $\dim V=2$ and $ \dim W=m-1$.

\item[(2)]$(B_3, \alpha_{1}, \alpha_{3})$ where $U=V_{\varpi_3}$; Denote $L_0=B_2$. Let $V$ be the spin representation of $B_2$. Let $W$ be the standard representation of $B_2$. Then
\begin{eqnarray*}
\mathfrak l_{-1}&=& W\\
U_{-1}&=& V.
\end{eqnarray*}
The closure $\mathcal H_{L_0}( U_{-1}\oplus \mathfrak l_{-1})$ of the $L_0$-orbit at $v+w$ is the horospherical variety of type $(C_2, \alpha_2, \alpha_1)$, the odd symplective Grassmannian $\Gr_w(2,\mathbb C^5)$ of isotropic $2$-subspaces in $\mathbb C^5$.

\item[$ {\rm (3)_{i+1}}$] $(C_m, \alpha_{i+1}, \alpha_{i})$, $1 < i+1 < m$ where $U=V_{\varpi_1}$; Denote $L_0=A_{i}\times C_{m-i-1}$. Let $V^*$ be the standard representation of $A_{i}$, let $Q^*$ be the standard representation of $C_{m-i-1}$ and $W:= \mathbb C \oplus Q$. Then
\begin{eqnarray*}
\mathfrak l_{-2}&=& \Sym^2V \\
U_{-1} \oplus\mathfrak l_{-1} &=& V \otimes W.
\end{eqnarray*}

The closure $\mathcal H_{L_0}(U_{-1} \oplus \frak l_{-1} \oplus \frak l_{-2}  )$ of the  $L_0$-orbit at $v\otimes w + v^2 $ is
\begin{eqnarray*}
\mathcal H_{L_0}(U_{-1} \oplus \frak l_{-1} \oplus \frak l_{-2} ) &=& \mathbb P \{c v \otimes w + v^2 : c \in \mathbb C, v \in V, w \in W \}\\
& \simeq &\mathbb P (\mathcal O_{\mathbb P(V)}(-1)^{2m-2i-1} \oplus \mathcal O_{\mathbb P(V)}(-2)),
\end{eqnarray*}
where $\dim V=i+1$ and $\dim W=2m-2i-1$.

\item[${\rm (3)_{m}}$] $(C_m, \alpha_{m}, \alpha_{m-1})$ where $U=V_{\varpi_1}$; Denote $L_0=A_{m-1}$. Let $V^*$ be the standard representation of $A_{m-1}$.
Then
\begin{eqnarray*}
\mathfrak l_{-1}&=& \Sym^2V \\
U_{-1}&=& V.
\end{eqnarray*}
The closure  $\mathcal H_{L_0}( U_{-1}\oplus \mathfrak l_{-1})$ of the $L_0$-orbit at $v+ v^2$ is
 \begin{eqnarray*}
\mathcal H_{L_0}( U_{-1}\oplus \mathfrak l_{-1})&=& \mathbb P \{c v + v^2 : c \in \mathbb C, v \in V\} \\
&\simeq& \mathbb P (\mathcal O_{\mathbb P(V)}(-1) \oplus \mathcal O_{\mathbb P(V)}(-2)),
\end{eqnarray*}
where $\dim V=m$.

\item[(4)] $(F_{4}, \alpha_{2}, \alpha_{3})$ where   $U=V_{\varpi_4}$; Denote $L_0=A_{1}\times A_2$. Let $V$ be the standard representation of $A_{1}$ and $V^*=V$ and let $W^*$ be the standard representation of $A_{2}$. Then
\begin{eqnarray*}
\mathfrak l_{-1}&=& \Sym^2V\otimes W \\
U_{-1}&=&V.
\end{eqnarray*}
The closure   $\mathcal H_{L_0}( U_{-1}\oplus \mathfrak l_{-1})$ of the $L_0$-orbit at $v+ v^2\otimes w$ is
\begin{eqnarray*}
\mathcal H_{L_0}( U_{-1}\oplus \mathfrak l_{-1})&=& \mathbb P \{c v + v^2 \otimes w: c \in \mathbb C, v \in V, w \in W\} \\
&\simeq& \mathbb P (\mathcal O_{\mathbb P(V)}(-1) \oplus \mathcal O_{\mathbb P(V)}(-2)^{2}),
\end{eqnarray*}
where $\dim V=3$ and $ \dim W= 2$.

\item[(5)] $(G_{2}, \alpha_{2}, \alpha_{1})$ where $U=V_{\varpi_1}$; Denote $L_0=A_{1}$. Let $V$ be the standard representation of $A_{1}$ such that $V^*=V$. Then
\begin{eqnarray*}
\mathfrak l_{-1}&=& \Sym^3V \\
U_{-1}&=& V.
\end{eqnarray*}
The closure $\mathcal H_{L_0}( U_{-1}\oplus \mathfrak l_{-1})$ of the  $L_0$-orbit at $v+ v^3$ is
\begin{eqnarray*}
\mathcal H_{L_0}( U_{-1}\oplus \mathfrak l_{-1})&=& \mathbb P \{c v + v^3 : c \in \mathbb C, v \in V\} \\
&\simeq& \mathbb P (\mathcal O_{\mathbb P(V)}(-1) \oplus \mathcal O_{\mathbb P(V)}(-3)),
\end{eqnarray*}
where $\dim V=2$.
\end{enumerate}

\end{lemma}

\begin{proposition} [Proposition 3.5.2 of \cite{Ki0}]\label{vmrt of X} Let $X$ be a smooth nonhomogeneous projective horospherical variety $(L, \alpha, \beta)$ of Picard number one.
 Let $\mathcal C_o(X) \subset \mathbb P (T_oX) $ denote  the variety of minimal rational tangents of $X$  at the base point $o$. Then
 \begin{eqnarray*}
 \mathcal C_o(X) = \left\{ \begin{array}{l}
 \mathcal H_{L_0}( U_{-1} \oplus\mathfrak l_{-1} \oplus \mathfrak l_{-2} ) \text{ if } X \text{ is }  (C_m, \alpha_{i+1}, \alpha_{i}) \text{ for }  1 \leq i < m \\
 \mathcal H_{L_0}(  U_{-1}\oplus \mathfrak l_{-1}), \text{ otherwise. }
 \end{array}\right.
   \end{eqnarray*}
\end{proposition}

\begin{proof} We note that the variety $\mathcal C_o(Y)$ of minimal rational tangents of $Y$ at $o$ is contained in the variety $\mathcal C_o(X)$ of minimal rational tangents of $X$ at $o$.

Assume that $X$ is not $(C_m, \alpha_{i+1}, \alpha_{i}), m>2, i=1,\ldots, m-2$. Since $Y=L/P_Y$ is associated with a long simple root $\alpha$, by Proposition 1 of Hwang-Mok (\cite{HM02}), the variety $\mathcal C_o(Y)$ of minimal rational tangents of $Y$ is  $\mathcal H_{L_0}(\mathfrak l_{-1})$. Furthermore, $L_0 \rhd U_{0}$ acts invariantly on $\mathcal C_0(X)$. From $[U_{0}, \frak l_{-1}] \subset U_{-1}$, it follows that the highest weight vector of $U_{-1}$ is contained in $\mathcal C_o(X)$. Therefore,  $\mathcal H_{L_0}( U_{-1}\oplus \mathfrak l_{-1})$ is contained in  $\mathcal C_o(X)$ so that $\dim \mathcal H_{L_0}( U_{-1}\oplus \mathfrak l_{-1})$ is less than or equal to $\dim \mathcal C_o(X)$. However, $\dim \mathcal C_o(X)$ cannot exceed $\dim H^0(C_{\alpha}, N_{C_{\alpha}|X} (-1))$ which is equal to   $K^{-1}_X \cdot C_{\alpha}-2$. 
Now by comparing the dimension of $\mathcal H_{L_0}( U_{-1}\oplus \mathfrak l_{-1})$ with   $K^{-1}_X\cdot C_{\alpha} -2$, we get the desired results.
\begin{center}
    \begin{tabular}{ c c  c  }
types & $K^{-1}_X\cdot C_{\alpha}-2$ \\
\hline
$(B_m, \alpha_{m-1}, \alpha_{m})$ & $m$ \\
$(B_3, \alpha_{1}, \alpha_{3})$ & 5 \\
$(C_m, \alpha_{m}, \alpha_{m-1})$ & $m$ \\
$(C_m, \alpha_{i+1}, \alpha_{i})_{ 1 < i+1 < m}$ & $2m-i-1$ \\
$(F_{4}, \alpha_{2}, \alpha_{3})$ & 4\\
$(G_{2}, \alpha_{2}, \alpha_{1})$ & 2 \\
    \end{tabular}
\end{center}
Here, we use the description of $\mathcal H_{L_0}( U_{-1}\oplus \mathfrak l_{-1})$ in Lemma \ref{lem:type_descrip}.

\medskip
Assume that  $X$ is  $(C_m, \alpha_{i+1}, \alpha_{i}), 1 < i+1 < m$. Then we have
$\mathcal C_o(Y)=\mathcal H_{L_0}(\frak l_{-1} \oplus \frak l_{-2})$. By a similar argument, after replacing $\mathcal H_{L_0}(\frak l_{-1})$ by $\mathcal H_{L_0}(\frak l_{-1} \oplus \frak l_{-2})$, we get the desired result.
\end{proof}

\section{Projective geometry of varieties of minimal rational tangents}
\label{sect:horospherical BF type}

In this section, let $X$ be  either $(B_m, \alpha_{m-1}, \alpha_{m})$, where $m \geq 3$, or $(F_4, \alpha_2, \alpha_3)$.
Let $\mathfrak g = \bigoplus_{p \in \mathbb Z} \mathfrak g_p $ be the Lie algebra of  $\Aut(X)$ with the  gradation given as in Proposition \ref{tangentsp of X} and  $\mathfrak m:= \bigoplus _{p <0} \mathfrak g_p$ be its negative part. As in Lemma \ref{lem:type_descrip}, let $V$ and $W$ be the vector spaces with $(\dim V, \dim W)=(2, m-1)$ for $(B_m, \alpha_{m-1}, \alpha_{m})$ and $(\dim V, \dim W)=(3, 2)$ for $(F_4, \alpha_2, \alpha_3)$, and set $${\bf U}:= V \oplus (\Sym^2(V) \otimes W).$$ Then $\bf U$ can be identified with $\frak g_{-1}$. Let
$${\bf S}=\mathbb P \{ v + v^2 \otimes w: v \in V, w \in W\} \subset \mathbb P\bf U $$

be the variety of minimal rational tangents of $X$ at the base point as in Proposition \ref{vmrt of X}.

\subsection{Projective geometries of VMRTs}

We will show that the variety ${{\bf S}}  \subset \mathbb P \bf U$   of minimal rational tangents of $X$ at the base point satisfies the conditions in Proposition \ref{parallel transport of III}.

\begin{lemma} \label{parallel transport of III-X case} 
 Let $s$ be  the codimension of $T^{(2)}\widehat{\bf S}$ in ${\bf U}=\frak g_{-1}$.

  \begin{enumerate}
  \item The third fundamental form of ${\bf S} \subset \mathbb P(\bf U)$ is surjective;
  \item The dimension of $  [\frak g_{-\alpha}, \frak g_{-1}]   $ is   $  s$, where $\alpha$ is the simple root which gives the gradation on $\frak g$.
\end{enumerate}
\end{lemma}

\begin{proof}
 (1) The tangent space $T_{\beta} \widehat{\bf S}$ at $\beta = v + v^2 \otimes w \in \widehat{\bf S} $ is given by
$$T_{\beta}\widehat{{\bf S}} = \{ v' + 2 v \circ v' \otimes w + v^2 \otimes w': v' \in V, w' \in W\}.$$
The second fundamental form $II_{\beta} : \Sym^2 T_{\beta}\widehat{{\bf S}} \rightarrow {\bf U} /T_{\beta}\widehat{{\bf S}}$ is
\begin{eqnarray*}
&& II_{\beta}(v' + 2v \circ v' \otimes w, v'' + 2 v\circ v'' \otimes w) = 2v' \circ v'' \otimes w \\
&& II_{\beta}(v' + 2v \circ v' \otimes w, v^2 \otimes w') = 2 v \circ v' \otimes w' \\
&& II_{\beta}(v^2 \otimes w', v^2 \otimes w'')=0
\end{eqnarray*}
where $v',v'' \in V$ and $w', w'' \in W$. The third fundamental form $III_{\beta} : \Sym^3 T_{\beta}\widehat{{\bf S}} \rightarrow {\bf U} /T_{\beta}^{(2)}\widehat{{\bf S}}$ is zero except
$$III_{\beta}(v' + 2 v\circ v' \otimes w, v'' + 2 v \circ v'' \otimes w, v^2 \otimes w') = 2 v' \circ v'' \otimes w' \mod T_{\beta}^{(2)}\widehat{{\bf S}}, $$
where $v',v'' \in V$ and $w' \in W$.
Thus the third osculating space is the whole space $\bf U$.

(2) The dimension $r$ of the image of $II_{\beta}$ is $m$ in the first case, and is $7$ in the second case. $(\dim{\bf{U}}, \dim\widehat{{\bf S}})=(\dim D_x, p+1)$ is $(2+3(m-1)=3m-1,m+1)$ in the first case, and is $(15, 5)$ in the second case. Thus $s=\dim {\frak g_{-1}} -(1+p+r)$ is $ m-2 $ in the first case, and is $ 3$ in the second case.

On the other hand, the minimum of $\dim [\xi, D_x]$ occurs when $\xi \in \frak g _{-\alpha}$ where $\alpha$ is the simple root which gives the gradation on $\frak g$. Furthermore, the dimension of $[\xi , D_x]=[\xi , \mathfrak l_{-1}]$  is equal to the number of roots $\gamma$ with $\mathfrak g_{-\gamma} \subset  \mathfrak l_{-1}$ such that $\alpha + \gamma$ is again a root. One can compute directly that this number is  $m-2 $ in the first case, and is $3$ in the second case.
\end{proof}

From the exact sequence of $G_0$-modules where $G_0$ is the subgroup of $G=\Aut(X)$ with Lie algebra $\frak g_0$
$$0 \rightarrow V \rightarrow {\bf U} \rightarrow \Sym^2V \otimes W \rightarrow 0, $$
we get the following exact sequences.

\begin{lemma} \label{lem:finer exact sequences} Let $\beta=v + v^2 \otimes w $ be an element of  $ \widehat{\bf S}$, where $v \not=0 \in V$ and $w\not=0 \in W$. Denote by $V_0$ the subspace of $V$ generated by $v$ and by $W_0$ the subspace of $W$ generated by $w$. Then
we have the following exact sequences.
\begin{eqnarray*}
0 \rightarrow &\mathbb C \beta& \rightarrow \Sym^2 V_0 \otimes W_0 \rightarrow 0 \\
0 \rightarrow V_0 \rightarrow &T_{\beta}{\widehat{\bf S}}/\mathbb C \beta& \rightarrow \left(V_0 \circ(V/V_0) \otimes W_0\right) \oplus \left( \Sym^2 V_0 \otimes(W/W_0)\right) \rightarrow 0 \\
0 \rightarrow V/V_0 \rightarrow &T^{(2)}_{\beta}\widehat{\bf S}/T_{\beta}\widehat{\bf S}& \rightarrow \left(\Sym^2(V/V_0) \otimes W_0\right) \oplus \left( V_0 \circ (V/V_0) \right) \otimes(W/W_0) \rightarrow 0 \\
0 \rightarrow &{\bf U}/T^{(2)}_{\beta}\widehat{\bf S}& \rightarrow \Sym^2(V/V_0) \otimes (W/W_0) \rightarrow 0
\end{eqnarray*}
\end{lemma}

\begin{proof} From the proof of Lemma \ref{parallel transport of III-X case}, we get the following exact sequences.
\begin{eqnarray*}
0 \rightarrow &\mathbb C \beta& \rightarrow \Sym^2 V_0 \otimes W_0 \rightarrow 0 \\
0 \rightarrow V_0 \rightarrow &T _{\beta}\widehat{\bf S}& \rightarrow V_0 \circ V \otimes W_0 + \Sym^2V_0 \otimes W \rightarrow 0 \\
0 \rightarrow V \rightarrow  &T^{(2)}_{\beta}\widehat{\bf S}& \rightarrow \Sym^2V \otimes W_0 + V_0 \circ V \otimes W \rightarrow 0.
\end{eqnarray*}
Here, we remark that $V_0 \circ V \otimes W_0 + \Sym^2V_0 \otimes W$ and $\Sym^2V \otimes W_0 + V_0 \circ V \otimes W$ are not direct sums:

\begin{eqnarray*}
(V_0 \circ V \otimes W_0)  \cap  (\Sym^2V_0 \otimes W)&=&\Sym^2V_0 \otimes W_0 \\
(\Sym^2V \otimes W_0) \cap ( V_0 \circ V \otimes W)&=& V_0 \circ V \otimes  W_0.
\end{eqnarray*}
By taking the quotients we get the desired exact sequences.
  \end{proof}

\begin{remark}\label{pqrst}
Let $p=\dim {\bf S}$, $q=\dim {\bf Q}/  T_{\beta}\widehat{\bf S}$, $r=\dim T^{(2)}_{\beta}\widehat{\bf S}/T_{\beta}\widehat{\bf S}$, $s=\dim {\bf U}/T^{(2)}_{\beta}\widehat{\bf S}$ as above, and $t=q-r-2s$, where ${\bf Q}:= \mathfrak{m}$ as a vector space. Then we have the following table.

\begin{center}
    \begin{tabular}{ r | c c c c c }
 & p & q & r & s & t\\
 \hline
$(B_m, \alpha_{m-1}, \alpha_{m})$ &$m$&$2m-2$&$m$&$m-2$&${(m-2)(m-3)}/{2}$ \\
$(F_{4}, \alpha_{2}, \alpha_{3})$ &$4$&$10$&$7$&$3$&$5$ \\
    \end{tabular}
\end{center}

\end{remark}

\subsection{Fundamental graded Lie algebras determined by VMRTs}    \label{sect: geometric structures modeled after X}

We will show that ${\bf S}$ and $\frak m$ satisfies the conditions in
Proposition \ref{S-str and G_0-str}, so that there is a one-to-one correspondence between $G_0$-structures and ${\bf S}$-structures on filtered manifolds of type $\frak m$.

\begin{proposition} \label{fundamental graded Lie algebra of vmrt}  $\,$
 \begin{enumerate}
\item The graded Lie algebra  $\mathfrak m $ is fundamental and determined by ${\bf S}  \subset \mathbb P \mathfrak g_{-1}$.
\item Let $\frak g(\widehat{\bf S})$ be the Lie algebra of the linear automorphism group   $G(\widehat{\bf S})$.
Then the induced homomorphism $\frak g(\widehat{\bf S}) \rightarrow \frak g_0(\frak m)$ is injective and its image is $\frak g_0$.
\end{enumerate}
 \end{proposition}

In the remaining part of this subsection we will prove Proposition \ref{fundamental graded Lie algebra of vmrt}. 
Recall that ${\bf U}$ is given by  $${\bf U} = V \oplus (\Sym^2(V) \otimes W)$$ where $V$ and $W$ are vector spaces with $(\dim V, \dim W)=(2, m-1)$ for $(B_m, \alpha_{m-1}, \alpha_{m})$ and $(\dim V, \dim W)=(3, 2)$ for $(F_4, \alpha_2, \alpha_3)$ and  that    ${{\bf S}}$ is given by
$${\bf S}=\mathbb P \{ v + v^2 \otimes w: v \in V, w \in W\}  \simeq \mathbb P (\mathcal O(-1) \oplus \mathcal O(-2)^{k}) \subset \mathbb P\bf U,$$ where $k=\dim W$.
Let ${\bf G}:=SL(V) \times SL(W)$ be a group acting on $\bf S$. Then, ${\bf S}$ is a smooth horospherical ${\bf G}$-variety of rank one. Moreover,  it has two closed ${\bf G}$-orbits, ${\bf Z}  :=\mathbb P(V)$ corresponding to $\mathbb P(\mathcal O(-1))$ and
\begin{eqnarray*}  {\bf Y}   := \mathbb P \{\lambda^2\otimes\mu \in \Sym^2V\otimes W | \lambda \in V ,
\mu \in W \}\end{eqnarray*}
corresponding to a choice of  $\mathbb P(\mathcal O(-2)^{k})$, has one open ${\bf G}$-orbit the complement of ${\bf Y} \cup {\bf Z}$ in ${\bf S}$.

Denote that ${\bf Y} $ is the variety of minimal rational tangents of $Y=L/P_Y$ that is a rational homogeneous associated with a long simple root of type $(B_m, \alpha_{m-1})$ (respectively, of type $(F_4, \alpha_2)$).

\begin{proposition} [Proposition 1 and Proposition 7 of \cite{HM02}, Proposition 4.1 of \cite{HH}] \label{Vmrt ss-LieAl Str}  Let $Y$ be a rational homogeneous space of type $(\mathfrak l , \alpha)$ where  $\alpha$ is  a long simple root. Let $\mathfrak n=\bigoplus _{p <0} \frak l_p$ be the negative part of the graded Lie algebra  $\frak l=\bigoplus_{p\in \mathbb Z}\frak l_p$ with gradation associated to $\alpha$. Let $D_o \subset T_oY$ be the linear span of the homogeneous cone of the variety $\mathcal C_o$ of minimal rational tangents  at a base point $o \in Y$. Then $\frak n$ is the fundamental graded Lie algebra determined by $ \mathcal C_o \subset \mathbb PD_o$.
\end{proposition}

Since $\wedge^2 (\Sym^2 V \otimes W) = \wedge^2 (\Sym^2 V) \otimes \Sym^2 W \oplus \Sym^2 (\Sym^2 V) \otimes \wedge^2 W$, if we decompose $\Sym^2(\Sym^2 V) = \Sym^4V \oplus (\Sym^4V)^{\perp}$, then the Lie bracket $[\,,\,]:\wedge^2 \frak l_{-1} \rightarrow \frak l_{-2}$ is given by the projection map
\begin{eqnarray*}
  \nu:  \wedge^2 (\Sym^2 V \otimes W) & \to &(\Sym^4V)^{\perp} \otimes \wedge^2 W.
\end{eqnarray*}
If $X$ is of type $(B_m, \alpha_{m-1}, \alpha_m)$, then $(\Sym^4V)^{\perp} \otimes \wedge^2 W =  \mathbb C  \otimes \wedge^2 W =  \wedge^2 W$. If $X$ is of type $(F_4, \alpha_2, \alpha_3)$, then $(\Sym^4V)^{\perp} \otimes \wedge^2 W =   \Sym^2(\wedge^2V)  \otimes \mathbb C=\Sym^2(\wedge^2V)$.
Thus $\wedge^2 (\Sym^2 V \otimes W)$ is the direct sum of $\Ker \nu$ and an irreducible representation of $SL(V) \times SL(W)$.

Note that $\frak g_{-1} ={\bf U}$ and $ \frak g_{-2} =  \frak l_{-2}$. The  Lie bracket $[\,,\,]:\wedge^2 \frak g_{-1} \rightarrow \frak g_{-2}$ defines
$$\omega:\wedge^2 (V\oplus \Sym^2 V \otimes W)  \rightarrow \frak l_{-2},$$  where $\omega|_{\wedge^2 (\Sym^2 V \otimes W)}=\nu$ and $\omega(V,V)= \omega(V, \Sym^2V \otimes W)=0$. Then,  $\wedge^2 {\bf U}=\wedge^2 (V\oplus \Sym^2 V \otimes W)$ is decomposed as
$$\wedge^2 (V\oplus \Sym^2 V \otimes W) = \wedge^2 V\oplus (V \wedge (\Sym^2 V\otimes W))\oplus \wedge^2(\Sym^2 V\otimes W) $$
and we have   $\Ker \omega = \wedge^2 V \oplus (V \wedge (\Sym^2 V\otimes W))\oplus \Ker \nu$ 

\begin{lemma}\label{tangent space of the cone V-2}
The kernel $\Ker \omega $ is spanned by $\wedge^2 P \subset \wedge^2 {\bf U}$ where  $P$ is 2-dimensional subspace of ${\bf U}$ tangent to $\widehat{ \bf S }$.
\end{lemma}

\begin{proof} Let $\Xi$ be the subspace of $\wedge^2 (V\oplus \Sym^2 V \otimes W)$ spanned by
   \begin{eqnarray*} \{ \wedge^2 P \subset \wedge^2(V\oplus \Sym^2V \otimes W) | P \mbox{ is tangent to } \widehat{ \bf S }, \dim P=2 \}. \end{eqnarray*}

Since $\widehat{ \bf S } \subset V\oplus \Sym^2V \otimes W $ is a $SL(V)\times SL(W)$-invariant subvariety,
$\Xi$  is also $SL(V)\times SL(W)$-invariant subspace of $\wedge^2 (V\oplus \Sym^2 V \otimes W)$. Furthermore, $\wedge^2 (V\oplus \Sym^2 V \otimes W)$ is the direct sum of $\Ker \omega$ and an irreducible representation of $SL(V) \times SL(W)$.
We will show that $\Ker \omega  $ is contained in $ \Xi$. If so, since $\Xi$ is an $SL(V) \times SL(W)$-invariant subspace and is proper, we get the desired equality.

As the subspace $V$ is contained in $\widehat{\bf S}$, the component $\wedge^2 V$ is contained in $\Xi$.
By Proposition \ref{Vmrt ss-LieAl Str}, the kernel of $\nu$ is spanned by $\wedge^2 P$, where $P$ is tangent to $\widehat{\bf Y}$, and thus is contained in $\Xi$.
We claim that  $V \wedge (\Sym^2 V \otimes W)$ is contained in $\Xi$, which completes the proof.

Let $\beta_t = v_t + v_t^2 \otimes w_t$ be a curve in $\widehat{\bf S}$, where $v_t \in V$ and $w_t \in W$.  By definition of $\Xi$,
 \begin{eqnarray*}
 \beta_0 \wedge \beta'_0 = (v_0 + v_0 ^2 \otimes w_0) \wedge (v_0' +  v'_0 \otimes v_0 \otimes w_0 + v_0 \otimes v'_0 \otimes w_0 + v_0^2 \otimes w_0')
 \end{eqnarray*}
is contained in $\Xi$. Since $\alpha_t:=v_t^2 \otimes w_t$ is a curve in $\widehat{\bf Y} \subset \widehat{\bf S}$,
 $$\alpha_0 \wedge \alpha'_0=(v_0 ^2 \otimes w_0) \wedge (v'_0 \otimes v_0 \otimes w_0 + v_0 \otimes v'_0 \otimes w_0+ v_0^2 \otimes w_0') $$ is contained in $\Xi$. Hence,
$\beta_0 \wedge \beta'_0 -\alpha_0 \wedge \alpha'_0-v_0\wedge v'_0$ which is equal to
 \begin{eqnarray*}
   v_0  \wedge ( v'_0 \otimes v_0 \otimes w_0 + v_0 \otimes v'_0 \otimes w_0 + v_0^2 \otimes w_0')+ (v_0 ^2 \otimes w_0) \wedge  v_0'
 \end{eqnarray*}
 is contained in $\Xi$. Note that
 \begin{eqnarray*}
&&v_0 \wedge (v_0' \otimes v_0 + v_0 \otimes v_0') \\
&& = v_0 \otimes v_0' \otimes v_0 + v_0 \otimes v_0 \otimes v_0' - v_0' \otimes v_0 \otimes v_0 - v_0 \otimes v_0' \otimes v_0 \\
 &&=  v_0^2 \wedge v_0'
 \end{eqnarray*}
 and
 \begin{eqnarray*}
 v_0 \wedge v_0^2 = v_0 \otimes (v_0 \otimes v_0) -(v_0 \otimes v_0) \otimes v_0 = 0
 \end{eqnarray*}
 Therefore, $(v_0' \wedge v_0^2)\otimes w_0$ is contained in $\Xi$.  Here, we consider $V \wedge (\Sym^2 V \otimes W)$ as $(V \wedge \Sym^2 V )\otimes W$.
 This is true for arbitrary $v_0, v_0' \in V$ and $w_0 \in W$, it follows that $V \wedge (\Sym ^2 V \otimes W)$ is contained in $\Xi$.
\end{proof}

\begin{lemma} \label{automorphism group of S} $\Aut^0({\bf S}) = ((SL(V) \times SL(W))/ Z(\bf G)) \times \mathbb C^* \vartriangleright V^* \otimes W^*$ where $Z(\bf G)$ is the center of ${\bf G}=SL(V) \times SL(W)$.
\end{lemma}
\begin{proof}
Since ${\bf S}$ is a smooth horospherical ${\bf G}$-variety of rank one with two closed ${\bf G}$-orbits ${\bf Z}$ and ${\bf Y}$, we will apply the same arguments as in the proof of Lemma 1.1. of \cite{Pa}.
\begin{center}
    \begin{tabular}{ c | c c c c l }
type &$\dim {\bf S}$ & $\dim {\bf Y}$ & $\dim {\bf Z}$ & $N_{{\bf Y}|{\bf S}}$ & $N_{{\bf Z}|{\bf S}}$ \\
 \hline
$(B_m, \alpha_{m-1}, \alpha_{m})$ &$m$&$m-1$&$1$& $\mathcal O(1)$&$\mathcal O(-1)^{m-1}$ \\
$(F_{4}, \alpha_{2}, \alpha_{3})$ &$4$&$3$&$1$&$\mathcal O(1)$&$\mathcal O(-1)^{2}$ \\
    \end{tabular}
\end{center}
Thus $H^0({\bf Y}, N_{{\bf Y}|{\bf S}}) =V^* \otimes W^*$ and $H^0({\bf Z}, N_{{\bf Z}|{\bf S}})=0$.

By the same arguments as in the proof of Lemma 1.1. of \cite{Pa}, the closed $\bf G$-orbit ${\bf Z} $ is stable under the action of $\Aut^0({\bf S})$ and we have
\begin{eqnarray*}
\Aut^0({\bf S}) &=& ({\bf G}/ Z(\bf G) \times \mathbb C^*)  \vartriangleright H^0({\bf Y}, N_{{\bf S}_Y/{\bf S}}) \\
&=& ((SL(V) \times SL(W))/ Z(\bf G))\times \mathbb C^* \vartriangleright V^* \otimes W^*,
\end{eqnarray*}
where $Z(\bf G)$ is the center of $\bf G$.
\end{proof}

\begin{proof} [Proof of Proposition \ref{fundamental graded Lie algebra of vmrt} (1)]
Proposition \ref{Vmrt ss-LieAl Str} says that $\frak l_-=\oplus_{k<0}\frak l_k$ is the fundamental graded Lie algebra determined by ${\bf Y}   \subset \mathbb P(\frak l_{-1})$.
Since $\frak m = U_{-1} \oplus \frak l_-$ and $[U_{-1}, \frak m]=0$,  By Proposition \ref{Vmrt ss-LieAl Str} and Lemma \ref{tangent space of the cone V-2},  the graded Lie algebra $\frak m$ is  fundamental and determined by ${\bf S} \subset \mathbb P\frak g_{-1}$.
\end{proof}

\begin{proof} [Proof of Proposition \ref{fundamental graded Lie algebra of vmrt} (2)]
 We recall that $\mathfrak g_0 =(\mathfrak l_0 \oplus \mathbb C   )\vartriangleright U_{0} \subset \mathfrak g_0(\mathfrak m)$ from Proposition \ref{tangentsp of X}. The center of $\frak l_0$ is of dimension one, the semisimple part of $\frak l_0$ is $\frak{sl}(V) + \frak {sl}(W)$ and the vector space $U_{0}$ is $V^* \otimes W^*$. We compare $\frak g_0$ with the Lie algebra of the neutral component $\Aut^0(\widehat{{\bf S}})$ of the automorphism group of the cone $\widehat{\bf S} \subset \frak g_{-1}$ over $\bf S$. 
 Then the rest of proof follows from Lemma \ref{automorphism group of S}, that is, $\Aut^0(\widehat{{\bf S}})$ is equal to the linear automorphism group $G(\widehat{\bf S})$    and the induced map $\frak g(\widehat{\bf S}) \rightarrow \frak g_0(\frak m)$ is injective whose image   in $\mathfrak g_0(\frak m)$ agrees with $\mathfrak g_0   \subset \mathfrak g_0(\mathfrak m)$.
\end{proof}

\subsection{Parallel transports of VMRTs}
We will show that ${\bf S}$ is not changed under the deformation keeping the second fundamental form and the third fundamental form constant (Proposition \ref{parallel transport of vmrt}).
We adapt arguments in the proof of  Proposition 8.9 of \cite{Hw15}, which proves the same statement as in Proposition \ref{parallel transport of vmrt}  for the case when $X$ is   $(G_2, \alpha_2,\alpha_1)$. \\

Assume that  $X$ is  $(F_4, \alpha_2, \alpha_3)$.
Then ${\bf U}$ and ${\bf S}$ are given by   
  ${\bf U} = V \oplus (\Sym^2(V) \otimes W)$ and
\begin{eqnarray*}{\bf S}=\mathbb P \{ c v + v^2 \otimes w: c \in \mathbb C, v \in V, w \in W\},
\end{eqnarray*}
where  $V$ is a vector space of dimension   $3$ and $W$ is a vector space of dimension  $2$.

Consider ${\bf S}$ as a projective bundle $\mathbb P(\mathcal O(-1) \oplus \mathcal O(-2)^2)$ over $\mathbb P(V) $. Let $\psi \colon {\bf S}  \rightarrow \mathbb P(V) $ denote this $\mathbb P^2$-fibration and
let $\xi$ denote the dual tautological line bundle on $\psi:{\bf S} \rightarrow \mathbb P(V)$. Then $H^0({\bf S}, \xi) = H^0(\mathbb P(V), \psi_*\xi) =V^* \oplus (\Sym^2 V ^*\otimes W^*)$. Hence $\xi$ induces the embedding ${\bf S} \subset \mathbb P({\bf U})$.

The $\mathcal O(-1)$-factor defines a subvariety $Z=\mathbb P(\mathcal O(-1)) \subset {\bf S}$, which is isomorphic to $   \mathbb P^2$  
 and is a section of $\psi$. A choice of an $\mathcal O(-2)$-factor gives a section $B$ of $\psi$ whose linear span $A$ is isomorphic to $\mathbb P^5$ and disjoint from $Z$. We call such a section a \emph{complementary section}. 

Fix an $\mathcal O(-2)$-factor and denote by $B_0$ the corresponding complementary section.
Then the complement ${\bf S} -Z$ of $Z$ is biholomorphic to the total space of the vector bundle $\mathcal O(1)^2$ on $\mathbb P(V)$ whose zero section corresponds to $B_0$. A complementary section corresponds to a section of $\mathcal O(1)^2$. Thus,  for a given   triple $(s_1^1, s_1^2, s_1^3)$ in ${\bf S} -Z$  such that $\psi(s_1^1), \psi(s_1^2), \psi(s_1^3)$ are distinct,  there is a unique complementary section  $B_1$   with  $s_1^1, s_1^2, s_1^3 \in B_1$.

\begin{proposition} \label{characterization of vmrt F}  When $X$ is   $(F_4, \alpha_2, \alpha_3)$,
 ${\bf S} \subset \mathbb P^{14}$ is obtained from the following four data:
   \begin{enumerate}
\item[(i)] a plane $Z \subset \mathbb P^{14}$,
\item[(ii)] two disjoint linear spaces $A_1 \cong \mathbb P^5$ and $A_2 \cong \mathbb P^5$
such that the union $Z \cup A_1 \cup A_2$ span $\mathbb P^{14}$, 
\item[(iii)] two subvarieties $B_1 \subset A_1$ and $B_2 \subset A_2$, each of which is  isomorphic to the Veronese surface $\nu(\mathbb P^2) \subset \mathbb P^5$,
\item[(iv)] two birational maps $\epsilon_1 \colon Z \rightarrow B_1$ and $\epsilon_2 \colon Z \rightarrow B_2$.
   \end{enumerate}
\end{proposition}

\begin{proof}
The locus of planes generated by $z \in Z$ and $\varepsilon_1(z), \varepsilon_2(z)$ is a subvariety of $\mathbb P^{14}$ projectively equivalent to ${\bf S} \subset \mathbb P^{14}$.
\end{proof}

Similarly we have the following.

\begin{proposition} \label{characterization of vmrt B}  When $X$ is  $(B_m, \alpha_{m-1}, \alpha_m)$,
 ${\bf S} \subset \mathbb P^{ 3m-2}$ is obtained from the following four data:
   \begin{enumerate}
\item[(i)] a line $Z \subset \mathbb P^{3m-2}$,
\item[(ii)] $m-1$ linear spaces $A_i \cong \mathbb P^3$, where $1 \leq i \leq m-1$, such that the union $Z \cup \left(\cup_{i=1}^{m-1} A_i\right) $ spans $ \mathbb P^{3m-2}$,
\item[(iii)] $m-1$ subvarieties $B_i \subset A_i$, where $1 \leq i \leq m-1$, each of   which is  isomorphic to the conic $\nu(\mathbb P^1) \subset \mathbb P^2$,
\item[(iv)] $m-1$ birational maps $\epsilon_i \colon Z \rightarrow B_i$, where $1 \leq i \leq m-1$.
   \end{enumerate}
\end{proposition}

\begin{proposition} \label{parallel transport of vmrt}
Let ${\bf S} \subset \mathbb P {\bf U} $ be the variety of minimal rational tangents of
$(B_m, \alpha_{m-1}, \alpha_{m})$, $m \geq 3$ or $(F_4, \alpha_2, \alpha_3)$ at the base point.

  Let $\pi: \mathbb P \mathcal U\rightarrow \mathbb P^1$ be the projectivization of a holomorphic vector bundle $\mathcal U$ over  $  \mathbb P^1$  and let $\mathcal C \subset \mathbb P \mathcal U$ be an irreducible subvariety. Denote by $\varpi$ the restriction of $\pi$ to $\mathcal C$. Assume that
   \begin{enumerate}
   \item $\mathcal C_t:=\varpi^{-1}(t) \subset \mathbb P  \mathcal U_t:=\pi^{-1}(t)$ is projectively equivalent to ${\bf S} \subset \mathbb P {\bf U}$ for all $t \in \mathbb P^1 -\{t_1, \dots, t_k\}$; 
  \item
      for  a general  section $\sigma \subset \mathcal C$ of $\varpi$, the relative second fundamental forms and the relative third fundamental forms of $\mathcal C$ along $\sigma$ are constants.
     \end{enumerate}

Then for any $t \in \mathbb P^1$, $\mathcal C_t \subset \mathbb P(\mathcal U_t)$ is projectively equivalent  to ${\bf S} \subset \mathbb P({\bf U})$.
\end{proposition}

\begin{proof} Assume that $\mathcal C_t \subset \mathbb P \mathcal  U_t$ is projectively equivalent to ${\bf S} \subset \mathbb P {\bf U}$ for all $t  $ in the unit disc $\Delta \subset \mathbb P^1$ except  for $t=0$. Assume further  that for  a general  section $\sigma \subset \mathcal C$ of $\varpi$, the relative second fundamental forms and the relative third fundamental forms of $\mathcal C$ along $\sigma$ are constants. Then there is an open submanifold $\mathcal C^0\subset \mathcal C$ such that $\mathcal C_t^0 \subset \mathcal C_t $ corresponds to an open subset of ${\bf S} -Z \subset {\bf S}$ for $t \not=0  \in \Delta $.   We claim that $\mathcal C_0 \subset \mathbb P \mathcal U_0$ is also projectively equivalent to ${\bf S} \subset \mathbb P {\bf U}$.
It suffice to show that we   get   the four data in Proposition \ref{characterization of vmrt F} and Proposition \ref{characterization of vmrt B} at $t=0$ from $\mathcal C_0=\varpi^{-1}(0) \subset \mathbb P \mathcal  U_0$.
\\

Assume that $X$ is   $(F_4, \alpha_2, \alpha_3)$.
Let $\beta_t \in \mathcal C_t$  be a section of $\varpi$ corresponding to a general point $\beta \in {\bf S}$ for $t\neq 0$ and let $\beta_0$ be the limit. According to the computation of the second fundamental form in the proof of  Lemma  \ref{parallel transport of III-X case} (1), we see that $Baselocus(II_{\beta_t})$ gives us a foliation whose leaves are isomorphic to $\mathbb P^2$-fiber of the projective bundle ${\bf S}$ over $\mathbb P^2$. By the assumption (2), there is also a foliation on the central fiber given by $Baselocus(II_{\beta_0})$, which is exactly $\mathbb P^2$ in $\mathcal C_0=\varpi^{-1}(0) \subset \mathbb P \mathcal  U_0$. We call these subvarieties $\mathbb P^2$ of $\mathcal C_t$ for $t \in \Delta$, corresponding to the $\mathbb P^2$-fiber of  ${\bf S}$, \emph{the planes of the rulings}. \\

(i) Let $\psi_t \colon \mathcal C_t  \rightarrow \mathbb P^2$ be the coresponding fibration for $t \neq 0$ of the projective bundle ${\bf S}=\mathbb P(\mathcal O(-1) \oplus \mathcal O(-2)^2) \rightarrow \mathbb P^2$. Then $Z \subset {\bf S}$ gives familly of distingushied planes $Z_t \simeq \mathbb P(\mathcal O(-1))) \subset \mathcal C_t \subset \mathbb P  \mathcal U_t$ for $t\neq 0$, which is a section of $\psi_t$. Then the limit $Z_0$ is also a $\mathbb P^2$. Hence, there are $\mathbb P^2$ subbundle  $\mathcal Z \subset \mathcal C$ of $\pi: \mathbb P \mathcal U\rightarrow \mathbb P^1$. \\

(ii) Pick 3 distinct points $s_1^1, s_1^2, s_1^3 \in \mathcal C_0 - Z_0$ that lie  in 3 distinct planes of  the rulings and choose local sections $\sigma^1_1, \sigma^2_1, \sigma^2_1$ of $\varpi^0:\mathcal C^0 =\mathcal C - \mathcal Z  \rightarrow \Delta$ such that $\sigma^i_1(0)=s_1^i$ for $i=1, 2, 3$. Then there exists a unique complementary section $B_{1,t}$ with $\sigma^1_1(t), \sigma^2_1(t), \sigma^3_1(t) \in B_{1,t}$  for any $t \not=0$.

Let  $A_{1,t}$ be the linear span of $B_{1,t}$.
Then,   $B_{1,t} \subset A_{1,t}$ is isomorphic to the Veronese surface $\nu(\mathbb P^2) \subset \mathbb P^5$. The limit $A_{1,0}$ of $A_{1,t}\simeq \mathbb P^5 \subset \mathbb P \mathcal U_t$ is  also a projective space  $  \mathbb P^5 \subset \mathbb P \mathcal U_0$, which contains $s_1, s_2, s_3  $.

Choose 3 distinct points  $s^1_2, s^2_2, s^3_2 $ in the complement $\mathcal C_0 - (Z_0 \cup A_{1,0})$ such that  each line $span(s^i_1, s^i_2)$ for $i=1,2,3$ is in the same  plane of the rulings. Choose local sections $\sigma^i_2$ of $\varpi$ such that $\sigma^i_2(t) \in \mathcal C_t  - (Z_t \cup  A_{1,t})$ and $\sigma^i_2(0)=s_2^i$ for $i=1, 2, 3$. Then there exist a  unique complementary section $B_{2,t} $
 such that $\sigma^1_2(t), \sigma^2_2(t), \sigma^3_2(t) \in B_{2,t}$.

 Let  $A_{2,t}$ be the linear span of $B_{2,t}$. 
 The two complementary sections $B_{1,t} \subset A_{1,t}$ and $B_{2,t} \subset A_{2,t}$ are isomorphic to the Veronese surface $\nu(\mathbb P^2) \subset \mathbb P^5$ for $t \not=0$.

 The limit $A_{j,0}$ of $A_{j,t}\simeq \mathbb P^5 \subset \mathbb P \mathcal U_t$  is also projective space $ \mathbb P^5 \subset \mathbb P \mathcal U_0$, $j=1, 2$ such that $s^1_1, s^2_1, s^3_1 \in A_{1,0}$ and $s^1_2, s^2_2, s^3_2 \in A_{2,0}$.
 Then the union $Z_0 \cup A_{1,0} \cup A_{2,0}$ spans $\mathbb P \mathcal U_0$. For, otherwise, there is hyperplane such that $Z_0 \cup A_{1,0}  \cup A_{2,0} \subset \mathbb P^{13} \subset \mathbb P \mathcal U_0$ which contradicts to the constancy of second and third fundamental form. \\

(iii) Consider the limit $B_{j.0}$ of $B_{j,t}$, which is a surface in $A_{j,0}=\mathbb P^5$ of degree less then $4$. The planes of rulings on $\mathcal C_0$ intersecting $A_{j,0}$, give  a analytic surface $B'_{j,0} \subset B_{j,0}$. Let $A'_{j,0}$ be the linear span of $B'_{j,0}$. If $B'_{j,0}$ lies in a hyperplace of $A_{j,0}$, then $\mathbb P (Z_0 \oplus A'_{1,0}\oplus A'_{2,0})$ is contained in a hyperplane of $\mathbb P U_0$ which contradiction to the fundamental forms. Hence, $B'_{j,0}$ must be non-degenerate in $A_{j,0}$ and this means $B_{j,0}$ is irreducible non-degenerate surface. Since the Veronese surface is minimal degree surface, a Veronese surface for $j=1,2$. \\

(iv) With a choice of $Z_t$, $B_{j,t}$ and $A_{1,t}$ for $j=1, 2$ as above in (i) - (ii), we have birational morphisms $\varepsilon_{j,t}\colon Z_0 \rightarrow B_{j,t}$ for $t \neq 0$ such that the planes of rulings on $\mathcal C_t$ are the plane spanned by $z_t \in Z_t$, $\varepsilon_{1,t}(z_t)  \in B_{1,t}$ and $\varepsilon_{2,t}(z_t)  \in B_{2,t}$.
Let $\tilde{\mathcal P}$ be the blow up of the bundle $\pi: \mathbb P \mathcal U\rightarrow \Delta$ along the submanifold $\mathcal Z$ and let ${\mathcal E}$ be the exceptional divisor,  which is biholomorphic to $\mathcal Z \times _{\Delta}\mathbb P (\widehat {\mathcal A_1} \oplus \widehat{\mathcal A_2})$. Let $E_t$ be the exceptional divisor of the blow up of $\mathbb P \mathcal U_t$, which is $Z_t \times \mathbb P (\widehat{A_{1,t}} \oplus \widehat{A_{2,t}})$. 
 For $t \neq 0$, the proper transform of the plane joining $z \in Z_t$, $\varepsilon_{1,t}(z)  \in B_{1,t}$ and $\varepsilon_{2,t}(z)  \in B_{2,t}$ intersects $E_t$ at a point $(z,\epsilon_t(z)) \in  E_t$, where $\epsilon_t(z):=\varepsilon_{1,t}(z)\oplus\varepsilon_{2,t}(z)$. The surface $\Gamma_t:=\{(z, \varepsilon_t (z)| z \in Z_t\}$ corresponds to the graph of $\epsilon_t$.
Let $\Gamma_0 \subset E_0$ be the limit of $\Gamma_t$.
Then there exists an irreducible component $\Gamma'_0$ of $\Gamma_0$ which is the intersection of $E_0$ with the proper transform of the planes in the planes of ruling at $t=0$. By definition, $\Gamma'_0$ dominate $B_{j,0}$ for $j=1, 2$. $\Gamma'_0$   dominates $Z_0$, too. For, otherwise, the intersection with the planes of ruling of $Z_0$ is codimension $\geq 1$, which contradicts to the constancy of the second fundamental forms and the third fundamental forms.

On $\mathcal E$, let $\mathcal L^Z$ be the line bundle which is the pull back of the relative $\mathcal O(1)$ bundle on $\mathcal Z$ and  let $\mathcal L^{A_j}$ be the line bundle which is the pull back of the relative $\mathcal O(1)$ bundle on $\mathcal A_{j}$. Then $\mathcal L^Z \otimes \mathcal L^{A_1} \otimes \mathcal L^{A_2}$ be a relative ample line bundle on $\mathcal E$ which has degree 9 with respect to $\Gamma_t$ for $t\neq 0$. Since $\Gamma'_0$ dominate $Z_0$ and $B_{j,0}$ for $j=1, 2$, we have $\Gamma'_0\cdot \mathcal L^{Z}\geq 1$ and $\Gamma'_0\cdot \mathcal L^{A_{j}} \geq 4$. Hence, $\Gamma'_0=\Gamma_0$ and $\Gamma_0$ determines birational map $\varepsilon_{j,0} \colon Z_0 \rightarrow B_{j,0}$ for $j=1, 2$. \\

In case when $X$ is $(B_m, \alpha_{m-1}, \alpha_{m})$, $m>2$,
\begin{eqnarray*}{\bf S}=\mathbb P \{ c v + v^2 \otimes w: c \in \mathbb C, v \in V, w \in W\} \simeq \mathbb P(\mathcal O(-1) \oplus \mathcal O(-2)^{m-1})\end{eqnarray*}
where $\dim V =2$ and $\dim W= m-1$. Since $ \nu_2 (\mathbb P^1) \subset \mathbb P \Sym^2 V \simeq  \mathbb P^2$ is rational normal curve with minimal degree, we conclude the result in a similar method.
\end{proof}

\section{$H^2$-cohomologies} \label{sect:H2 cohomology}


 We keep the same assumptions and notations as in Section \ref{sect:horospherical BF type}. We write $\mathbb C$ as $\frak z$ to emphasize that $\mathbb C$ commutes with $\mathfrak l$.

\subsection{Reductions}   \label{sect:computation of H2 cohomology}
A main ingredient is that the computation of $H^2(\frak m, \frak g)$ can be reduced to that of  $H^2(\frak l_-, \frak g)$ and $H^1(\frak l_-, \frak g)$, which can be computed  by applying Kostant's theory.

\begin{proposition} \label{prop: computation of H2 cohomology} Let $k \geq 1$.
 Let $\phi \in \Hom(\wedge^2 \frak m, \frak g)_k$ be such that $\partial \phi=0$.
 Then there is $\eta \in  \Hom(\frak m, \frak g)_k$ and  $\zeta \in \Hom(\wedge^2 \frak g_{-1}, U )_k$  such that
 $$\phi = \partial \eta + \zeta   $$
 and $\zeta|_{\wedge^2 \frak l_-} \in H^2(\frak l_-, U  )_k$ and $\zeta(Y^{-1}, \,\cdot\,)|_{\frak l_{-}} \in H^1(\frak l_-, U)_{k-1}+H^1(\frak l_-, \frak l)_{k-1}$ for any $Y^{-1} \in U_-$ and $\zeta|_{\wedge^2U_-} \in H^2(U_-, U)_1$.
 Furthermore, the choice of $\zeta$ in the  expression $\phi = \partial \eta + \zeta   $ is unique.
\end{proposition}

 We postpone the proof of  Proposition \ref{prop: computation of H2 cohomology} until Section \ref{sect: proof of proposition} and proceed with the computation of the  cohomology $H^2(\frak m, \frak g)$.

\begin{lemma} [Computation of $H^1(\frak l_-, \frak g)$] \label{lem:H1 vanishing semisimple} Let $k \geq 1$.

\begin{enumerate}
\item[(i)] $H^1(\frak l_-, \frak l)_{k-1}$ vanishes except for $k=1$ of type $(B_3, \alpha_{2}, \alpha_3 )$ and $$H^1(\frak l_-, \frak l)_{0} \subset \Hom(\frak l_{-1}, \frak l_{-1}).$$
\item[(ii)] $H^1(\frak l_-, \frak z)_{k-1}$ vanishes except for $k =2$ and $H^1(\frak l_-,\frak z)_1 \subset \Hom(\frak l_{-1}, \frak z)$.
\item[(iii)] $H^1(\frak l_-, U)_{k-1}$ vanishes except for $k =1$ and $$H^1(\frak l_-, U)_0\subset \Hom( \frak l_{-1}, U_{-1}).$$
\end{enumerate}

\end{lemma}

\begin{lemma} [Computation of $H^2(\frak l_-, \frak g)$] \label{lem:H2 vanishing semisimple}
 Let $k \geq 1$.

\begin{enumerate}

\item[(i)] $H^2(\frak l_-, \frak l)_k $ vanishes.

\item[(ii)] $H^2(\frak l_-, \frak z)_k  $ vanishes except for $k=2$, and
$$H^2(\frak l_-, \frak z)_2 \subset \wedge^2 \frak l_{-1}^* \otimes \frak z \subset (\wedge^2 \frak l_{-1} \otimes U_{-1})^*\otimes U_{-1} $$

\item[(iii)]
\begin{enumerate}
\item[(1)] If $(\frak m,\frak g_0)$ is of type $(B_m, \alpha_{m-1}, \alpha_m )$, then  $H^{ 2}(\frak l_-, U)_k$  vanishes except for $k=1,2$,  and we have
    \begin{eqnarray*}
    H^2(\frak l_-,  U)_1 &\subset& \wedge^2 \frak l_{-1}^* \otimes U_{-1} \\
    H^2(\frak l_-, U)_2 & \subset& \wedge^2 \frak l_{-1}^* \otimes U_0 \subset (\wedge ^2 \frak  l_{-1}  \otimes \mathfrak l_{-1})^* \otimes U_{-1} \\
    \end{eqnarray*}
\item[(2)] If $(\frak m,\frak g_0)$ is of type $(F_4, \alpha_2, \alpha_3 )$, then $H^{ 2}(\frak l_-, U)_k$   vanishes except for $k=1$, and we have
    \begin{eqnarray*}
    H^2(\frak l_-, U)_1 &\subset& \wedge^2 \frak l_{-1}^* \otimes U_{-1}
    \end{eqnarray*}
\end{enumerate}
\end{enumerate}
\end{lemma}
\begin{proof} Use Kostant's theory (\cite{Ko61}).

\end{proof}

We remark that, in the proof of Proposition \ref{prop: computation of H2 cohomology}, we use the property that both $H^1(\frak l_-, \frak l)_{k-1}$ and $H^2(\frak l_-, \frak l)_k$ vanish for any $k \geq 1$ (Lemma \ref{lem:H1 vanishing semisimple} (i) and Lemma \ref{lem:H2 vanishing semisimple} (i)).

\begin{proposition} [Computation of $H^2(\frak m, \frak g)$]  \label{prop:H2 vanishing}
 Let $k \geq 1$.  
\begin{enumerate}
\item If $(\frak m,\frak g_0)$ is of type $(B_3, \alpha_{2}, \alpha_3 )$,  then $H^{ 2}(\frak m, \frak g)_k$ vanishes except for $k=1,2$, and
    \begin{eqnarray*}
    H^2(\frak m, \frak g)_1 &\subset& \wedge^2 \frak g_{-1}^* \otimes \frak g_{-1} \\
    H^2(\frak m, \frak g)_2 &\subset& \wedge^2 \frak g_{-1}^* \otimes U_0  \subset (\wedge ^2 \frak g_{-1}  \otimes \mathfrak g_{-1})^* \otimes U_{-1}.
    \end{eqnarray*}
\item If $(\frak m,\frak g_0)$ is of type $(B_m, \alpha_{m-1}, \alpha_m )$, where $m >3$,  then $H^{ 2}(\frak m, \frak g)_k$ vanishes except for $k=1,2$, and
    \begin{eqnarray*}
    H^2(\frak m, \frak g)_1 &\subset& \wedge^2 \frak g_{-1}^* \otimes U_{-1} \\
    H^2(\frak m, \frak g)_2 &\subset& \wedge^2 \frak g_{-1}^* \otimes U_0  \subset (\wedge ^2 \frak g_{-1}  \otimes \mathfrak g_{-1})^* \otimes U_{-1}.
    \end{eqnarray*}
\item If $(\frak m,\frak g_0)$ is of type $(F_4, \alpha_2, \alpha_3 )$, then $H^{ 2}(\frak m, \frak g)_k$ vanishes except for $k=1 $, and
    \begin{eqnarray*}
    H^2(\frak m, \frak g)_1 &\subset& \wedge^2 \frak g_{-1}^* \otimes U_{-1}
    \end{eqnarray*}
\end{enumerate}

\end{proposition}

\begin{proof}
By Proposition \ref{prop: computation of H2 cohomology} and Lemma \ref{lem:H1 vanishing semisimple} and Lemma \ref{lem:H2 vanishing semisimple}, $H^{ 2}(\frak m, \frak g)_k$ vanishes for any $ k \geq 3$. By the uniqueness of $\zeta$ in the expression $\phi = \partial \eta + \zeta$ in Proposition \ref{prop: computation of H2 cohomology}, the nonvanishing $H^2(\frak m, \frak g)_k$ can be thought of as a subspace of the spaces in the right hand side.  In the first case, we can regards $\zeta_{\frak l}({Y^{-1}}, \cdot ) \in \Hom(\frak l_{-1},\frak l_{-1})$ because $\dim \frak l_{-2} =1$ implies that $\zeta_{\frak l}({Y^{-1}}, \cdot )\in \Hom(\frak l_{-2},\frak l_{-2})$ is just a constant multiple which comes from a boundary.
\end{proof}

 \subsection{Technical Lemmata} \label{sect:technical lemma}
We denote the restriction of $\partial$ to $\frak l_-$ by $\partial_0$, so that we have a subcomplex
 \begin{eqnarray*}
 0 \stackrel{\partial_0}{\longrightarrow} \frak g \stackrel{\partial_0}{\longrightarrow} \Hom( \frak l_-, \frak g ) \stackrel{\partial_0}{\longrightarrow} \Hom(\wedge^2 \frak l_-  ,\frak g   ) \stackrel{\partial_0}{\longrightarrow} \dots
 \end{eqnarray*}

 Recall that in the proof of the vanishing of $H^1(\frak m, \frak g)_k$ for positive $k$ (Proposition 48 of \cite{Ki}) a main difficulty is to show that  the nonvanishing  cohomology $H^1(\frak l_-, \mathfrak z)_1$ does not contribute to the cohomology $H^1(\frak m, \frak g)_1$, and a crucial Lemma is   the following.

\begin{lemma} [Lemma 27 (3) of \cite{Ki}] \label{lem: cohomology lemma 1 old version} For  $A \in U_0$, if the image of $\partial_0 A: \mathfrak l_{-1} \rightarrow U_{-1}$ has dimension $\leq 1$, then we have $\partial_0A=0$.
\end{lemma}

Under the assumption that $X$ is either of type $(B_m, \alpha_{m-1}, \alpha_m)$   or of type $(F_4, \alpha_2, \alpha_3)$, Lemma \ref{lem: cohomology lemma 1 old version}  can be improved.

\begin{lemma}   \label{lem: cohomology lemma 1} For  $A \in U_0$, if the image of $\partial_0 A: \mathfrak l_{-1} \rightarrow U_{-1}$ has dimension $\leq     2$, then we have $\partial_0A=0$.
\end{lemma}
Similarly, we have the following Lemma.

\begin{lemma} \label{lem: cohomology lemma 2}
For $A \in \Hom(\frak l_-, U)_1$, if the image of $\partial_0A: \wedge^2 \frak l_{-1} \rightarrow U_{-1}$ has dimension $\leq 1$, then we have $\partial_0 A=0$.

\end{lemma}

\begin{proof}
Let $\{x_{-\alpha}\}$ be a basis of $\frak l_{-1}$ consisting of root vectors.  We may assume that $[x_{-\alpha}. x_{-\beta}]=c_{\alpha \beta}x_{-\alpha-\beta}$, $c_{\alpha \beta} \in \mathbb C$,  form a basis of $\frak l_{-2}$.
 Let $\{u_{\mu}\}$ ($\{u_{\lambda}\}$, respectively) be a basis of $U_{-1}$ ($U_0$, respectively) consisting of weight vectors. We may assume that $[x_{-\alpha}, u_{\lambda}] = u_{-\alpha+\lambda}$ if $-\alpha + \lambda$ is a weight.\\

    For $A \in \Hom(\frak l_{-}, U)_1$, we have
 \begin{eqnarray*}
 A(x_{-\alpha}) &=& \sum_{\lambda} A_{\lambda, \alpha} u_{\lambda} \\
 A(x_{-\alpha -\beta})&=& \sum_{\mu}A_{\mu, \alpha+\beta} u_{\mu},
 \end{eqnarray*}
 for  some matrices $A_{\lambda, \alpha}$ and $A_{\mu, \alpha+\beta}$.
 Then
 \begin{eqnarray*}
 (\partial_0A)(x_{-\alpha}, x_{-\beta}) &=&[x_{-\alpha}, A(x_{-\beta})] - [x_{-\beta}, A(x_{-\alpha}) - A([x_{-\alpha}, x_{-\beta}]) \\
&=& [x_{-\alpha}, \sum_{\lambda} A_{\lambda, \beta}u_{\lambda}] -[x_{-\beta}, \sum_{\lambda}A_{\lambda, \alpha}u_{\lambda}] -  A(x_{-\alpha-\beta}c_{\alpha \beta})  \\
&=& \sum_{\lambda} A_{\lambda, \beta} u_{-\alpha+\lambda} - \sum_{\lambda} A_{\lambda, \alpha} u_{-\beta+\lambda}- \sum_{\mu} A_{\mu, \alpha+\beta}c_{\alpha \beta} u_{\mu} \\
&=&\sum_{\mu}  \left( A_{\mu+ \alpha, \beta} - A_{\mu+\beta, \alpha} + A_{\mu, \alpha+\beta} c_{\alpha \beta}  \right) u_{\mu}.
 \end{eqnarray*}
The condition that the image of $\partial_0A $ has dimension $\leq 1$ implies that for any choice of a pair $(x_{-\alpha}, x_{-\beta})$, $(\partial_0A)(x_{-\alpha}, x_{-\beta})$ is parallel to each other. \\

 \begin{enumerate}
 \item If $(\frak m, \frak g_0)$ is of type $(B_m, \alpha_{m-1}, \alpha_m)$, the action $ \frak l_1 \times U_{-1} \rightarrow U_0$ is given by
 \begin{eqnarray*}
 (\Sym ^2 V^* \otimes W^*) \times  V  &\rightarrow&  V^* \otimes W^* \\
(v_i^{*2} \otimes w^*, v_j) &\mapsto& \delta_{ij}v_i \otimes w^*
 \end{eqnarray*}
where $\{v_i : i=1,2\}$ is a basis of $V$ and $w$ is an element of $W$. Furthermore,  for $i=1,2$, the Lie bracket of any two element of $\{v_i^{*2} \otimes w^*: w \in W\}$ is zero. Write $v_1, v_2$ as $u_{\mu}, u_{\nu}$. Then
  $(\partial_0A)(x_{-\alpha}, x_{-\beta})$ is given by  $$\left( A_{\mu+ \alpha, \beta} - A_{\mu+\beta, \alpha} + A_{\mu, \alpha+\beta} c_{\alpha \beta}  \right) u_{\mu} + \left( A_{\nu+ \alpha, \beta} - A_{\nu+\beta, \alpha} + A_{\nu, \alpha+\beta} c_{\alpha \beta}  \right) u_{\nu}. $$
 If we take $x_{-\alpha}, x_{-\beta}$ in $\{v_1^2 \otimes w : w \in W\}$, then the coefficient of $u_{\mu}$ is zero. If we take $x_{-\alpha}, x_{-\beta}$ in $\{v_2^2 \otimes w : w \in W\}$, then the coefficient of $u_{\nu}$ is also zero.
 Since  $(\partial_0A)(x_{-\alpha}, x_{-\beta})$ is parallel to each other for any choice of a pair $(x_{-\alpha},x_{-\beta})$, both coefficients are zero for any $\alpha, \beta$. \\

 \item If $(\frak m, \frak g_0)$ is of type $(F_4, \alpha_2, \alpha_3)$, the action $ \frak l_1 \times U_{-1} \rightarrow U_0$ is given by
 \begin{eqnarray*}
 (\Sym^2 V^* \otimes W^*) \times  V &\rightarrow& V^* \otimes W^*\\
 ( {v_i^*}\circ{v_j^*} \otimes w^* , v_k ) &\mapsto & (\delta_{ik}v_j^* + \delta_{jk}v_i^*) \otimes w^*.
  \end{eqnarray*}
In particular, $({v_i^*}^2 \otimes w^*, v_j )$ maps to $\delta_{ij}v_i^* \otimes w^*$, where $\{v_i : i=1,2,3\}$ is a basis of $V$ and $w$ is an element of $W$.  Furthermore, for  $i=1,2,3$,  the Lie bracket of any two elements in $\{ {v_i^*}^2 \otimes w: w \in W\}$ is zero.

Write $v_1, v_2,v_3$ as $u_{\mu}, u_{\nu}, u_{\xi}$. Then
  $(\partial_0A)(x_{-\alpha}, x_{-\beta})$ is given by
  \begin{eqnarray*}
 && \left( A_{\mu+ \alpha, \beta} - A_{\mu+\beta, \alpha} + A_{\mu, \alpha+\beta} c_{\alpha \beta}  \right) u_{\mu} + \left( A_{\nu+ \alpha, \beta} - A_{\nu+\beta, \alpha} + A_{\nu, \alpha+\beta} c_{\alpha \beta}  \right) u_{\nu} \\
 && + \left( A_{\xi+ \alpha, \beta} - A_{\xi+\beta, \alpha} + A_{\xi, \alpha+\beta} c_{\alpha \beta}  \right) u_{\xi}
 \end{eqnarray*}
 By the same arguments as in (1), we get that all coefficients are zero.
 \end{enumerate}
\end{proof}

\subsection{A proof of Proposition \ref{prop: computation of H2 cohomology}}   \label{sect: proof of proposition}
We list up properties of $\frak g = U + \frak z + \frak l$.

\begin{enumerate}
\item [(${\bf P_0}$)] $[\frak l, \frak z]=0$ and $[\frak l_-, U_-]=0$
\item [(${\bf P_1}$)] For $X \in \frak l_{\geq - \mu+1} +U_{\geq 0} $, if $[\frak l_-, X]=0$, then $X=0$. \\
\end{enumerate}

\begin{proof} [Proof of Proposition \ref{prop: computation of H2 cohomology}]
 Let $\phi \in \Hom(\wedge^2 \frak m, \frak g)_k$ be such that $\partial \phi=0$, where $k \geq 1$. We will show that there exist   $\eta \in \Hom(\frak m, \frak g)_k$, $\zeta=\zeta_{\frak l}+\zeta_U$ with $\zeta_{\frak l} \in  H^2(\frak l_-, U)_k$ and $\zeta_U \in H^2(U_-, U)_1$, and $\xi=\xi_{\frak l} + \xi_U $ with $\xi_U(X^{-1}, \,\cdot\,) \in H^1(\frak l_-, U)_{k-1}$ and $\xi_{\frak l}(X^{-1}, \,\cdot\,) \in H^1(\frak l_-, \frak l)_{k-1}$ for $X^{-1} \in U_-$ satisfying that  $$\phi = \partial \eta + \zeta +\xi .$$ In other words, for any $X^{-1}, Y^{-1} \in U_-$ and $X^0, Y^0 \in \frak l_-$, we have
 \begin{eqnarray*}
 \phi(X^{-1} + X^{0}, Y^{-1}+ Y^0)   &=&  [X^{-1}+X^0, \eta(Y^{-1}+ Y^0)] -[Y^{-1}+Y^0, \eta(X^{-1} + X^0)]\\
 &&- \eta([X^0, Y^0])  + \zeta(X^0, Y^0) +\xi(X^0, Y^{-1})\\ && + \xi(X^{-1}, Y^0) + \zeta(X^{-1}, Y^{-1}),
 \end{eqnarray*}
 or equivalently,
 \begin{eqnarray*}
 \phi_{\frak l+\frak z}(X^{-1} + X^0, Y^{-1}+Y^0)
   &=& [X^0, \eta_{\frak l+\frak z}(Y^{-1} +Y^0)] - [Y^0, \eta_{\frak l+\frak z}(X^{-1} + X^0)] \\
  &&- \eta_{\frak l+\frak z}([X^0, Y^0])+\xi_{\frak l}(X^0, Y^{-1}) + \xi_{\frak l}(X^{-1}, Y^0) \\
   \phi_U(X^{-1}+X^0, Y^{-1}+Y^0)
 &=& [X^{-1}, \eta_{\frak l + \frak z}(Y^{-1} +Y^0)]-[Y^{-1}, \eta_{\frak l + \frak z}(X^{-1}+X^0)] \\
 &&  + [X^0, \eta_U(Y^{-1} + Y^0)] -[Y^0, \eta_U(X^{-1} + X^0)]   \\
 && - \eta_U([X^0, Y^0]) + \zeta_U(X^0, Y^0)+\xi_U(X^0, Y^{-1}) \\ && + \xi_U(X^{-1}, Y^0) + \zeta_U(X^{-1}, Y^{-1}).
 \end{eqnarray*}

\vskip 10 pt
 To do this, we decompose the identity
   $$\partial \phi(X^{-1}+ X^0, Y^{-1} + Y^0, Z^{-1} +Z^0)=0$$ into the sum of two identities: \\


\noindent (I) $[X^0, \phi_{\frak l+ \frak z}(Y^{-1}+Y^0, Z^{-1}+Z^0)] -[Y^0, \phi_{\frak l+ \frak z}(X^{-1}+X^0 , Z^{-1} +Z^0)]
+ [Z^0, \phi_{\frak l+\frak z}(X^{-1}+ X^0, Y^{-1} +Y^0)] -\phi_{\frak l+ \frak z}([X^0, Y^0],Z^{-1}+ Z^0) + \phi_{\frak l + \frak z}([X^0,Z^0],Y^{-1}+Y^0) - \phi_{\frak l + \frak z}([Y^0, Z^0], X^{-1} +X^0) =0$. \vskip 10 pt

\noindent (II) $[X^0, \phi_U(Y^{-1}+Y^0,Z^{-1}+Z^0)] - [Y^0, \phi_U(X^{-1}+X^0, Z^{-1}+Z^0)] +[Z^0, \phi_U(X^{-1}+X^0, Y^{-1}+Y^0)] +[X^{-1}, \phi_{\frak l + \frak z}(Y^{-1}+Y^0, Z^{-1}+Z^0)] -[Y^{-1}, \phi_{\frak l +\frak z}(X^{-1}+X^0, Z^{-1}+Z^0)] +[Z^{-1}, \phi_{\frak l + \frak z}(X^{-1}+X^0, Y^{-1}+Y^0)]  -\phi_U([X^0, Y^0],Z^{-1}+Z^0) + \phi_U([X^0, Z^0], Y^{-1}+Y0) - \phi_U([Y^0, Z^0], X^{-1}+X^0)=0$.
\vskip 10 pt

 \noindent {\bf  R1.} There exist    $\eta_{\frak l+\frak z}  +\eta_{U} \in \Hom(\frak l_-, \frak g)_k $ and $\zeta_{\frak l+ \frak z}  + \zeta_U \in H^2(\frak l_-, \frak g)_k$ such that
 \begin{eqnarray*}
 \phi_{\frak l+\frak z}(X^0, Y^0) &= & [X^0, \eta_{\frak l+ \frak z}(Y^0)] -[Y^0, \eta_{\frak l+\frak z}(X^0)] - \eta_{\frak l+\frak z}([X^0, Y^0]) + \zeta_{\frak l+\frak z}(X^0, Y^0)   \\
 \phi_{U}(X^0,Y^0) &=&  [X^0, \eta_U(Y^0)] -[Y^0, \eta_U(X^0)] -\eta_U([X^0,Y^0])+ \zeta_{U}(X^0, Y^0)
 \end{eqnarray*}
 Indeed, putting $X^{-1}=Y^{-1}=Z^{-1}=0$ into (I) and (II), we get $\partial_0(\phi|_{\frak l_-})=0$.

Note that    $\zeta_{\frak l+ \frak z} = \zeta_{\frak z}$ if $k=2$, and  $\zeta_{\frak l+ \frak z}=0$, otherwise,  by Lemma \ref{lem:H2 vanishing semisimple} (i) and (ii).  \\

\noindent {\bf  R2.} We have  $\phi_{\frak l_{r}}(Y^{-1}, Z^{-1})=0$ for $r \geq - \mu +1$ and $r \neq 0$. Indeed,
 putting $Y^0=Z^0=0$ to (I) for $k\neq 2$, we get $[X^0, \phi_{\frak l_r}(Y^{-1},Z^{-1}) ]=0$ for $r\neq 0$. The desired vanishing follows from (${\bf P_1}$).
 \\

  \noindent {\bf  R3.} We may extend $\eta_{\frak l + \frak z}$ from $\frak l_-$ to $\frak m=\frak l_- + U_-$ 
   so that
  $$\phi_{\frak l+\frak z}(Y^{-1}, \,\cdot\,) = [\eta_{\frak l+\frak z }(Y^{-1}),\,\cdot\,] -\xi_{\frak l+\frak z}^{Y^{-1}}(\,\cdot\,) \text{ in }\Hom(\frak l_-, \frak l+ \frak z)_{k-1}$$
  for some $\xi_{\frak l+\frak z}^{Y^{-1}} \in H^1(\frak l_-, \frak l+\frak z)_{k-1}$.
  Indeed, putting $X^{-1} =Y^0 =Z^{-1}=0$ to (I) we get
  $$[X^0, \phi_{\frak l+ \frak z}(Y^{-1},Z^0)] - [Z^0, \phi_{\frak l + \frak z}(Y^{-1}, X^0)] - \phi_{\frak l+ \frak z}(Y^{-1}, [X^0, Z^0]) =0,$$
  or equivalently,
  \begin{eqnarray*}
  \partial_0(\phi_{\frak l+\frak z}(Y^{-1}, \,\cdot\,))=0.  
  \end{eqnarray*}
 Thus
  there exit $\eta_{\frak l + \frak z}^{Y^{-1}} \in (\frak l + \frak z)_{k-1}$ and   $\xi_{\frak l+\frak z}^{Y^{-1}}(\,\cdot\,) \in H^1(\frak l_-, \frak l+\frak z)_{k-1}$ such that
 $$\phi_{\frak l+\frak z}(Y^{-1}, \,\cdot\,) = [\eta_{\frak l+\frak z }^{Y^{-1}},\,\cdot\,] -\xi_{\frak l+\frak z}^{Y^{-1}}(\,\cdot\,) .$$
  Note that $\xi_{\frak l+ \frak z}^{Y^{-1}} =\xi_{\frak l}^{Y^{-1}}$ if $k=1$ (for $B_3$ type), $\xi_{\frak l+ \frak z}^{Y^{-1}} =\xi_{\frak z}^{Y^{-1}}$ if $k=2$ and $\xi_{\frak l+\frak z}^{Y^{-1}}=0$ otherwise, by Lemma \ref{lem:H1 vanishing semisimple} (i) and (ii). 
Define $\eta_{\frak l + \frak z}(Y^{-1})$ by $\eta_{\frak l + \frak z}(Y^{-1}):= \eta_{\frak l+ \frak z}^{Y^{-1}}$.     \\

\noindent {\bf R4.} We have  $\phi_{\frak l_0 + \frak z}(Y^{-1}, Z^{-1}) =0 $. Indeed,  putting $Y^0=Z^0=0$ to (I) for $k=2$, we get $[X^0, \phi_{\frak l_0 +\frak z}(Y^{-1},Z^{-1}) ]=0$ and putting $X^0=Y^0=Z^0=0$ to (II) for $k=2$ we get
  \begin{eqnarray*}
  [X^{-1}, \phi_{\frak l_0 + \frak z}(Y^{-1}, Z^{-1})] - [Y^{-1}, \phi_{\frak l_0+ \frak z}(X^{-1}, Z^{-1})] + [Z^{-1}, \phi_{\frak l_0+ \frak z}(X^{-1}, Y^{-1})] =0.
  \end{eqnarray*}
  Take independent vectors  $X^{-1}, Y^{-1}, Z^{-1}$  in $U_-$ and arbitrary $X^0\in \frak l_-$ to get the desired result. \\

  \noindent {\bf R5.} We may extend $\zeta_U$ from $\frak l_-$ to $\frak m = \frak l_-+ U_-$ such that $\zeta|_{\wedge^2U_-} \in H^2(U_-, U)_1$, satisfying
  $$\phi_U(Y^{-1}, Z^{-1}) =[Y^{-1}, \eta_{\frak l+ \frak z}(Z^{-1})] - [Z^{-1}, \eta_{\frak l + \frak z}(Y^{-1})] +\zeta_U(Y^{-1}, Z^{-1})$$
   and
   $$\phi_{\frak z}(X^0, Z^{-1})
   = 0.$$

Indeed, putting $X^{-1}=Y^0=Z^0=0$ to (II) we get
  \begin{eqnarray*}
  [X^0, \phi_U(Y^{-1}, Z^{-1})]
     - [Y^{-1}, \phi_{\frak l+ \frak z}( X^{0}, Z^{-1})] + [Z^{-1}, \phi_{\frak l+ \frak z}( X^0, Y^{-1})] =0.
  \end{eqnarray*}
Thus
  \begin{eqnarray*}
  &&[X^0, \phi_U(Y^{-1}, Z^{-1})] \\
   &\stackrel{\bf R3.}{=}& [Y^{-1},[X^0, \eta_{\frak l + \frak z}(Z^{-1})] +\xi_{\frak l+\frak z}^{Z^{-1}}(X^0)] - [Z^{-1},[X ^0, \eta_{\frak l + \frak z}(Y^{-1})] +\xi_{\frak l+\frak z}^{Y^{-1}}(X^0)]  \\
   &\stackrel{(\bf P_0)}{=}&  [X^0,[Y^{-1}, \eta_{\frak l + \frak z}(Z^{-1})]]  - [X^0,[ Z^{-1}, \eta_{\frak l + \frak z}(Y^{-1})]]
     +[Y^{-1}, \xi_{\frak l+\frak z}^{Z^{-1}}(X^0)]  -[Z^{-1}, \xi_{\frak l+\frak z}^{Y^{-1}}(X^0)].
  \end{eqnarray*}
  Hence
  \begin{eqnarray*}
   &&[X^0, \phi_U(Y^{-1}, Z^{-1})]
   -[X^0,[Y^{-1}, \eta_{\frak l + \frak z}(Z^{-1})]]  + [X^0,[ Z^{-1}, \eta_{\frak l + \frak z}(Y^{-1})]] \\
   &&=  [Y^{-1}, \xi_{\frak l+\frak z}^{Z^{-1}}(X^0)]  -[Z^{-1}, \xi_{\frak l+\frak z}^{Y^{-1}}(X^0)].
    \end{eqnarray*}
    Therefore, we have
    \begin{eqnarray*}
    \partial_0\left( \phi_U(Y^{-1}, Z^{-1}) -[Y^{-1}, \eta_{\frak l+ \frak z}(Z^{-1})] + [Z^{-1}, \eta_{\frak l + \frak z}(Y^{-1})]\right)
     =[Y^{-1}, \xi_{ \frak l+\frak z}^{Z^{-1}}(\,\cdot\,)]  -[Z^{-1}, \xi_{\frak l+ \frak z}^{Y^{-1}}(\,\cdot\,)].
    \end{eqnarray*}
    We claim that both sides vanish.

   By Lemma \ref{lem:H1 vanishing semisimple} (i) and (ii), $\xi_{\frak l+ \frak z}^{Y^{-1}}$ is $\xi_{\frak l}^{Y^{-1}}$ or zero if $k=1$, $\xi_{\frak l+ \frak z}^{Y^{-1}} $ is  $\xi_{\frak z}^{Y^{-1}}$ if $k=2$, and is zero otherwise.
   In the last case,
   $$\partial_0\left( \phi_U(Y^{-1}, Z^{-1}) -[Y^{-1}, \eta_{\frak l+ \frak z}(Z^{-1})] + [Z^{-1}, \eta_{\frak l + \frak z}(Y^{-1})]\right) =0.$$

   In the second case,
   $$A^{Y^{-1},Z^{-1}}:=\phi_U(Y^{-1}, Z^{-1}) -[Y^{-1}, \eta_{\frak l+ \frak z}(Z^{-1})] + [Z^{-1}, \eta_{\frak l + \frak z}(Y^{-1})]$$
   is an element of $U_0$ and we have
    $$\partial_0 \,A^{Y^{-1},Z^{-1}} =[Y^{-1}, \xi_{ \frak z}^{Z^{-1}}(\,\cdot\,)]  -[Z^{-1}, \xi_{ \frak z}^{Y^{-1}}(\,\cdot\,)].$$
   Both sides vanish by Lemma \ref{lem: cohomology lemma 1} because $[Y^{-1},\mathfrak z] -[Z^{-1}, \mathfrak z]$ has dimension two.
    In the first case, it is clear that both sides vanish because $[U_, \, \frak g_-]=0$.

   Therefore, we have
   $$\phi_U(Y^{-1}, Z^{-1}) =[Y^{-1}, \eta_{\frak l+ \frak z}(Z^{-1})] - [Z^{-1}, \eta_{\frak l + \frak z}(Y^{-1})]+\zeta_U(Y^{-1}, Z^{-1}) \text{ by } {\rm(}{\bf P_1}{\rm)} $$
   and
   $$\phi_{\frak z}(X^0, Z^{-1}) = \xi_{\frak z}^{Z^{-1}}(X^0) = 0.$$
 Note that $\zeta_U(Y^{-1}, Z^{-1}) \in U_-$.\\

   \noindent  {\bf R6.} We may extend $\eta_U$ from $\frak l_-$ to $\frak m=\frak l_- + U_-$ so that   
  $$\phi_U(Y^{-1},\,\cdot\,) - [Y^{-1}, \eta_{\frak l+ \frak z}(\,\cdot\,)] =-[\,\cdot\,, \eta_U(Y^{-1})] \text{ in } \Hom(\frak l_-, U)_{k-1}, $$
   and $$\zeta_{\frak z} (X^0,\,Y^0)=0.$$
  Indeed, putting $X^{-1}=Y^0=Z^{-1}=0$ we get
   \begin{eqnarray*}
   &&[X^0, \phi_U(Y^{-1}, Z^0) + [Z^0, \phi_U(X^0, Y^{-1})] -[Y^{-1}, \phi_{\frak l+ \frak z}(X^0), Z^0)] \\
   && + \phi_U([X^0, Z^0], Y^{-1}) =0.
   \end{eqnarray*}
   Thus
   \begin{eqnarray*}
   &&[X^0, \phi_U(Y^{-1},Z^0)] + [Z^0, \phi_U(X^0, Y^{-1})] + \phi_U([X^0, Z^0],Y^{-1}) \\
   &=& [Y^{-1}, \phi_{\frak l + \frak z}(X^0, Z^0)] \\
   &\stackrel{\bf R1.}{=}&[Y^{-1}, [X^0, \eta_{\frak l + \frak z}(Z^0)] - [Z^0, \eta_{\frak l + \frak z}(X^0)]-\eta_{\frak l+ \frak z}([X^0, Z^0]) + \zeta_{\frak l+ \frak z}(X^0, Z^0) ] \\
   &\stackrel{({\bf P_0})}{=}& [X^0, [Y^{-1}, \eta_{\frak l + \frak z}(Y^{0})]] - [Y^{-1}, [Z^0, \eta_{\frak l + \frak z}(Y^0)]] -[Y^{-1}, \eta_{\frak l + \frak z}([X^0, Z^0])] \\
   && + [Y^{-1}, \zeta_{\frak l+ \frak z}(X^0, Z^0)  ]
   \end{eqnarray*}
   Therefore, we have
   \begin{eqnarray*}
   \partial_0(\phi_U(Y^{-1},\,\cdot\,) -[Y^{-1}, \eta_{\frak l + \frak z}(\,\cdot\,)] ) (\,\cdot\,,\,\cdot\,) =[Y^{-1}, \zeta_{\frak l+ \frak z}(\,\cdot\,,\,\cdot\,)] .
   \end{eqnarray*}
    We claim that both  sides vanish.

   By Lemma \ref{lem:H2 vanishing semisimple} (i) and (ii),   $\zeta_{\frak l+ \frak z}$  is $\zeta_{\frak z} $ if $k=2$, and is zero, otherwise.
   In the second case,
   $$\partial_0(\phi_U(Y^{-1},\,\cdot\,) -[Y^{-1}, \eta_{\frak l + \frak z}(\,\cdot\,)] ) (\,\cdot\,,\,\cdot\,) = 0.$$
    In the first case,  $$A^{Y^{-1}}:=\phi_U(Y^{-1},\,\cdot\,) -[Y^{-1}, \eta_{\frak l + \frak z}(\,\cdot\,)] $$ is an element of $\Hom(\frak l_-, U)_1$, and
   we have
    \begin{eqnarray*}
   \partial_0 A^{Y^{-1}} =[Y^{-1}, \zeta_{\frak z}(\,\cdot\,,\,\cdot\,)].
   \end{eqnarray*}
   Both sides vanish by Lemma \ref{lem: cohomology lemma 2} because $[Y^{-1},\frak z]$ has dimension one.

   By Lemma \ref{lem:H1 vanishing semisimple} 
   there exist    $\eta_U^{Y^{-1}} \in U_{k-1}$ and $\xi_U^{Y^{-1}} \in H^1(\frak l_-, U)_{k-1}$ such that
  $$\phi_U(Y^{-1},\,\cdot\,) - [Y^{-1}, \eta_{\frak l+ \frak z}(\,\cdot\,)] =-[\,\cdot\,, \eta_U^{Y^{-1}}] +\xi_U^{Y^{-1}}(\,\cdot\,) \text{ in } \Hom(\frak l_-, U)_{k-1}, $$
  and $$\zeta_{\frak z} (\,\cdot\,,\,\cdot\,)=0.$$
  Define $\eta_U(Y^{-1})$ by $\eta_U(Y^{-1}):=\eta_U^{Y^{-1}}$ and define $\xi_U \in \Hom (\frak l_{-1} \wedge U_{-1}, U_{-1})$ by $\xi_U(Y^{-1}, X^0)=-\xi_U(X^0, Y^{-1}):=\xi_U^{Y^{-1}}(X^0)$. \\

Consequently, there exist $\eta = \eta_{\frak l+ \frak z} + \eta_U \in \Hom(\frak m, \frak g)_k$, $\zeta =  \zeta_U \in H^2(\frak l_-, U)_k+H^2(U_-, U)_1$ and $\xi_U+ \xi_{\frak l} \in (\frak l_{-1} \wedge U_{-1})^* \otimes (U_{-1}+ \frak l_{-1})$ satisfying
\begin{eqnarray*}
\phi_{\frak l+\frak z}(X^0, Y^0) &\stackrel{{\bf R1.}, {\bf R6}}{=}& [X^0, \eta_{\frak l+\frak z}(Y^0)] -[Y^0,\eta_{\frak l+\frak z}(X^0)] - \eta_{\frak l+\frak z}([X^0,Y^0])\\
\phi_{\frak l+\frak z}(X^0, Y^{-1})&\stackrel{{\bf R3.},{\bf R5.}}{=}&[X^0, \eta_{\frak l}(Y^{-1})]+\xi_{\frak l}(X^0, Y^{-1}) \\
  \phi_{\frak l_{\geq -\mu+1}+\frak z}(X^{-1},Y^{-1}) &\stackrel{{\bf R2.}, {\bf R4.}}{=}&0 \\
  \phi_U(X^0, Y^0)&\stackrel{\bf R1.}{=}& [X^0, \eta_U(Y^0)] -[Y^0,\eta_U(X^0)] - \eta_U([X^0, Y^0]) + \zeta_U(X^0,Y^0) \\
  \phi_U(X^0, Y^{-1}) &\stackrel{\bf R6.}{=}& -[Y^{-1}, \eta_{\frak l+\frak z}(X^0)] +[X^0, \eta_U(Y^{-1})] +\xi_U (X^0, Y^{-1}) \\
  \phi_U(X^{-1}, Y^{-1}) &\stackrel{\bf R5.}{=}& [X^{-1}, \eta_{\frak l+ \frak z}(Y^{-1})] -[Y^{-1},\eta_{\frak l+\frak z}(X^{-1})]+\zeta_U(Y^{-1}, Z^{-1}).
\end{eqnarray*}

  It remains to show that the choice of $\zeta$ in the  expression $\phi = \partial \eta + \zeta   $ is unique. Indeed, if $\partial \eta + \zeta   =\partial \eta' + \zeta'   $, then $\partial(\eta-\eta') = \zeta' - \zeta$. It suffices to show that if $\partial \eta = \zeta$, then   $\zeta$ is zero.
From
$$(\partial \eta)(X^{-1}+ X^0, Y^{-1}+ Y^0) = \zeta(X^{-1}+X^0, Y^{-1}+Y^0)$$
it follows that
\begin{eqnarray*}
\zeta(X^0, Y^0) &=&[X^0, \eta(Y^0)]-[Y^0, \eta(X^0)] -\eta([X^0, Y^0]) \\
\zeta(X^0, Y^{-1})&=& [X^0, \eta(Y^{-1})] - [Y^{-1}, \eta(X^0)].
\end{eqnarray*}
The first identity implies that $\zeta = \partial_0 \eta|_{\frak l_-}$. Since $\zeta|_{\wedge^2 \frak l_-} \in H^2(\frak l_-, U )_k$, $\zeta|_{\wedge^2 \frak l_-}$ vanishes and so does $\partial_0 \eta|_{\frak l_-}$. Since $H^1(\frak l_-, \frak g)_k=0$ for any $k\geq 1$, there is $\chi$ such that $\eta|_{\frak l_-} =\partial_0\chi$. Then $\eta(X^0) =[\chi, X^0]$ for $X^0 \in \frak l_-$.

The second identity becomes
\begin{eqnarray*}
\zeta(X^0, Y^{-1}) &=& [X^0, \eta(Y^{-1})] -[Y^{-1}, [\chi, X^0]] \\
&=& [X^0, \eta(Y^{-1}) + [Y^{-1}, \chi]] \text{ because } [\frak l_-,   U_-]=0
\end{eqnarray*}
Since $\zeta(\,\cdot\,, Y^{-1}) |_{\frak l_-} \in H^1(\frak l_-, U)_{k-1}$, we have  $\zeta(\,\cdot\,, Y^{-1}) |_{\frak l_-} =0$.

This completes the proof of Proposition \ref{prop: computation of H2 cohomology}.
  \end{proof}

\section{Local equivalence of geometric structures} \label{sect:local equivalence}

 We keep the same assumptions and notations as in Section \ref{sect:horospherical BF type}.

\begin{proposition} \label{extension to rationally connected open subset}
Let $M$ be a Fano manifold of Picard number one and $\mathcal C_x(M) \subset \mathbb PT_xM$ be the variety of minimal rational tangents at a general point $x \in X$ associated with a minimal dominant rational component $\mathcal K$. Suppose that $\mathcal C_x(M) \subset \mathbb PT_xM$ is projectively equivalent to ${\bf S} \subset \mathbb P \frak m$ for general $x \in M$. Denote by $D \subset TM$   the distribution on $M$ obtained by the linear span of the affine cone of $\mathcal C_x(M)$ in $T_xM$.
Then there is a Zariski open subset $M^0\subset M$ such that a general member  of $\mathcal K$ lies on $M^0$, satisfying  the following properties:
\begin{enumerate}
\item $D|_{M^0}$ is of type $\frak m$;
\item $\mathcal C|_{M^0}\subset \mathbb P D|_{M^0}$ defines an ${\bf S}$-structure on $(M^0, D|_{M^0})$ and there corresponds to a $G_0$-structure on $(M^0, D|_{M^0})$, where $G_0=G(\widehat{\bf S})$.
\end{enumerate}

\end{proposition}

For the definition of ${\bf S}$-structures and $G_0$-structures, see Section \ref{sect:G0 and S structures}.

\begin{proof}

The distribution $D$ is holomorphic outside its singular set $Sing(D) \subset M$ of codimension $\geq 2$.
 Let $M' \subset M-Sing(D)$ be a  Zariski open subset such that for all $x \in M'$ satisfying $\mathcal C_x \subset \mathbb P D_x$ is projectively equivalent to $ {\bf S} \subset \mathbb P {\frak g_{-1}}$.

Let $C $ be a standard minimal rational curve represented by $f:\mathbb P^1 \rightarrow M$ with $C \cap M' \not = \emptyset$ and $C \not\subset bad(\mathcal K)$ and  $C \cap Sing(D) =\emptyset$. Then for a generic point $y \in C$, $\mathcal C_y \subset \mathbb PD_y$ is projectively equivalent to $\bf S\subset \mathbb P \frak g_{-1}$. By Proposition \ref{parallel transport of II} and Proposition \ref{parallel transport of III} and Lemma  \ref{parallel transport of III-X case}, the relative second fundamental form and the relative third fundamental form of $\mathcal C(M)$ along the lifting $C^{\sharp}$ of $C$ is constant, and we have 
\begin{eqnarray*}
f^*D &=& \mathcal O(2) \oplus \mathcal O(1)^p \oplus \mathcal O^r \oplus \mathcal O(-1)^s \\
f^*(TM/D) &=& \mathcal O(1)^s \oplus \mathcal O^t.
\end{eqnarray*}
For the values $p,q,r,s,t$, see Remark \ref{pqrst}.
Furthermore,
the pull-back $(f^{\sharp})^*\widehat{T}^{\varpi} $ of the relative affine tangent bundle $ \widehat{T}^{\varpi} $ of $\mathcal C \subset \mathbb P(TM)$ is the positive part $ P:=\mathcal O(2) \oplus \mathcal O(1)^p$ of $f^*D$, and
   the pull-back $(f^{\sharp})^*\widehat{T}^{(2), \varpi} $ of the relative  second osculating affine   bundle   $ \widehat{T}^{(2), \varpi} $ of $\mathcal C \subset \mathbb P(TM)$  is the subbundle  $P^{(2)} :=\mathcal O(2) \oplus \mathcal O(1)^p \oplus \mathcal O^r$ of $f^*D$. 

 By Proposition \ref{parallel transport of vmrt}, for any $y\in C$, $\mathcal C_y \subset \mathbb PD_y$ is projectively equivalent to $\bf S\subset \mathbb P \frak g_{-1}$.

Let $D^{-1} :=D$ and $D^{-2} := D + [D,D]$. For generic $x \in M$, the Frobenius bracket $ [\,,\,]$ is given by
\begin{eqnarray*}
    [\,,\,]_1&:&\wedge^2 D_x^{-1} \rightarrow D_x^{-2}/D_x \subset T_xM/D_x \\ [1 pt]
    [\,,\,]_2&:&D_x^{-1} \wedge D_x^{-2} \rightarrow (T_xM/D_x)/(D_x^{-2}/D_x)= T_xM/D_x^{-2}.
    \end{eqnarray*}
In particular, if $X$ is $(B_m, \alpha_{m-1}, \alpha_m)$, then $D^{-2}=TM$ and simply the bracket is
    \begin{eqnarray*} [\,,\,]:\wedge^2 D_x \rightarrow T_xM/D_x.
    \end{eqnarray*}

For generic $v \in \widehat {\mathcal C}_x$ and $w \in T_v(\widehat{\mathcal C}_x)$, we have $[v,w]=0$ by Proposition \ref{kernel of frobenius braket}. Hence, by Proposition \ref{fundamental graded Lie algebra of vmrt} (1) and its proof, the symbol algebra $\Symb_x(D)$ is isomorphic to $\mathfrak m$ at general $x$. It remains to show that $\Symb_y(D)$ is isomorphic to $\mathfrak m$ for any $y \in C$. For this, we need to know the decompositions of vector bundles related to the Frobenius bracket. From now on,  we use notations $D|_C$, $TM/D|_C$, etc., instead of $f^*D$, $f^*(TM/D)$, for simplicity.

Recall that
\begin{enumerate}
\item  when $X$ is $(B_m, \alpha_{m-1}, \alpha_m)$,
\begin{enumerate}
\item $\dim \frak g_{-1} = 3m-1$, $\dim \widehat{\bf S} =m+1$, $\dim \frak g_{-2} = (m-1)(m-2)/2$
\item $D|_C=\mathcal O(2) \oplus \mathcal O(1)^m \oplus \mathcal O^m \oplus \mathcal O(-1)^{m-2}$, $TM/D|_C = \mathcal O(1)^{m-2} \oplus \mathcal O^t$, where $t=(m-2)(m-3)/2$;
\end{enumerate} 

\item when $X$ is $(F_4, \alpha_2, \alpha_3)$,
\begin{enumerate}
\item $\dim \frak g_{-1} = 15$, $\dim \widehat{\bf S} =5$, $\dim \frak g_{-2} = 6$, $\dim \frak g_{-3} = 2$
\item $D|_C=\mathcal O(2) \oplus \mathcal O(1)^4 \oplus \mathcal O^7 \oplus \mathcal O(-1)^{3}$, $TM/D|_C = \mathcal O(1)^{3} \oplus \mathcal O^5$.
\end{enumerate}

\end{enumerate}

\medskip
We need the following Lemma to complete the proof of Proposition \ref{extension to rationally connected open subset}.

\begin{lemma}   \label{lem:splitting type of V}
Let $\mathcal V$ and $\mathcal W $ be vector bundles on $M'$ associated with $V$ and $W$.
Let $C$ be a general member of $\mathcal K$ passing through $x \in M'$.
\begin{enumerate}
 \item  When $X$ is $(B_m, \alpha_{m-1}, \alpha_m)$, we have
$$\mathcal V|_C= \mathcal O(1) \oplus \mathcal O \text{ and } \mathcal W|_C= \mathcal O \oplus \mathcal O(-1)^{m-2}.  $$

 \item  When $X$ is $(F_4, \alpha_2, \alpha_3)$,
 we have
$$\mathcal V|_C= \mathcal O(1) \oplus \mathcal O^2 \text{ and } \mathcal W|_C= \mathcal O \oplus \mathcal O(-1)  .$$

\end{enumerate}

\end{lemma}

Set
\begin{eqnarray*}
\mathcal E &:=&(\Sym^4\mathcal V)^{\perp} \otimes \wedge^2 \mathcal W = \Sym^2(\wedge^2 \mathcal V) \otimes \wedge^2 \mathcal W.
\end{eqnarray*}
By Lemma \ref{lem:splitting type of V}, we have
\begin{eqnarray*}
\mathcal E|_C
&=& \left\{\begin{array}{l}
 \mathcal O(2) \otimes (\mathcal O(-1)^{m-2} \oplus \mathcal O(-2)^t)   \text{ when } X \text{ is } (B_m, \alpha_{m-1}, \alpha_m) \\[5 pt]
\Sym^2 (\mathcal O(1)^2 \oplus \mathcal O) \otimes \mathcal O(-1)   \text{ when } X \text{ is } (F_4, \alpha_2, \alpha_3).
\end{array}\right.
\end{eqnarray*}
The Frobenius bracket $[\,,\,]:\wedge^2D|_C \rightarrow TM/D|_C$ induces a map
$$ \mathcal E |_C \rightarrow TM/D|_C, $$
which is an isomorphism onto its image for generic $y \in C$ by Proposition \ref{fundamental graded Lie algebra of vmrt} (1).

When $X$ is $(B_m, \alpha_{m-1}, \alpha_m)$, since $\mathcal E|_C$ and $TM/D|_C$ have the same splitting type, $\mathcal E|_C \rightarrow TM/D|_C$ is an isomorphism at any point of $C$.

When $X$ is $(F_4, \alpha_2, \alpha_3)$, $\mathcal E|_C \rightarrow TM/D|_C$ is an isomorphism onto $D^{-2}/D|_C$ at any point of $C$.
Now we have
$$0 \rightarrow D^{-2}/D \rightarrow TM/D \rightarrow TM/D^{-2} \rightarrow 0.$$
From $D^{-2}/D|_C=\mathcal O(1)^3 \oplus \mathcal O \oplus \mathcal O(-1)$ and $TM/D|_C=\mathcal O(1)^3 \oplus \mathcal O^5$, it follows that $TM/D^{-2}|_C= \mathcal O(1) \oplus \mathcal O$.

Set
\begin{eqnarray*}
 \mathcal E^{-3}&:=&\Sym^2(\wedge^3\mathcal V) \otimes (\mathcal W \otimes \wedge^2 \mathcal W)
\end{eqnarray*}
Then $\mathcal E^{-3}|_C = \mathcal O(1) \oplus \mathcal O$.
The Frobenius  bracket $[\,,\,]:D \wedge   (D^{-2}/D) \rightarrow  TM /D^{-2}$ induces a map

$$\mathcal E^{-3}|_C \rightarrow  TM  /D^{-2}|_{C} $$
which is an isomorphism for generic $y \in C$ by Proposition \ref{fundamental graded Lie algebra of vmrt} (1). Hence it  is an isomorphism at any point of $C$ because  $TM/D^{-2}|_C =\mathcal O(1) \oplus \mathcal O$   has same splitting type as $\mathcal E^{-3}|_C$,

Consequently, $\Symb_y(D)$ is isomorphic to $\mathfrak m$ for any $y \in C$.

In sum, there exists a Zariski open subset $M'\subset M^0\subset M$ such that a general member  of $\mathcal K$ lies on $M^0$ and the varieties of minimal rational tangents $\mathcal C|_{M^0}\subset \mathbb P D|_{M^0}$ defines an ${\bf S}$-structure on $(M^0, D|_{M^0})$ of type $\frak m$.
Proposition \ref{S-str and G_0-str} together with Proposition \ref{fundamental graded Lie algebra of vmrt} implies that
there exist $G_0$-structure on $(M^0, D|_{M^0})$ corresponding to the ${\bf S}$-structure, where $G_0=G(\widehat{\bf S})$. This completes the proof of Proposition  \ref{extension to rationally connected open subset}.
\end{proof}

\begin{proof}[Proof of Lemma \ref{lem:splitting type of V}]
We adapt an argument in the proof of Proposition 7.3 of \cite{HL21}.

Recall that we have the following exact sequence:
\begin{eqnarray*}
0 \rightarrow \mathcal V|_C \rightarrow D |_C = \mathcal O(2) \oplus \mathcal O(1)^p \oplus \mathcal O^r \oplus \mathcal O(-1)^s \rightarrow \Sym^2 \mathcal V \otimes \mathcal W |_C \rightarrow 0.
\end{eqnarray*}

Let $v \in V$ and $w \in W$ be vectors  with $[T_xC] =[v + v^2 \otimes w]$. Then they define  vector subbundles $\mathcal V_0 \subset \mathcal V|_C$ and $\mathcal W_0 \subset \mathcal W|_C$. Set $\mathcal V_0^{(1)}:=\mathcal V|_C/\mathcal V_0$ and $\mathcal W_0^{(2)}:=\mathcal W|_C/\mathcal W_0$.
By Lemma \ref{lem:finer exact sequences}, we get the following exact sequences.
\begin{eqnarray*}
 0 \rightarrow &\mathcal O(2)& \rightarrow\mathcal Q^{(0)} \rightarrow 0 \\
0 \rightarrow  \mathcal V_0  \rightarrow &\mathcal O(1)^p& \rightarrow \mathcal Q^{(1)}  \rightarrow 0 \\
0 \rightarrow  \mathcal V_0^{(1)}   \rightarrow &\mathcal O^r& \rightarrow \mathcal Q^{(2)}  \rightarrow 0\\
 0  \rightarrow &\mathcal O(-1)^s& \rightarrow \mathcal Q^{(3)} \rightarrow 0
\end{eqnarray*}
with
\begin{eqnarray*}
  \mathcal Q^{(0)} &= &    \Sym^2 \mathcal V_0 \otimes \mathcal W_0  \\
\deg \mathcal Q^{(1)} &=& \deg \left(\mathcal V_0 \circ \mathcal V_0^{(1)}    \otimes \mathcal W_0\right)  \oplus \left( \Sym^2 \mathcal V_0 \otimes \mathcal W_0^{(1)} \right)  \\
\deg \mathcal Q^{(2)} &=& \deg \left(\Sym^2 \mathcal V_0^{(1)}  \otimes \mathcal W_0\right)  \oplus \left( \mathcal V_0 \circ  \mathcal V_0^{(1)}  \otimes \mathcal W _0^{(1)} \right)  \\
 \mathcal Q^{(3)} &= &     \Sym^2  \mathcal V _0^{(1)}  \otimes  \mathcal W_0^{(1)} .
\end{eqnarray*}
Furthermore,
\begin{enumerate}
\item[]
any direct summand of $\mathcal V_0^{(1)}$ has degree $\leq 0$, and
\item[]
any direct summand of $\Sym^2 \mathcal V_0^{(1)} \otimes \mathcal W_0^{(1)}$ has degree  $  -1$.
\end{enumerate}
The first statement follows from the third exact sequence, and the second statement follows from the last one.

When $X$ is $(B_m, \alpha_{m-1}, \alpha_m)$, write $$\mathcal V_0  =\mathcal O(a_1), \mathcal V_0^{(1)}=\mathcal O(a_2), \mathcal W_0 =\mathcal O(b_1), \text{ and } \mathcal W_0^{(1)}=\mathcal O(b_2) \oplus \dots \oplus \mathcal O(b_{m-1}).$$
Then we have
\begin{eqnarray*}
2&=& 2a_1 + b_1 \\
m&=& a_1 +(a_1 +a_2 +b_1) + \sum_{i=2}^{m-1}(2a_1+b_i) =(2m-2)a_1 + a_2 + \sum_{i=1}^{m-1}b_i \\
0&=&a_2 + (2a_2 + b_1) + \sum_{i=2}^{m-1}(a_1 +a_2 + b_i) = (m-2)a_1 + (m+1) a_2 + \sum_{i=1}^{m-1} b_i \\
-m+2 &=& \sum_{i=2}^{m-1}(2a_2 + b_i) = 2(m-2) a_2+ \sum_{i=2}^{m-1} b_i.
\end{eqnarray*}
Therefore, $a_1=1$ and $a_2=0$ and $b_1=0$ and $\sum_{i=2}^{m-1} b_i = -m+2$.
Since the degree of a direct summand of $\Sym^2 \mathcal V_0^{(1)} \otimes \mathcal W_0^{(1)}$ is $  -1$, we have $b_i=-1$ for all $2 \leq i \leq m-1$.

Since the exact sequences
 \begin{eqnarray*}
 0 \rightarrow \mathcal V_0 \rightarrow &\mathcal V|_C& \rightarrow \mathcal V_0^{(1)} \rightarrow 0\\
 0 \rightarrow \mathcal W_0 \rightarrow &\mathcal W|_C& \rightarrow \mathcal W_0^{(1)} \rightarrow 0
 \end{eqnarray*}
split, we have the desired decompositions of
$\mathcal V|_C$ and $\mathcal W|_C$.

\medskip
When $X$ is $(F_4, \alpha_{2}, \alpha_3)$, write
$$\mathcal V_0  =\mathcal O(a_1), \mathcal V_0^{(1)}=   \mathcal O(a_2)\oplus \mathcal O(a_3), \mathcal W_0 =\mathcal O(b_1), \text{ and } \mathcal W_0^{(1)}= \mathcal O(b_2).$$
 Then we have
 \begin{eqnarray*}
 2&=& 2a_1 + b_1 \\
 4&=& a_1 + (a_1 +a_2 +b_1) +(a_1+a_3 +b_1) +(2a_1+b_2) \\
 &=& 5a_1 +(a_2 +a_3) +2 b_1 +b_2 \\
 0&=& a_2+a_3 + (3(a_2 +a_3) + 3 b_1)) + (a_1 +a_2 + b_2) + (a_1 +a_3+ b_2)\\
 &=& 2a_1 + 5(a_2 +a_3) +3 b_1 +2 b_2  \\
 -3&=& 3(a_2+a_3) +3b_2.
 \end{eqnarray*}
Therefore, $a_1=1$ and $b_1=0$ and $a_2+a_3=0$ and $b_2=-1$.
Since the degree of a direct summand of $\mathcal V_0^{(1)}$ is $\leq 0$, we get $a_2=a_3=0$.
By the same reason as in the case of $(B_m, \alpha_{m-1}, \alpha_m)$,  we have the desired decompositions of
$\mathcal V|_C$ and $\mathcal W|_C$.

This completes the proof of Lemma \ref{lem:splitting type of V}.
\end{proof}

 By Proposition \ref{extension to rationally connected open subset}, there is a Zariski open subset $M^0\subset M$ such that a general member  of $\mathcal K$ lies on $M^0$, satisfying  the following properties:
\begin{enumerate}
\item $D|_{M^0}$ is of type $\frak m$;
\item $\mathcal C|_{M^0}\subset \mathbb P D|_{M^0}$ defines an ${\bf S}$-structure on $(M^0, D|_{M^0})$ and there corresponds to a $G_0$-structure $\mathscr P$ on $(M^0, D|_{M^0})$, where $G_0=G(\widehat{\bf S})$.
\end{enumerate}
Define a vector bundle $\mathcal H^2_k$ on $M^0$  by  $\mathcal H^2_k:=\mathscr P \times _{G_0} H^2(\frak m, \frak g)_k$.

\begin{lemma} \label{lem:section vanishes} 
  $H^0(M^0, \mathcal H^2_k)$ is zero for all $k \geq 1$.
\end{lemma}

\begin{proof}
By Proposition \ref{prop:H2 vanishing}, it suffices to show that the followings;
\begin{enumerate}
\item $H^0(M^0, \wedge ^2 D^* \otimes \mathcal V)=0$ when $(\frak m, \frak g_0)$ is of type $(B_m, \alpha_{m-1}, \alpha_m )$ for $m>3$ or of type $(F_4, \alpha_2, \alpha_3 )$;
\item $H^0(M^0, \wedge ^2 D^* \otimes D)=0$ when $(\frak m, \frak g_0)$ is of type $(B_3, \alpha_{2}, \alpha_3 )$;
\item  $H^0(M^0, (\wedge^2 D \otimes D)^* \otimes \mathcal V)=0$ when $(\frak m, \frak g_0)$ is of type $(B_m, \alpha_{m-1}, \alpha_m )$ for $m \geq 3$;
\end{enumerate}

We will adapt an argument similar to the proof of Proposition 6.2 of \cite{HH}. Assume that $(\frak m, \frak g_0)$ is either of type $(B_m, \alpha_{m-1}, \alpha_m )$ for $m > 3$ or of type $(F_4, \alpha_2, \alpha_3 )$. Let   $\varphi :\wedge^2 D \rightarrow \mathcal V$  be a nontrivial  vector bundle map.    For  $x \in M^0$ and $\beta \in D_x$ with $[\beta]\in \mathcal C_x$, take $C$ to be a member of $\mathcal K$ passing through $x$ with $[T_xC] =[\beta]$. Then $TC \wedge D|_C $  is decomposed as a sum of $\mathcal O(a)$'s with $a \geq 1$. Thus $\varphi|_C$ maps $TC \wedge D|_C $ into the $\mathcal O(1)$-factor of $\mathcal V|_C$, the intersection of $\mathcal V_x $ with $T_{\beta}\widehat{\mathcal C}_x$. Applying this argument to a general $[\beta] \in \mathcal C_x$, we see that the image of $\varphi_x$ is contained in the intersection
\begin{equation*}
    \cap_{[\beta] \in \mathcal C_x}(T_{ \beta }\widehat{\mathcal C}_x \cap \mathcal V_x),
\end{equation*}
which is $G_0$-invariant and degenerate in $\mathcal V_x$, contradicting to the fact that $\mathcal V$ is an irreducible $G_0$-bundle. Therefore, $\varphi_x$ is zero.

When $(\frak m, \frak g_0)$ is of type $(B_3, \alpha_{2}, \alpha_3 )$, Let $\varphi :\wedge^2 D \rightarrow \mathcal V$ and $\psi :\wedge^2 D \rightarrow \Sym^2 \mathcal V \otimes \mathcal W $  be nontrivial  vector bundle maps.  For  $x \in M^0$ and $\beta \in D_x$ with $[\beta]\in \mathcal C_x$, take $C$ to be a member of $\mathcal K$ passing through $x$ with $[T_xC] =[\beta]$. Then $TC \wedge D|_C $  is decomposed as a sum of $\mathcal O(a)$'s with $a \geq 1$. Thus $\varphi|_C$ maps $TC \wedge D|_C $ into the $\mathcal O(1)$-factor of $\mathcal V|_C$ and $\psi|_C$ maps $TC \wedge D|_C $ into the positive-factors of $\Sym^2 \mathcal V \otimes \mathcal W|_C$, the intersection of $\mathcal V_x $ with $T_{\beta}\widehat{\mathcal C}_x$ and the intersection of $\Sym^2 \mathcal V_x \otimes \mathcal W _x $ with $T_{\beta}\widehat{\mathcal C}_x$ . Applying this argument to a general $[\beta] \in \mathcal C_x$, we see that the image of $\varphi_x$ and $\psi_x$ are contained in the intersections
\begin{equation*}
    \cap_{[\beta] \in \mathcal C_x}(T_{ \beta }\widehat{\mathcal C}_x \cap \mathcal V_x) \text{ and }  \cap_{[\beta] \in \mathcal C_x} \{ T_{ \beta }\widehat{\mathcal C}_x \cap (\Sym^2 \mathcal V_x \otimes \mathcal W _x )\}
\end{equation*}
which are $G_0$-invariant and degenerate in $\mathcal V_x$ and $\Sym^2 \mathcal V_x \otimes \mathcal W _x$ respectively, contradicting to the fact that $\mathcal V$ and $\Sym^2 \mathcal V \otimes \mathcal W $ are irreducible $G_0$-bundles. Therefore, $\varphi_x$ and $\psi_x$ are zero. Hence, $H^0(M^0, \wedge ^2 D^* \otimes D)=0$.

Assume that $(\frak m, \frak g_0)$ is of type $(B_m, \alpha_{m-1}, \alpha_m )$ for $m \geq 3$. Let $\psi: \wedge ^2 D \otimes D \rightarrow \mathcal V$ be a vector bundle map.
Then, by the same argument as above, for  $x \in M^0$ and $[\beta] \in \mathcal C_x$ with $[\beta]=[T_xC]$, $\psi_x$ maps $\beta \wedge T^{(2)}_{\beta}\widehat{\mathcal C}_x \otimes D_x$ into the intersection $\mathcal V_x \cap T_{\beta}\widehat{\mathcal C}_x$.
 Thus the restriction of $\psi_x$ to $(\span \{\beta \wedge T_{\beta}^{(2)}\mathcal C_x : \beta \in \mathcal C_x\})\otimes D_x$ is zero.
By Proposition \ref{kernel of frobenius braket}, $\span \{\beta \wedge T_{\beta}^{(2)}\mathcal C_x: \beta \in \mathcal C_x\}$ is the kernel of the Frobenius bracket $[\,,\,]:\wedge ^2D \rightarrow TM/D$ at each point $x\in M^0$, and thus $\psi $ induces a bundle map
$$\mathcal E \otimes D \rightarrow \mathcal V.$$

On the other hand,
for  $x \in M^0$ and $[\beta] \in \mathcal C_x$ with $[\beta] = [T_xC]$, we have $\mathcal E|_C =\mathcal O(1)^{m-2} \oplus \mathcal O^t$. Therefore, by the induced map $\mathcal E \otimes D \rightarrow \mathcal V$, the image from $\mathcal E|_C \otimes TC$ to $\mathcal V|_C=\mathcal O(1) \oplus \mathcal O$ is zero. Since $[\beta] \in \mathcal C_x$ span $D_x$, the map $\psi_x$ is zero.

This completes the proof of Lemma \ref{lem:section vanishes}.
 \end{proof}

\begin{proof} [Proof of Theorem \ref{main results} in the case when $X$ is $(B_m, \alpha_{m-1}, \alpha_m)$ for $m \geq 3$ or $(F_4, \alpha_2, \alpha_3)$]
By Proposition \ref{prop:prolongation methods} and  Lemma \ref{lem:section vanishes}, $G_0$-structure $\mathscr P$  on $(M^0, D|_{M^0})$ is locally equivalent to the standard one. By Proposition \ref{S-str and G_0-str} together with  Proposition \ref{fundamental graded Lie algebra of vmrt}, the ${\bf S}$-structure on $M^0$ defined by $\mathcal C(M)|_{M^0}$ is locally equivalent to the standard one. By Theorem \ref{Cartan-Fubini},  a local map preserving the varieties of minimal rational tangents can be extended to a global biholomorphism. Hence, $M$ is biholomorphic to $X$.
This completes the proof of Theorem \ref{main results} in the case when $X$ is $(B_m, \alpha_{m-1}, \alpha_m)$ for $m \geq 3$ or $(F_4, \alpha_2, \alpha_3)$.
\end{proof}

\section{$(B_3,\alpha_1,\alpha_3)$ case} \label{sect:B3 case}

In this  section, we assume that $X$ is the horospherical variety $(B_3, \alpha_1, \alpha_3)$. We prove Theorem \ref{main results} in this case by the   method   we use for the horospherical varieties $(B_m, \alpha_{m-1}, \alpha_m)$ for $m \geq 3$ or $(F_4, \alpha_2, \alpha_3)$ in Section \ref{sect:horospherical BF type},  Section \ref{sect:H2 cohomology}, and Section \ref{sect:local equivalence}.

 Since the structure of the Lie algebra   of $\Aut(X)$ and the projective geometry of the variety of minimal rational tangents are less complicated, the computations are relatively shorter. \\

Let $V $ be the spin representation of $L_0=B_2$ and let $W$ be the standard representation of $B_2$.
By the isomorphism $B_2 \simeq C_2$, we may consider $V$ as the standard representation of $C_2$ with a nondegenerate skew symmetric bilinear form $\omega$ and $W$ as the  subspace   $\wedge^2_{\omega} V$ of $\wedge^2 V$ generated by $v \wedge u$ with  $\omega(v,u)=0$.

Set $${\bf U}:= V \oplus  W =  V \oplus \wedge^2_{\omega}V$$
Let ${{\bf S}}  \subset \mathbb P \bf U$ be the variety of minimal rational tangents of $X$ at the base point (Proposition \ref{vmrt of X}):
$${\bf S}=\mathbb P \{ v +  v\wedge u: v, u \in V, \ \omega(v,u)=0 \} \subset \mathbb P\bf U.$$

Let $\frak g= (\frak l + \mathbb C)  \triangleright U  $ be the Lie algebra of $\Aut(X)$ graded as in Proposition \ref{tangentsp of X} and $\frak m =\bigoplus_{p<0}\frak g_p$ be its negative part.

\subsection{Projective geometry of varieties of minimal rational tangents}
We list up properties of ${\bf S}$ as a projective subvariety of $\mathbb P{\bf U}$ as in Section \ref{sect:horospherical BF type}.

\begin{lemma}   [Lemma 1.17 of \cite{Pa} or Theorem 1.1 of \cite{Mi}]
The variety ${\bf S}$ is the odd symplectic Grassmannian $\Gr_{\omega}(2,5)$ and the automorphism group $\Aut^0(\widehat{\bf S})$ of the cone $\widehat{\bf S}$ is  $\PSp(5)=(\Sp(4)\times \mathbb C ^*)/\{\pm 1\} \ltimes \mathbb C ^4$.
\end{lemma}


The tangent space $T_{\beta} \widehat{\bf S}$ at $\beta = v + v \wedge u \in \widehat{\bf S} $ is given by
$$T_{\beta}\widehat{{\bf S}} = \{ v' +  v' \wedge u +  v \wedge u' \in {\bf U} : v', u' \in V \}.$$
The second fundamental form $II_{\beta} : \Sym^2 T_{\beta}\widehat{{\bf S}} \rightarrow {\bf U} /T_{\beta}\widehat{{\bf S}}$ is
\begin{eqnarray*}
&& II_{\beta}(v' +  v' \wedge u, v'' +  v'' \wedge u) =0 \\
&& II_{\beta}(v' +  v' \wedge u, v \wedge u') =  v' \wedge u' \\
&& II_{\beta}( v \wedge u',  v \wedge u'')=0,
\end{eqnarray*}
where $v', v'', u', u'' \in V$. Thus the second osculating space $T_{\beta}^{(2)}\widehat{{\bf S}}$ is $W+T_{\beta} \widehat{\bf S}=\bf U$.
Therefore, the third fundamental form $III_{\beta} : \Sym^3 T_{\beta}\widehat{{\bf S}} \rightarrow {\bf U} /T_{\beta}^{(2)}\widehat{{\bf S}}$ is zero. \\

Let $G_0$ be the subgroup of $G=\Aut(X)$ with Lie algebra $\frak g_0$. From the exact sequence of $G_0$-modules
$$0 \rightarrow V \rightarrow {\bf U} \rightarrow W \rightarrow 0$$
we get the following exact sequences.

\begin{lemma} \label{lem: B_3 exact sequences} Let $\beta=v_1 + v_1\wedge v_2$ be an element of $\widehat{\bf S}$, where $v_1, v_2 \in V$ such that $v_1 \wedge v_2 \neq 0$. For $i=1,\cdots 4$, denote by $V_i$ the subspace of $V$ generated by $v_i$ and $\omega(v_i,v_j)=\delta_{i+2,j}$ for $i=1,2$ $1\leq j \leq 4$. Then
we have the following exact sequences.
\begin{eqnarray*}
0 \rightarrow &\mathbb C \beta& \rightarrow  V_1 \wedge V_2 \rightarrow 0 \\
0 \rightarrow V_1 \oplus V_2 \rightarrow &T_{\beta}{\widehat{\bf S}}/\mathbb C \beta& \rightarrow \{ (V/V_1) \wedge V_2 +  V_1 \wedge (V/V_2) \}_{\omega} \rightarrow 0 \\
0 \rightarrow V/(V_1 \oplus V_2) \rightarrow &{\bf U}/T_{\beta}\widehat{\bf S}& \rightarrow V_3 \wedge V_4 \rightarrow 0
\end{eqnarray*}
\end{lemma}

\begin{proof} We get the following exact sequences.
\begin{eqnarray*}
0 \rightarrow &\mathbb C \beta& \rightarrow  V_1 \wedge V_2 \rightarrow 0 \\
0 \rightarrow V_1 \oplus V_2 \rightarrow &T_{\beta}{\widehat{\bf S}}& \rightarrow \{ V \wedge V_2 +  V_1 \wedge V \}_{\omega} \rightarrow 0
\end{eqnarray*}
Denote that $(V \wedge V_2)\cap (V_1 \wedge V)=V_1 \wedge V_2$.
By taking the quotients we get the desired exact sequences.
\end{proof}

\begin{lemma}\label{lem: B_3 Lie alg} \
\begin{enumerate}
\item $\dim \frak g_{-1} = 9$ and $\dim \frak g_{-k} = 0$ for $k>1$
\item $\frak g_0=\aut^0(\widehat{\bf S})$
\end{enumerate}
\end{lemma}

\begin{proof}
Since $\frak m =\frak g_{-1}$ isomorphic to ${\bf U}$ as vector space, (1) follows. Since $\Aut^0(\widehat{{\bf S}})$ is equal to the linear automorphism group $G(\widehat{\bf S})$    and the induced map $\frak g(\widehat{\bf S}) \rightarrow \frak g_0(\frak m)$ is injective whose image   in $\mathfrak g_0(\frak m)$ agrees with $\mathfrak g_0   \subset \mathfrak g_0(\mathfrak m)$.
\end{proof}

\begin{proposition} \label{prop: B_3 parallel transprt vmrt}
Let ${\bf S} \subset \mathbb P {\bf U} $ be the variety of minimal rational tangents of
$(B_3,\alpha_1,\alpha_3)$ at the base point.
Let $\pi: \mathbb P \mathcal U\rightarrow \mathbb P^1$ be the projectivization of a holomorphic vector bundle $\mathcal U$ over  $  \mathbb P^1$  and let $\mathcal C \subset \mathbb P \mathcal U$ be an irreducible subvariety. Denote by $\varpi$ the restriction of $\pi$ to $\mathcal C$. Assume that
   \begin{enumerate}
   \item $\mathcal C_t:=\varpi^{-1}(t) \subset \mathbb P  \mathcal U_t:=\pi^{-1}(t)$ is projectively equivalent to ${\bf S} \subset \mathbb P {\bf U}$ for all $t \in \mathbb P^1 -\{t_1, \dots, t_k\}$; 
  \item
      for  a general  section $\sigma \subset \mathcal C$ of $\varpi$, the relative second fundamental forms of $\mathcal C$ along $\sigma$ are constants.
     \end{enumerate}
Then for any $t \in \mathbb P^1$, $\mathcal C_t \subset \mathbb P(\mathcal U_t)$ is projectively equivalent  to ${\bf S} \subset \mathbb P({\bf U})$.
\end{proposition}

\begin{proof}
The Picard number of the odd Lagrangian Grassmannian $\Gr_{\omega}(n,2n+1)$ is one and ${\bf S} \subset \mathbb P({\bf U})$ is the minimal embedding by the line bundle $\mathcal O(1)$. The deformation rigidity of the odd Lagrangian Grassmannian $\Gr_{\omega}(n,2n+1)$ is also known in Theorem 1.2 of \cite{Park} and hence, for any $t \in \mathbb P^1$, $\mathcal C_t$ is biholomorphic to ${\bf S}$.
 The constancy of the second fundamental form implies that  $\mathcal C_t \subset T_{\beta}^{(2)}\widehat{\mathcal C_t}=\mathcal U_t$ is non-degenerate. Hence, $\mathcal C_t \subset \mathbb P(\mathcal U_t)$ is projectively equivalent to the minimal embedding ${\bf S} \subset \mathbb P({\bf U})$ for all $t$.
\end{proof}

\subsection{$H^2$-cohomology}
 Now we compute Lie algebra cohomologies as in Section \ref{sect:H2 cohomology}.

\begin{lemma}\label{lem: B_3 H^1 and H^2}  For $X=(B_3,\alpha_1,\alpha_3)$, we have the followings:
\begin{enumerate}
\item[(i)] $H^1(\frak l_-, \frak l)_{k-1}$ vanishes except for $k=1$ and $$H^1(\frak l_-, \frak l)_{0} \subset \Hom(\frak l_{-1}, \frak l_{-1}).$$
\item[(ii)] $H^{ 2}(\frak l_-, U)_k$  vanishes except for $k=1$, and we have
\begin{eqnarray*}
    H^2(\frak l_-,  U)_1 &\subset& \wedge^2 \frak l_{-1}^* \otimes U_{-1}.
\end{eqnarray*}
\item[(iii)] except (i) and (ii), Lemma \ref{lem:H1 vanishing semisimple} and Lemma \ref{lem:H2 vanishing semisimple} are satisfies.
    \end{enumerate}

\end{lemma}
\begin{proof} Apply the Kostant theory to get the desired result. \end{proof}

\begin{lemma} \label{lem: cohomology lemma 2 B_3}
 For $A \in \Hom(\frak l_-, U)_1$, if the image of $\partial_0A: \wedge^2 \frak l_{-1} \rightarrow U_{-1}$ has dimension $\leq 1$, then we have $\partial_0 A=0$.
\end{lemma}

\begin{proof}
We recall the notions : let $\{x_{-\alpha}\}$ be a basis of $\frak l_{-1}$ consisting of root vectors and let $\{u_{\mu}\}$ ($\{u_{\lambda}\}$, respectively) be a basis of $U_{-1}$ ($U_0$, respectively) consisting of weight vectors. We may assume that $[x_{-\alpha}, u_{\lambda}] = u_{-\alpha+\lambda}$ if $-\alpha + \lambda$ is a weight.\\
For $A \in \Hom(\frak l_{-}, U)_1$, we have
 \begin{eqnarray*}
 A(x_{-\alpha}) &=& \sum_{\lambda} A_{\lambda, \alpha} u_{\lambda}.
\end{eqnarray*}

If $(\frak m, \frak g_0)$ is of type $(B_3, \alpha_{1}, \alpha_3)$, the action $ \frak l_{-1} \times U_{1} \rightarrow U_0$ (equivalently, $ \frak l_1 \times U_{-1} \rightarrow U_0$) is given by
\begin{center}
    \begin{tabular}{ r | c c c c c }
     $\times$ & $w_0$ & $w_1$ &  $w_2$ &  $w_3$ &  $w_4$ \\
 \hline

     $v^*_1$ &   $v_3$ &  $v_2$ &  $\cdot$ &  $\cdot$ & $-v_4$ \\

     $v^*_2$ &   $-v_4$ &  $-v_1$ &  $v_3$ &  $\cdot$ & $\cdot$ \\

     $v^*_3$ &   $-v_1$ &  $\cdot$ &  $-v_2$ &  $v_4$ & $\cdot$ \\

     $v^*_4$ &   $v_2$ &  $\cdot$ &  $\cdot$ &  $-v_3$ & $v_1$ \\

    \end{tabular}
\end{center}
where $\{v_i\}$, $1\leq i \leq 4$, is a basis of $V$ with the skew-symmetric 2-form $\omega(v_i,v_j)=\delta_{i+2,j}$ for $i=1,2$ and $1\leq j \leq 4$ and $\{w_i\}$ is a basis of $W$ given by $w_0=v_1\wedge v_3 - v_2 \wedge v_4$, $w_1=v_1\wedge v_2$, $w_2=v_2\wedge v_3$, $w_3=v_3\wedge v_4$ and $w_4=v_4\wedge v_1$.

Write $v_i$ as $u_{\mu_i}$. Then
$(\partial_0A)(x_{-\alpha}, x_{-\beta})$ is given by  $$\sum_{\mu_i}  \left( A_{\mu_i+ \alpha, \beta} - A_{\mu_i+\beta, \alpha} \right) u_{\mu_i}. $$

We can take $x_{-\alpha}$ and $x_{-\beta}$ in $\{w_0, w_1, w_2, w_3, w_4\}$ such that $x_{-\alpha} \neq x_{-\beta}$ and the coefficient of $u_{\mu_i}$ is zero.  Since  $(\partial_0A)(x_{-\alpha}, x_{-\beta})$ is parallel to each other for any choice of a pair $(x_{-\alpha},x_{-\beta})$, the coefficients are zero for any $\alpha$ and $\beta$.
\end{proof}

\begin{proposition} \label{H2 cohomology B3}
 $H^{ 2}(\frak m, \frak g)_k$ vanishes except for $k=1$, and
\begin{eqnarray*}
H^2(\frak m, \frak g)_1 &\subset& \wedge^2 \frak g_{-1}^* \otimes \frak g_{-1}. \end{eqnarray*}
\end{proposition}

\begin{proof} By Lemma \ref{lem: B_3 H^1 and H^2} and Lemma \ref{lem: cohomology lemma 2 B_3}, the same argument as in the proof of Proposition \ref{prop: computation of H2 cohomology} apply to the case when $X$ is $(B_3, \alpha_1, \alpha_3)$ to get the desired result.
\end{proof}

\subsection{Local equivalence of geometric structures} \label{sect:local equivalence B3}
We complete the proof of Theorem \ref{main results}  as in Section \ref{sect:local equivalence}.

 Let $M' \subset M $ be a  Zariski open subset such that for all $x \in M'$ satisfying $\mathcal C_x \subset \mathbb P TM_x$ is projectively equivalent to $ {\bf S} \subset \mathbb P {\frak g_{-1}}$.

Let $C $ be a standard minimal rational curve represented by $f:\mathbb P^1 \rightarrow M$ with $C \cap M' \not = \emptyset$ and $C \not\subset bad(\mathcal K)$. Then for a generic point $y \in C$, $\mathcal C_y \subset \mathbb PD_y$ is projectively equivalent to $\bf S\subset \mathbb P \frak g_{-1}$. By Proposition \ref{parallel transport of II}, the relative second fundamental form of $\mathcal C(M)$ along the lifting $C^{\sharp}$ of $C$ is constant, and we have
\begin{eqnarray*}
f^*TM &=& \mathcal O(2) \oplus \mathcal O(1)^5 \oplus \mathcal O^3.
\end{eqnarray*}

Furthermore,
the pull-back $(f^{\sharp})^*\widehat{T}^{\varpi} $ of the relative affine tangent bundle $ \widehat{T}^{\varpi} $ of $\mathcal C \subset \mathbb P(TM)$ is the positive part $ P:=\mathcal O(2) \oplus \mathcal O(1)^5$ of $f^*TM$, and
   the pull-back $(f^{\sharp})^*\widehat{T}^{(2), \varpi} $ of the relative  second osculating affine bundle   $ \widehat{T}^{(2), \varpi} $ of $\mathcal C \subset \mathbb P(TM)$ is the subbundle  $P^{(2)} :=\mathcal O(2) \oplus \mathcal O(1)^5 \oplus \mathcal O^3$ of $f^*TM$.

 By Proposition \ref{prop: B_3 parallel transprt vmrt}, for any $y\in C$, $\mathcal C_y \subset \mathbb PTM_y$ is projectively equivalent to $\bf S\subset \mathbb P \frak g_{-1}$.

\begin{lemma}\label{lem:splitting type of V B_3}
Let $\mathcal V$ and $\mathcal W$ be vector bundles on $M'$ associated with $V$ and $W$.
Let $C$ be a general member of $\mathcal K$ passing through $x \in M'$.
When $X$ is $(B_3, \alpha_{1}, \alpha_3)$, we have
$$\mathcal V|_C= \mathcal O(1)^2 \oplus \mathcal O^2 \text{ and } \mathcal W|_C= \mathcal O(2) \oplus \mathcal O(1)^3 \oplus \mathcal O.  $$
\end{lemma}

\begin{proof}
From Lemma \ref{lem: B_3 exact sequences}, we have the exact sequences

\begin{eqnarray*}
0 \rightarrow &\mathcal O(2) & \rightarrow \mathcal Q^{(0)} \rightarrow 0 \\
0 \rightarrow \mathcal F_0 \rightarrow &\mathcal O(1)^5 & \rightarrow \mathcal Q^{(1)} \rightarrow 0 \\
 0 \rightarrow \mathcal V/\mathcal F_0 \rightarrow &\mathcal O^3& \rightarrow \mathcal Q^{(2)} \rightarrow 0
\end{eqnarray*}
with
\begin{eqnarray*}
\mathcal Q^{(0)} &=& \wedge^2 \mathcal F_0 \\
\deg \mathcal Q^{(1)} &=& \deg(\mathcal F_0 \wedge_{\omega}\mathcal V/\mathcal F_0)\\
\deg \mathcal Q^{(2)} &=& \deg \wedge^2(\mathcal V/\mathcal F_0).
\end{eqnarray*}
Write $\mathcal F_0 =\mathcal O(a_1) \oplus \mathcal O(a_2)$, $\mathcal V/\mathcal F_0=\mathcal O(a_3) \oplus \mathcal O(a_4)$, $\mathcal Q^{(0)}=\mathcal O(b_1)$, $\mathcal Q^{(1)} = \mathcal O(b_2) \oplus \mathcal O(b_3) \oplus \mathcal O(b_4)$ and $\mathcal Q^{(2)} = \mathcal O(b_5)$.
Then $2=b_1=a_1 + a_2  $, $5=(a_1 + a_2) + b_2 +b_3 +b_4$ and $0 = (a_3+a_4) + b_5$ with $b_5=a_3+a_4$. Thus $b_5=0$.

From the second exact sequence,
$a_i \leq 1$ for $i=1,2$ and $b_j \geq 0$ for $j =2,3,4$. Thus the second exact sequence splits and $a_1 = a_2 = b_2 = b_3 = b_4 =1$. From the third exact sequence we get $a_3 =a_4 =0$.
\end{proof}

In sum, there exists a Zariski open subset $M'\subset M^0\subset M$ such that a general member  of $\mathcal K$ lies on $M^0$ and the varieties of minimal rational tangents $\mathcal C|_{M^0}\subset \mathbb P TM|_{M^0}$ defines an ${\bf S}$-structure on $ M^0 $.
Proposition \ref{S-str and G_0-str} and Lemma \ref{lem: B_3 Lie alg} implies that
there exist $G_0$-structure on $(M^0, D|_{M^0})$ corresponding to the ${\bf S}$-structure, where $G_0=G(\widehat{\bf S})$.

Define a vector bundle $\mathcal H^2_k$ on $M^0$  by  $\mathcal H^2_k:=\mathscr P \times _{G_0} H^2(\frak m, \frak g)_k$.

\begin{lemma} \label{lem: B_3 section vanishes}
  $H^0(M^0, \mathcal H^2_k)$ is zero for all $k \geq 1$.
\end{lemma}

\begin{proof} By proposition \ref{H2 cohomology B3} it suffices to show the vanishing $H^0(M^0, \wedge ^2 D^* \otimes D)=0$. 
Let $\varphi :\wedge^2 D \rightarrow \mathcal V$ and $\psi :\wedge^2 D \rightarrow \mathcal W$  be nontrivial  vector bundle maps.  For  $x \in M^0$ and $\beta \in D_x$ with $[\beta]\in \mathcal C_x$, take $C$ to be a member of $\mathcal K$ passing through $x$ with $[T_xC] =[\beta]$. Then $TC \wedge D|_C $  is decomposed as a sum of $\mathcal O(a)$'s with $a \geq 1$. Thus $\varphi|_C$ maps $TC \wedge D|_C $ into the $\mathcal O(1)$-factor of $\mathcal V|_C$ and $\psi|_C$ maps $TC \wedge D|_C $ into the positive-factors of $\mathcal W|_C$, the intersection of $\mathcal V_x $ with $T_{\beta}\widehat{\mathcal C}_x$ and the intersection of $\mathcal W_x $ with $T_{\beta}\widehat{\mathcal C}_x$ . Applying this argument to a general $[\beta] \in \mathcal C_x$, we see that the image of $\varphi_x$ and $\psi_x$ are contained in the intersections
\begin{equation*}
    \cap_{[\beta] \in \mathcal C_x}(T_{ \beta }\widehat{\mathcal C}_x \cap \mathcal V_x) \text{ and }  \cap_{[\beta] \in \mathcal C_x}(T_{ \beta }\widehat{\mathcal C}_x \cap \mathcal W_x)
\end{equation*}
which are $G_0$-invariant and degenerate in $\mathcal V_x$ and $\mathcal W_x$ respectively, contradicting to the fact that $\mathcal V$ and $\mathcal W$ are irreducible $G_0$-bundles. Therefore, $\varphi_x$ and $\psi_x$ are zero. Hence, $H^0(M^0, \wedge ^2 D^* \otimes D)=0$.
\end{proof}

\begin{proof}[Proof of Theorem \ref{main results} in the case when $X$ is $(B_3, \alpha_1,\alpha_3)$]
By Proposition \ref{prop:prolongation methods} and  Lemma \ref{lem: B_3 section vanishes}, $G_0$-structure $\mathscr P$  on $ M^0 $ is locally equivalent to the standard one. By Proposition \ref{S-str and G_0-str}, the ${\bf S}$-structure on $M^0$ defined by $\mathcal C(M)|_{M^0}$ is locally equivalent to the standard one. By Theorem \ref{Cartan-Fubini},  a local map preserving the varieties of minimal rational tangents can be extended to a global biholomorphism. Hence, $M$ is biholomorphic to $X$.
This completes the proof of Theorem \ref{main results} in the case when $X$ is $(B_3, \alpha_1,\alpha_3)$.
\end{proof}

\end{document}